\documentclass[11pt]{amsart}
\usepackage[leqno]{amsmath}
\usepackage{amssymb}
\usepackage{enumerate}
\usepackage{xcolor}
\usepackage{subfigure}
\usepackage{graphicx}
\usepackage{hyperref}
\usepackage{microtype}
\usepackage[all]{xy}
\numberwithin{equation}{section}

\newtheorem{theorem}{Theorem}

\newtheorem{lemma}[theorem]{Lemma}

\newtheorem{corollary}[theorem]{Corollary}

\newcommand{\D}{\mathbf{D}}
\newcommand{\E}{\mathbf{E}}

\newcommand{\Z}{\mathbf{Z}}
\renewcommand{\P}{\mathbf{P}}

\newcommand{\R}{\mathbf{R}}

\newcommand{\HH}{\mathbf{H}}

\newcommand{\eps}{\varepsilon}
\newcommand{\SLE}{{\rm SLE}}
\newcommand{\CLE}{{\rm CLE}}

\newcommand{\one}{{\bf 1}}

\begin{document}

\title{Connection probabilities for conformal loop ensembles}

\author{Jason Miller}
\address{Statslab, Center for Mathematical Sciences, University of Cambridge, Wilberforce Road, Cambridge CB3 0WB, UK}
\email {jpmiller@statslab.cam.ac.uk}

\author{Wendelin Werner}
\address{Department of Mathematics, ETH Z\"urich, R\"amistr. 101, 8092 Z\"urich, Switzerland} 
\email{wendelin.werner@math.ethz.ch}

\begin{abstract}
The goal of the present paper is to explain, based on properties of the conformal loop ensembles $\CLE_\kappa$ (both with simple and non-simple loops, i.e., for the whole range $\kappa \in (8/3, 8)$) how to derive the connection probabilities 
in domains with four marked boundary points for a conditioned version of $\CLE_\kappa$ which can be interpreted as a $\CLE_{{\kappa}}$
with wired/free/wired/free {boundary conditions} on four boundary arcs (the wired parts being viewed as portions of to-be-completed loops). 
In particular, in the case of a square, we prove that the probability that the two wired sides of the square hook up so that they create one single loop 
is equal to $1/(1 - 2 \cos (4 \pi / \kappa ))$. 

Comparing this with the corresponding connection probabilities for discrete O($N$) models for instance indicates
that if a dilute O($N$) model (respectively a critical FK($q$)-percolation model on the square lattice)
has a non-trivial conformally invariant scaling limit, then necessarily  this scaling limit is $\CLE_\kappa$ where $\kappa$ is the value in $(8/3, 4]$ such that 
$-2 \cos (4 \pi / \kappa )$ is equal to  $N$ (resp.\ the value in $[4,8)$ such that $-2 \cos (4\pi / \kappa)$ is equal to $\sqrt {q}$). 

Our arguments and computations build on the one hand on Dub\'edat's SLE commutation relations (as developed and used by Dub\'edat, Zhan or Bauer-Bernard-Kyt\"ol\"a) and on the other hand, on the construction and properties of the conformal loop ensembles and their relation to Brownian loop-soups, restriction measures, and the Gaussian free field (as recently derived in works with Sheffield and with Qian).
\end{abstract}

\maketitle

\setcounter{tocdepth}{1}
\tableofcontents

\label{sec:intro}

\section {Introduction}

\subsection {A motivation from discrete models}

The fact that {\em if} critical bond percolation on the square lattice possesses a conformally invariant scaling limit (recall that establishing 
this is still an open question), {\em then} the only possible candidate for the scaling limit of the  interfaces is the Schramm-Loewner evolution ($\SLE_\kappa$)
with parameter $\kappa =6$ was first pointed out by Oded Schramm \cite{Schramm0,Schramm} in 1999 using the following simple argument: 
First, the (by now classical) argument based on the conformal Markov property shows that this
scaling limit would have to be an $\SLE_\kappa$ path {for some $\kappa > 0$}. It then remains to pin down the actual value of $\kappa$:  
To do so, consider an $\SLE_\kappa$ in a square from the bottom left corner to {the} top left corner, and note that one can compute 
the probability that it hits the right-hand side of the square along the way.  {It} turns out that $\kappa = 6$ is the only value for which this probability is equal to $1/2$.
On the other hand, for discrete critical bond  percolation on  a (quasi)-square $[0, n+1] \times [ 0, n]$ portion of the square grid,
duality shows that the probability that there exists a left-to-right open crossing (which means that the discrete percolation interface analog of the previous $\SLE_\kappa$ 
hits the right-hand side of the square) is equal to $1/2$ for all $N \ge 1$, and one can therefore deduce that this hitting probability should still be $1/2$ in the scaling limit.  
One can then conclude that $\SLE_6$ is the only possible candidate for a conformally invariant scaling limit of the percolation interface. 

Note that $\SLE_6$ also possesses other properties (for instance the locality property derived in \cite{LSW1}) that also imply that it is the only possible conformally invariant 
scaling limit for critical percolation interfaces without referring to any discrete crossing probability, but the above argument is {already} short, direct and convincing.
 {Recall that it is known by the work of Smirnov  \cite{SmirnovPercolation} that critical site percolation on the triangular lattice is 
indeed conformally invariant in the scaling limit and that discrete percolation interfaces do converge to $\SLE_6$ 
(see for instance \cite{Wlnperco} for a survey of the actual non-trivial proof of the convergence of the interfaces once one controls crossing probabilities). Recall also that 
describing the scaling limit of critical percolation on other planar lattices is still an open problem.}  

Two of the most natural and classical classes of discrete models that are supposed to give rise to SLE curves in the scaling limit are the O($N$) models (both the dense and the dilute versions) 
and the critical bond FK($q$)-percolation models (we will briefly recall the definition of these models in Appendix~\ref{App1}). 
{On the square lattice,} exactly three of these models have been proved to indeed converge to an SLE-based scaling limit: The Ising model (which is the O($1$) model, where {the} interfaces 
converge to $\SLE_3$ paths), the FK-percolation models for $q=2$ (this is the FK model related to the Ising model,
where {the} interfaces converge to $\SLE_{16/3}$ paths) -- see \cite{CDHKS,KempaSmirnov} and the references therein, or \cite{SmirnovICM}
for a survey, and the FK($q$)-percolation model in the limit $q=0^+$ (this is the uniform spanning tree model, 
and its boundary Peano curve is shown to converge to $\SLE_8$) -- see \cite{LSWLERW}. See also \cite{SmirnovPercolation} (and \cite{CN,Wlnperco})
for site percolation on the triangular lattice that can be viewed as an O($0$) model on the hexagonal grid that gives rise to $\SLE_6$ paths. 

It has been conjectured (see for instance \cite{kn,RS,Dup}), based on the identification between exponents, probabilities or dimensions
that one can rigorously compute for SLE processes on the one hand and the corresponding quantities that had been previously predicted 
using methods from theoretical physics (conformal field theory, quantum gravity or Coulomb gas methods, see for instance \cite{N,CardyLN})
for the asymptotic behavior of the discrete models on the other hand, 
that the O($N$) models have a non-trivial and conformally invariant scaling limit 
for all $N \in (0,2]$ and that this scaling limit should be related to $\SLE_\kappa$ curves for $N = - 2 \cos (4 \pi / \kappa )$, where $\kappa \in (8/3, 4]$ if one considers the dilute O($N$) model and where $\kappa \in (4,8)$ if one considers the dense O($N$) model. 
Similarly, the scaling limit of the critical FK($q$)-percolation model interfaces should be non-trivial for $q \in (0,4]$ and described by $\SLE_\kappa$ curves, for $\sqrt {q} = - 2 \cos (4 \pi / \kappa)$, where $\kappa \in [4, 8)$ (mind of course that $\cos (4 \pi / \kappa )$ is negative for all $\kappa \in (8/3, 8)$).  

\subsection{Content of the present paper}
Recall that a $\CLE_\kappa$ (conformal loop ensemble) for $\kappa \in (8/3, 8)$ is a particular 
random collection of loops in a simply connected planar domain such that the loops in a $\CLE_\kappa$ are $\SLE_\kappa$ type loops. While the $\SLE_\kappa$ curves can be argued (via the conformal Markov property) to be the only possible conformally invariant 
scaling limit of single interfaces for a wide family of discrete interfaces for lattice models with well-chosen boundary conditions (that involve choosing two special 
points on the boundary of the domain, and choosing one boundary arc to be ``wired'' while the other one is ``free''), the $\CLE_\kappa$ can be argued to be the only possible conformally invariant scaling limit for the joint law of {\em all} {of the} macroscopic interfaces for the same models with ``uniform'' free boundary conditions. 
For the aforementioned lattice models (FK($q$)-percolation and O($N$)), this convergence to CLE has now been proved for the very same cases as for individual interfaces, see \cite{LSWLERW,CN,ks2015fkcle2,bh}.

The results of the present paper on $\CLE_\kappa$ connection probabilities provide a generalization of Schramm's original argument for percolation that we outlined above to all {of} these models. More specifically, for all $\kappa \in (8/3, 8)$, we will first explain how to define the law of a $\CLE_\kappa$ in a conformal rectangle, with ``wired'' boundary conditions on two opposite sides of the rectangle.  The idea is to start with the usual $\CLE_\kappa$ and to partially discover it starting from two boundary points; in other words, our $\CLE_\kappa$ with free/wired/free/wired boundary conditions  will be defined as a conditioned version of the usual $\CLE_\kappa$. 
The wired portions of the boundary should be thought of as the discovered parts of partially discovered loops of an usual $\CLE_\kappa$ (see Figure~\ref{pic}). 
\begin{figure}[ht!]
  \includegraphics[width=0.4\textwidth, angle=180]{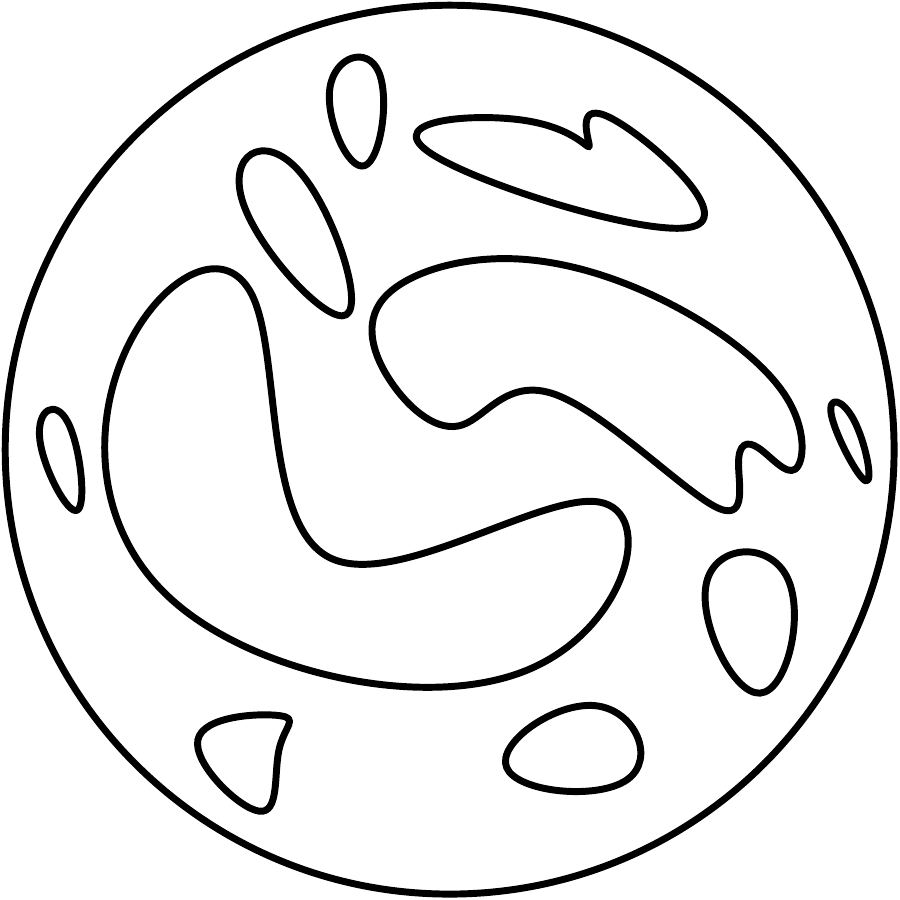}
  \quad
  \includegraphics[width=0.4\textwidth, angle=180]{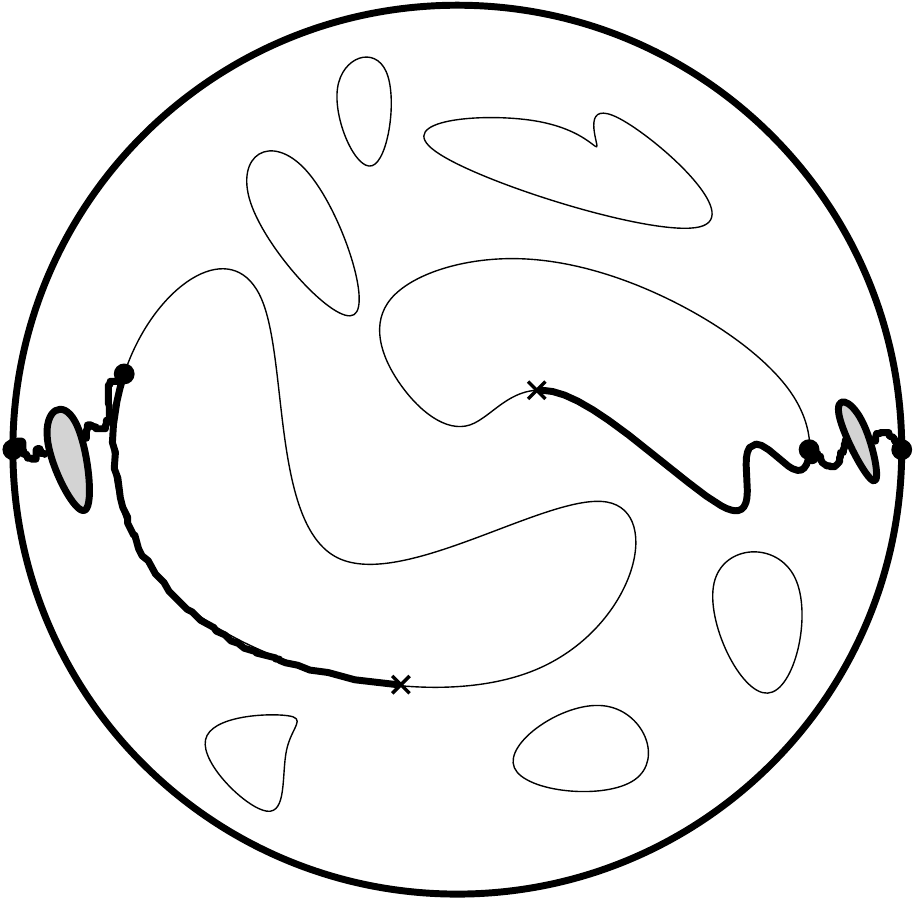}
  \caption{Discovering a simple CLE starting from two boundary points and creating a CLE with two wired boundary arcs (sketch). In this example, the two wired boundary arcs are not part of the same loops.}
  \label{pic}
\end{figure}

Then, for these $\CLE_\kappa$ with two wired boundary arcs, it turns out that the probability that the
 two wired pieces are part of the same loop  (see Figure \ref{pic2}) is a conformally invariant ($\kappa$-dependent) function of the domain with its two boundary arcs. 
It therefore suffices to determine it for $\CLE_\kappa$ in rectangles $[0,L]\times [0,1]$ with wired boundary conditions on the two vertical sides. 
As we shall explain below, Dub\'edat's commutation relations (see \cite{Dubedat, Dubedat2, Dubedat3, BBK, Zhan1, Zhan2, Zhan3}) 
show that for some explicit function $Y_\kappa (L)$ (see Section~\ref{S2}), the   
probability that the two vertical sides of the rectangle are part of the same $\CLE_\kappa$ loop is of the form 
$$ \frac { Y_\kappa (L)} { Y_\kappa (L) + \theta_\kappa Y_\kappa (1/L)}$$
for some unknown value $\theta_\kappa$. 
The goal of the present paper will be to determine the value $\theta_\kappa$ for the entire range $\kappa \in (8/3, 8)$: 

\begin{figure}[ht!]
  \includegraphics[scale=0.7]{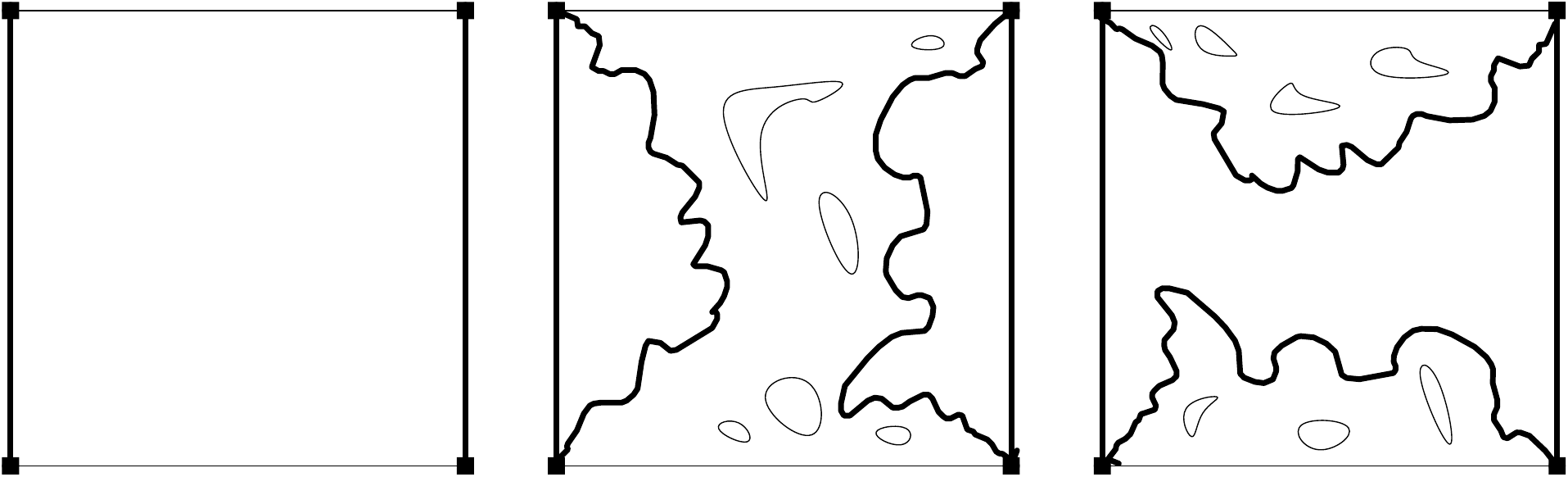}
  \caption{A square with wired vertical boundaries, and a sketch of the two possible connection configurations, together with the remaining outermost CLE loops}
  \label{pic2}
\end{figure}
 
\begin{theorem}
\label{mainprop}
The constant $\theta_\kappa$ is equal to $- 2 \cos (4 \pi / \kappa )$. In other words, when one considers a conformal square (i.e., $L=1$), the probability that 
the boundary arcs hook up into one single loop is equal to $1 / (1 -2 \cos (4\pi / \kappa) )$. 
\end{theorem} 

While in the present introduction, we put some emphasis on the connection probabilities (we will also sometimes use equivalently the term hook-up probability) of conformal squares, we should emphasize that we will actually derive Theorem~\ref{mainprop} by estimating 
the asymptotics of the connection probabilities of very long rectangles (when $L \to \infty$). 

One can note that the connection probability in {a} square first decreases from $1$ to $1/3$ when $\kappa$ {increases} from $8/3$ to $4$ (this is the regime where the $\CLE_\kappa$ consists of simple disjoint loops), 
and then increases again from $1/3$ to~$1$, when $\kappa$ {increases} from $4$ to $8$ (which is the regime of non-simple loops). Hence, the hook-up probability in a square with the alternating boundary conditions is always at least $1/3$. 

Deriving Theorem~\ref{mainprop} is the core of our paper and the proofs will make no reference to discrete models. 
The result for conformal squares has nevertheless some implications about the 
conjectural relation between conformal loop ensembles and discrete models that we now briefly discuss: 
One can compare it with the fact that for  discrete O($N$) models (on any lattice with some symmetries) and for the critical FK($q$) percolation model on $\Z^2$, 
the corresponding discrete connection probabilities in {a} square (or some other given symmetric shape if one {considers a lattice other than $\Z^2$}) are equal  
to $1/(1+N)$ and $1/(1+ \sqrt {q})$, independently of the size of the square, just because of a 
 symmetry/duality argument (we will recall this in Appendix~\ref{App1}). 
 This  can be viewed as the natural generalization to those models of the crossing probability of squares feature of 
critical percolation (which is the special case $q=1$, and for percolation, boundary conditions do not matter). 
Hence, one gets the following conditional results, that generalize Schramm's 1999 statement for critical percolation and $\SLE_6$ that we outlined at the very beginning of this introduction (note also that Theorem 7 in \cite{DST} shows 
rigorously that if an FK($q$)-model scaling limit for $q <4$ has a conformally invariant scaling limit, it would be boundary touching i.e., the value of $\kappa$ would have to be greater than $4$):
\begin {itemize}
\item If a dilute O($N$) model has a non-trivial and conformally invariant scaling limit consisting of simple loops, then necessarily  $N \in (0,2]$ and the scaling limit has to be 
$\CLE_\kappa$ for the value of $\kappa \in (8/3, 4]$ such that $N= - 2 \cos (4 \pi / \kappa)$.
\item If a dense O($N$) model has a non-trivial and conformally invariant scaling limit consisting of non-simple loops, then necessarily  $N \in (0,2]$ and the scaling limit has to be
$\CLE_\kappa$ for the value of $\kappa \in [4,8)$ such that $N= - 2 \cos (4 \pi / \kappa)$.
\item If a critical FK($q$)-model on $\Z^2$ has a non-trivial conformally invariant scaling limit, then necessarily  $q \in (0,4]$ and the scaling limit has to be 
$\CLE_\kappa$ for the value of $\kappa \in [4,8)$ such that  $\sqrt {q} = - 2 \cos (4 \pi / \kappa)$.
\end {itemize}

\subsection{Brief discussion of the related literature} 
Let us briefly discuss the {relationship} between the present contribution and some of the closely related results in the existing literature: 
\begin {itemize}
\item
Connection probabilities and related questions for families of SLE paths have 
 been the focus of a number of interesting mathematical works by Dub\'edat, Zhan, Bauer-Bernard-Kyt\"ol\"a and others (see for instance \cite{Dubedat, Dubedat2, Dubedat3, BBK, Zhan1, Zhan2, Zhan3,kz,kp} and the references therein) that are very much relevant and related
to the present paper, and that we will in fact use.
Let us now very briefly explain where our contribution lies with respect to these references. In these papers, the main focus {is} on the description and classification of the joint law of 
``commuting'' SLE curves (i.e., that satisfy the appropriate multiple-strands generalization of the conformal Markov property that characterize SLE curves), that start from a
 {given even number of boundary points of a simply connected domain (this classification is provided by  Dub\'edat's {\em commutation relations}).}
In the particular case where one looks at four boundary points $a_1, \ldots, a_4$ and assumes that 
all {of the} SLE curves are locally absolutely continuous with respect to $\SLE_\kappa$ curves for the same value of $\kappa$, these commutation relations allow {one} to 
describe all {of the} possible  laws of these curves {up until the curves hit each other}. If one applies this to our precise setup, these results state that one has (for each value of $\kappa$) 
exactly a one-parameter family of possible candidates for the joint law of the four strands of our $\CLE_\kappa$ with two wired boundary conditions.
Basically, one has two extremal solutions such that for the first one, the strand started from~$a_1$ always ends up at~$a_2$, while for 
the second one, it ends up at~$a_4$ (in each of these cases, the law of the pair of strands, or of the individual strands are known under several names:  
intermediate SLE processes, hypergeometric SLEs, bichordal SLE processes). In our setting, these extremal solutions
describe the law of the strands, when one conditions on one or the other of the two connection events. 
In particular, if one {\em assumes} the value of the connection probability in a square, then these commutation ideas
provide the exact form of the connection probability in terms of the  aspect-ratio of the  rectangle 
(and because {these} connection probabilities allow {one} to define a martingale when one explores one strand, it also 
provides the exact description of the driving function of one strand), we will come back to this in Section~\ref{S2}. 
The purpose of the present paper is to provide an actual computation of this connection probability in conformal squares for our wired/free/wired/free boundary conditions, uniquely via CLE considerations and no conjectural relation to discrete models, which (we believe) is a new input. 
\item
The hook-up probability in squares can be derived by other means for some special values of~$\kappa$. The fact that it is equal to $1/2$ for $\CLE_6$
is a direct consequence of the locality properties of $\SLE_6$.  
The fact that the hook-up probability is also equal to  $1/2$ when $\kappa =3$ can be viewed as a consequence of the combination of the fact that $\CLE_3$  
is the scaling limit of the Ising model \cite{bh} and the symmetries in the Ising model (this observation already appears in \cite{Dubedat,BBK} -- at that time the case $\kappa =3$ was a conditional result since it had not yet been established that the Ising model converges to $\SLE_3$ which is now established, see also the recent paper \cite{Wu}). 
Similarly, the fact that the hook-up probability is equal to $1 / (1 + \sqrt{2})$ in the special case where $\kappa =16/3$ can be viewed as a consequence of the fact \cite{ks2015fkcle,ks2015fkcle2} that $\CLE_{16/3}$ is proved to be the scaling limit of the FK($2$) model (see Appendix \ref{App1} for why the crossing probability for this discrete model is $1/(1+ \sqrt{2})$). 
Also, the fact that the hook-up probability tends to $1$ as $\kappa \to {8/3}^+$ and $\kappa \to 8^-$ could be inferred directly from the Brownian loop-soup description from \cite{SheffieldWerner} for the $\kappa \to {8/3}^+$ case, and from the relation to uniform spanning trees when $\kappa \to 8^-$ (see \cite{LSWLERW}). As we will explain in Section~\ref{SGFF}, the fact that the 
connection probability is $1/3$ in the case where $\kappa=4$ can be worked out using the relation between $\CLE_4$ and the Gaussian Free Field. 
\item
 The CLE percolation approach that we developed with
Sheffield \cite{MSWCLEpercolation} gives a continuous version of the relation between FK-percolation and 
the corresponding Potts models in terms of conformal loop ensembles $\CLE_\kappa$ and $\CLE_{16/\kappa}$. 
In \cite{mswbetarho}, combining the results of \cite{MSWCLEpercolation} with ideas from Liouville quantum gravity (LQG), we provide another way to 
get the (conjectural) relation $\sqrt {q} = -2 \cos (4\pi / \kappa)$ out of $\CLE_\kappa$ considerations only 
(this time without any reference to discrete crossing probabilities). Note that this other
CLE/LQG approach does not directly yield the value of the connection probabilities or the relation to the dilute $O(N)$ models, and that it is somewhat more elaborate than the present one. 
It is also related to the body of work from recent years that relate such questions to structures in random geometries/random maps (that can be viewed as 
trying to put some of the theoretical physics considerations onto firm mathematical ground), see for example \cite{shef-kpz,shef-burger,dms2014mating} and the references therein.
\end {itemize}

Let us now make some further bibliographic comments on the conjectural relation with discrete models: 
\begin {itemize}
\item It is interesting to see that the $q=4$ and $N=2$ thresholds for the nature of the phase transition of FK($q$) percolation and O($N$) models
show up from this $\CLE_\kappa$ {computation}, via the fact that the lowest possible $\CLE_\kappa$ hook-up probability in
conformal squares is $1/3$. Note that $q=4$ has been recently proved (rigorously and based on the study of 
discrete models) to be the threshold for existence of a continuous phase transition for FK($q$)-percolation models on $\Z^2$ \cite{DST,DGHMT}. In particular, \cite{DGHMT} shows 
rigorously that the scaling limit of the critical FK($q$) model's interfaces are trivial when $q {>} 4$ (which is of course a much stronger statement than the conditional ``then necessarily $q \in (0,4]$'' in the FK-part of our statement above). 
\item
As we have already mentioned, these relations between $q$ and $\kappa$, and between $N$ and $\kappa$ have appeared in numerous papers before. But, to the best of our knowledge, except for 
the specific particular cases of $\kappa$ that we have already mentioned, they were not based on rigorous SLE-type considerations. More precisely, the argument was the following: If these models have a conformally invariant scaling 
limit, then it must be described by an $\SLE_\kappa$ for some $\kappa$. In order to identify which value of $\kappa$ is the right one, 
the idea in those papers was to match some computation of probabilities of 
events for SLE (or of critical exponents, or of dimensions) with the corresponding values that were {\em predicted} to be the correct ones for the scaling limit of the discrete model, based on theoretical physics considerations (see for instance \cite{RS} for the FK($q$) conjecture based on the physics dimension predictions; another approach is related to the discrete parafermionic observable for FK models -- see \cite{SmirnovFK} or \cite{DuminilParafermionic} for a detailed discussion and more references -- where
the spin is defined as $\sigma = 1 -  (2  / \pi ) \arccos (\sqrt {q} /2 )$, and that is conjectured 
to correspond in the scaling limit to some SLE martingale, see for instance \cite{Werness}).
In the present paper,  we identify the candidate value of $\kappa$ using a feature that is rigorously known to hold in the discrete model (and therefore in its scaling limit, if it exists). 
So, in a way, one could view it as an SLE/CLE derivation of the conjectural relation between the lattice models and the corresponding conformal field theory (i.e., a relation between the $q$ or $N$ and the central charge $c = (6 -\kappa ) (3\kappa - 8) / (2\kappa)$ which can be derived via the SLE restriction property \cite{LSWr} 
or the loop-soup construction of CLE \cite{SheffieldWerner}). 
\item
In relation with the CLE percolation item mentioned above, let us just stress that the relation
$ (2 / \kappa) + (2 / \tilde \kappa) = 1$ between the values $\kappa \in (8/3, 4)$ and $\tilde \kappa \in (4,8)$ that have the same connection probability is not at 
all the same as the $\kappa \kappa' = 16$ duality relation between $\SLE_\kappa$ and $\SLE_{\kappa'}$ from \cite{Zhan1,Dubedat2} 
or the Edwards-Sokal coupling between $\CLE_\kappa$ and $\CLE_{\kappa'}$ derived in \cite{MSWCLEpercolation}. 
It is however possible to combine the present work (or the results of \cite{mswbetarho}) with the CLE percolation results of \cite{MSWCLEpercolation} -- this for instance indicates 
based on CLE considerations only that if the scaling limit of the critical FK-percolation model for $q=3$ on $\Z^2$ is non-trivial and conformally invariant, then
 {(modulo some a priori 
arm exponents estimates)}
the scaling limit of the critical Potts model for $q=3$ would exist as well and could be described in terms of $\CLE_\kappa$ for the value of $\kappa \in (8/3, 4]$ such that 
$\sqrt {q} = - 2 \cos ( 4\pi \kappa / 16) = -2 \cos (\pi \kappa /4)$ i.e., $\kappa = 10/3$.
\end {itemize}

\subsection {Structure of the paper} 

Let us describe the structure of the paper, and explain where we use which results from other papers:  

In Section~\ref{SCLE}, we will first recall some background material about CLEs and their properties, and then define what we will call CLE$_\kappa$ with two wired boundary arcs
when $\kappa \in (8/3, 4) \cup (4,8)$. As we will explain, this builds on the shoulders of some previous work: The definition and conformal invariance of CLEs from \cite{Sheffield,SheffieldWerner,IG3}, the exploration features of CLE as studied in \cite{WernerWu,MSWCLEpercolation}, and last but not least, the conformal invariance of hook-up probabilities as derived in \cite{mswrandomness}.  This is a section where we use directly and indirectly various background material from earlier papers on CLE. Those readers who want to focus on the computational part of the derivation of Theorem~\ref{mainprop} can choose to take the (rather intuitive) results of that section for granted.

In Section~\ref{SGFF}, which can be viewed as a brief interlude, we discuss the special case of CLE$_4$, and explain how it is possible to prove Theorem \ref{mainprop} in 
that case, using the coupling between the Gaussian free field and $\CLE_4$ (from \cite{MScle}, see also \cite{ASW}).

In Section~\ref{S2}, we briefly survey what the aforementioned works on commutation relations  \cite{Dubedat, Dubedat2, Dubedat3, BBK, Zhan1, Zhan2, Zhan3} do imply in our 
setup of wired CLEs. 

In Section~\ref{Smain}, we describe the main steps of our proofs of Theorem~\ref{mainprop}, separately in the cases $\kappa \in (8/3, 4)$ and $\kappa \in (4, 8)$, that allow {us} to reduce the determination of the connection probability {to} actual concrete estimate{s} of {probabilities of} events that we then derive in Sections~\ref{Sproof1} and~\ref{Sproof2}.
In the case of simple CLEs, the arguments in Section~\ref{Smain} will rely on the loop-soup construction of $\CLE_\kappa$, and more specifically on the decomposition of loop-soup clusters and the relation to restriction measures, as studied in \cite{QianWerner,Qian2}. 
The actual computations in  Sections~\ref{Sproof1} and~\ref{Sproof2} will involve some considerations {involving} SLE and hypergeometric functions. 

We conclude with two short appendices, recalling very briefly some basics about O($N$) models and their connection probabilities, and about the properties of hypergeometric functions that we are using in our proofs.

\section {Defining CLE with two wired boundary arcs} 
\label {SCLE}

\subsection {The case where $\kappa \in (4,8)$}

We first explain how to define the CLE$_\kappa$ with two wired boundary arcs in the case where $\kappa \in (4,8)$. 
Let us first quickly recall some features of {the} CLE$_\kappa$  itself for $\kappa \in (4,8)$ -- the reader may wish to consult \cite{MSWCLEpercolation} for more details and further references. For our purpose, it will be sufficient to focus on their non-nested versions.

Conformal loop ensembles {were} proposed in \cite{Sheffield} as the natural candidates that should describe the joint law of outermost interfaces in a number of 
critical models from statistical physics in their scaling limit. Sheffield's construction in  \cite{Sheffield} is based on the target-invariance (see \cite{Dubedat,sw2005coordinate}) of special variants of $\SLE_\kappa$, the $\SLE_\kappa (\kappa -6)$ processes. 
In the case where $\kappa \in (4,8)$, the definition of these SLE$_\kappa (\kappa -6)$ processes is rather direct 
 (see \cite{Sheffield,MSWCLEpercolation,WernerWu} for background).

 Their target-invariance property makes it possible to define for each simply connected domain $D$ and each given 
 boundary point $x$, a branching tree of such $\SLE_\kappa (\kappa - 6)$ processes, starting from some point $x$ (that is referred to as the root 
 of the branching tree) and that targets all points in $D$. 
 Branches of the tree trace loops along the way, so that this branching tree defines a random collection of loops, that Sheffield called $\CLE_\kappa$. 

 Note that with this construction, the law of this $\CLE_\kappa$ seems to depend on the choice of the root.
However, when $\kappa \in (4, 8)$, one can show that this is not the case 
(and more generally, that the law of CLE$_\kappa$ is invariant under any given conformal automorphism of $D$) 
as a direct consequence of the reversibility properties of $\SLE_\kappa (\kappa -6)$ processes. This reversibility for non-simple SLE 
 paths is a non-trivial fact, that has been established in \cite{IG3} using the connection with the Gaussian free field (GFF) and more precisely via the ``imaginary geometry'' approach developed in \cite{IG1,IG2,IG3,IG4}. 
However, with this (non-trivial) fact in hand, constructing the CLE and deriving its properties is a rather easy task. 
 We refer to \cite{MSWCLEpercolation,mswrandomness} for more details and references. 

When $\kappa \in (4,8)$, the CLE$_\kappa$ is a collection of non-simple SLE$_{\kappa}$-type 
loops in $D$. Some of them do touch each other, and some of them do touch the boundary of 
the domain. For the purpose of the present paper, it will be in fact sufficient to focus on the set of CLE$_\kappa$ loops that do touch $\partial D$. 

An important feature of this case where $\kappa \in (4,8)$ is that the branching tree is in fact a simple deterministic function of the CLE that it constructs. 
For instance, if one chooses two distinct boundary points $x$ and $x'$ and the counterclockwise boundary arc $\partial$ of $D$ from $x$ to $x'$, one can 
look at the collection of all the CLE loops that touch this arc. One can define the path obtained by starting from $x$ and that moves along $\partial$ and each time it meets 
a CLE loop for the first time on this arc, it goes around it clockwise. In this way, one obtains a continuous path from $x$ to $x'$. If one reparametrizes this path 
by its ``size'' seen from $x'$ (and therefore excises from it all the parts that are in fact ``hidden'' from $x'$ at the moment at which they are drawn), one obtains exactly the branch  from $x$ to $x'$ in the SLE$_\kappa (\kappa -6)$ tree (that is an SLE$_\kappa (\kappa -6)$ from $x$ to $x'$). 

 \begin{figure}[ht!]
  \includegraphics[width=0.4\textwidth]{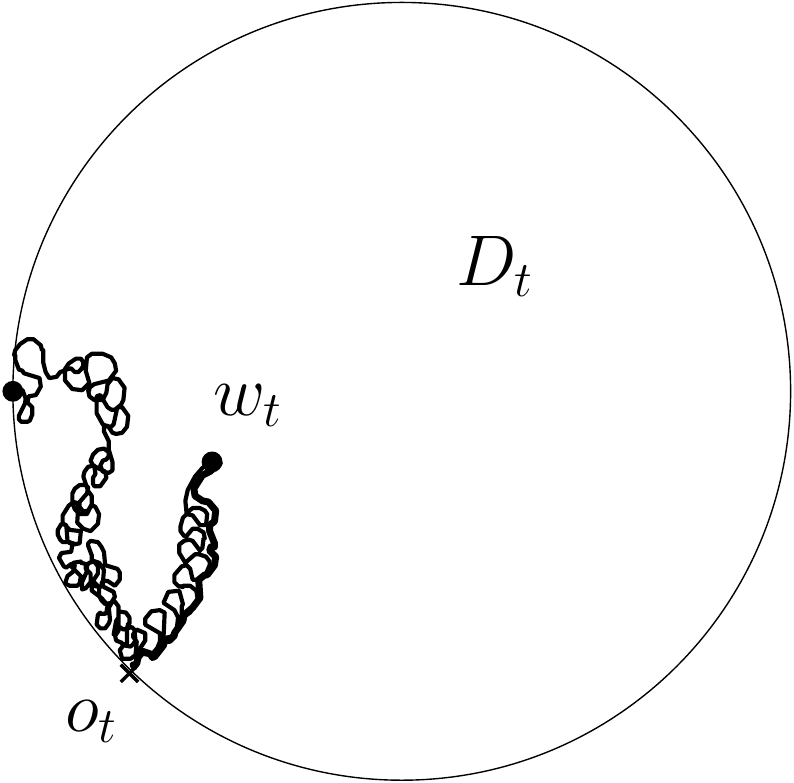}
  \caption{Sketch of the beginning of the SLE$_\kappa (\kappa -6)$ arc of the branching tree (for non-simple  CLEs) from $-1$ to $1$ in the unit disk.  The bold simple 
  arc from $o_t$ to $w_t$ denotes the ``wired'' portion of $D_t$.}
  \label{Pic21b}
\end{figure}

The branching tree description allows {one} to define immediately what one can refer to as $\CLE_\kappa$ with one wired boundary arc. 
Suppose that one grows an $\SLE_\kappa (\kappa -6)$ process $\gamma$ in a simply connected domain $D$, starting from the boundary point $x$ and 
targeting some other boundary point $x'$, and that one stops this branch of the branching SLE tree at a (deterministic or stopping) time $t$. 
 {We will say in the sequel that 
the exploration at time $t$ is {\em in the middle of tracing a loop} if at time $t$, it has started to draw a portion CLE loop that is visible from $x'$ (see Figure~\ref{Pic21b}). In this $\kappa \in (4,8)$ case, this just means that $\gamma (t) \notin \partial$. 
Suppose that $t$ is such a time.} We will write $w_t = \gamma (t)$ and $o_t$ will denote the point of the CLE loop that $w_t$ is part of that $\gamma$ did hit first
(so in this $\kappa \in (4,8)$ case, it is the last point on $\partial$ that $\gamma$ has visited before $t$).  
We can then ask what is the conditional distribution of the $\CLE_\kappa$ in the component $D_t$ of $D \setminus \gamma [0,t]$ which has $x'$ on its boundary. 
If we map  $D_t$ conformally onto the upper half-plane via the conformal transformation $g_t$ 
with $\{ g_t (o_t), g_t (w_t) \} = \{ 0, 1 \}$ and $g_t (x') = \infty $ (which one of $o_t$ or $w_t$ is mapped onto $0$ depends on whether the loop containing $w_t$  is being traced counterclockwise or clockwise), then 
the SLE$_\kappa (\kappa -6)$ features and the properties of the conformal loop ensembles immediately show that the image under $g_t$ of this conditional distribution can 
be described as follows (this works for all $\kappa  \in (4, 8)$): 
\begin{itemize}
\item First finish the currently traced loop by sampling an ordinary $\SLE_\kappa$ in $\HH$ from $1$ to $0$.
\item Then sample independent CLEs in the connected components of the complement of the SLE that are ``outside'' of the loop obtained by concatenating the SLE with the segment $[0,1]$. 
\end{itemize}
This is what we will call the law of 
CLE in $\HH$  with one wired boundary arc on $[0,1]$.
By conformal invariance, this can then also be defined for any simply connected domain $D$ with a given boundary arc.

 We can now move to the definition of CLE with two wired boundary arcs: If we are in the same setup as above, we can trace the branch of the
SLE$_\kappa (\kappa -6)$ tree that joins $x$ with $x'$ by starting at $x$, or by starting at $x'$ (so, we can consider both the SLE$_\kappa (\kappa -6)$ from 
$x$ to $x'$ and its time-reversal; note that this time-reversal will ``move clockwise'' along $\partial$ and collect the CLE loops along the way, so that there is 
some chiral change when one takes the time-reversal).

\begin{figure}[ht!]
  \includegraphics[width=0.4\textwidth]{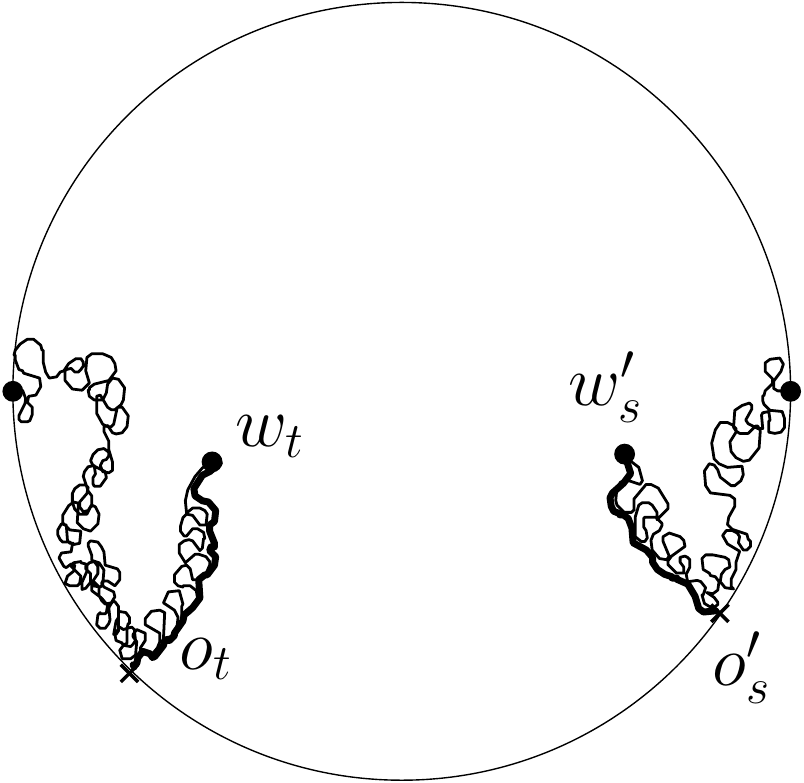}
  \caption{Exploring CLEs from both sides and creating the CLEs with two wired boundaries.}
  \label{pic23b}
\end{figure}

We can stop both the forward path and its time-reversal at some stopping times during which they are both tracing loops, 
and in such a way that the two paths have not yet hooked up (which when the union of the two paths form the entire SLE$_\kappa (\kappa -6)$); 
we allow them to hit each other though -- this can happen before they hook-up because the CLE loops 
are not all disjoint. 
More precisely, we let $t$ be a stopping time defined for the natural filtration
of the first exploration, and then we let $s$ be a stopping time for the filtration generated by the second exploration, possibly augmented by
the knowledge of the first exploration until $t$.  
We denote the remaining to be explored domain by $D_{t,s}$ and the four marked boundary points corresponding to the special points of each of the explorations 
by $(w_t, o_t, o_s', w_s')$  with obvious notation (see Figure~\ref{pic23b}).

Then,  \cite[Lemma~3.1]{mswrandomness} states that the conditional distribution of the part of the CLE that lies in $D_{t,s}$ is conformally invariant 
with respect to the configuration $(D_{t,s}, w_t, o_t, o_s', w_s')$. Note that this is now a configuration consisting of loops in $D_{t,s}$ and of two disjoints arcs (that correspond to the missing parts of the loops the $w_t$ and $o_s'$ are part of) in $D_{t,s}$. 
Let us insist on the fact that this conformal invariance statement is not trivial to prove, even if it may at first glance seem intuitively obvious
(one should keep in mind that the definitions of CLE themselves are not straightforward at all). 
This conditional distribution is what we will refer to as the $\CLE_\kappa$ in $D_{t,s}$ with the wired boundary arcs $w_to_t$ and $o_s'w_s'$.

\subsection {CLE background in the case $\kappa \in (8/3, 4)$}  

When $\kappa \in (8/3, 4)$, the construction of the CLE$_\kappa$ with two wired boundary arcs is a little more complicated due to 
the fact that the SLE$_\kappa (\kappa -6)$ path is not a deterministic function of the CLE$_\kappa$. 
In the present subsection, we recall (in a rather narrative way) some 
relevant background on SLE$_\kappa (\kappa -6)$ exploration trees in this case (we again refer to \cite{SheffieldWerner,WernerWu,MSWCLEpercolation} for 
references and details). 
We will also discuss some features involving Brownian loop-soups that will be useful later in the paper. 

When $\kappa \in (8/3, 4)$, one can still define the CLE$_\kappa$ via an SLE$_\kappa (\kappa -6)$ branching tree rooted at some boundary point $x$, but there are several points that we would like to emphasize: 

- In order to define the SLE$_\kappa (\kappa -6)$ processes, one has to use L\'evy compensation and/or side-swapping. In other words, one has to choose 
 a side-swapping parameter $\beta \in [-1,1]$ (that will eventually describe the probability for each individual CLE loop to be traced clockwise or counterclockwise
 by the exploration tree). 
 So, for each $\kappa \in (8/3, 4)$, one has a one-parameter family of SLE$_\kappa^\beta (\kappa -6)$ processes. 
 In the sequel, we will work only with the totally asymmetric cases $\beta=1$ and $\beta=-1$ (in the former case, the loops are all traced counterclockwise and in the 
 latter case, they are all traced clockwise).  

- The target-invariance of these processes enable us, for each choice of the root, to define the SLE$_\kappa^\beta (\kappa -6)$ branching tree, and the collection of 
loops that it traces. In order to prove that the law of this collection of loops does not depend on the choice of the root, 
 the proof in \cite{SheffieldWerner} uses the Brownian loop-soup in $D$ introduced in \cite{lw}, which is a natural 
 Poisson point process of Brownian loops in $D$, with intensity given by a constant $c$ times a certain natural measure on Brownian loops.
 
 Loop-soups will be useful in the present paper too, so we give a few more details about those: 
 The loops in a loop-soup can be thought of as being independent, so that they can overlap and therefore create clusters of Brownian loops. The fact that the outer boundaries of loop-soup clusters would give 
 rise to random collections of SLE-type loops that behave nicely under perturbations of the domains that they are defined in had been outlined in \cite{Wcras}. 
 As it turns out, it has then been proved in  \cite{SheffieldWerner} that: (a) for all $c \le 1$, if one considers the outermost 
 boundaries of these Brownian loop-soup clusters, one obtains a countable conformally invariant collection of mutually disjoint simple loops and (b) that for all $x$, this collection of loops 
 coincides with the $\CLE_\kappa$ defined by the branching tree construction  {where $\kappa \in (8/3, 4)$ is given by the relation $c = (6 -\kappa ) (3\kappa - 8) / (2\kappa)$
 (this intensity $c$ is exactly the central charge appearing in CFT).} Since the loop-soup construction of $\CLE_\kappa$ does not involve any boundary root, this
 therefore proves that the law of the  $\CLE_\kappa$ defined by the branching tree construction is indeed independent of the root of the tree. 
 The fact that one has these two different descriptions of the same $\CLE_\kappa$  (via the Brownian loop-soup or via branching $\SLE_\kappa (\kappa -6)$ processes) is useful  when one tries to derive further properties of {the} conformal loop ensembles, and in the present paper, we will in fact use both these constructions in our proofs. 
   
The branching tree description allows {one} to define (just as in the case $\kappa \in (4,8)$) 
the $\CLE_\kappa$ with one wired boundary arc. See Figure~\ref{Pic21}.

 \begin{figure}[ht!]
  \includegraphics[width=0.4\textwidth]{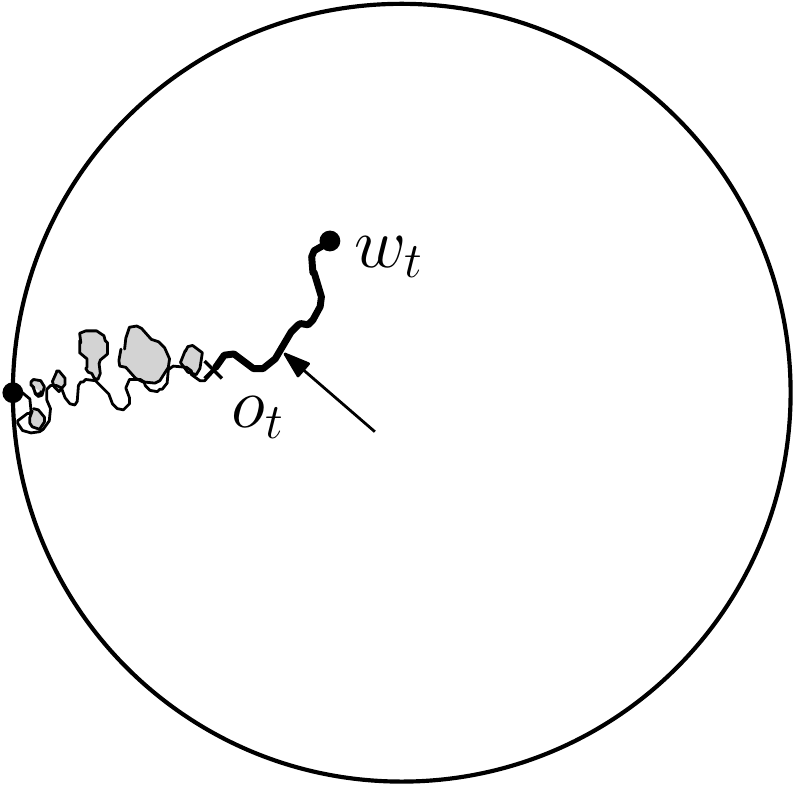}
  \caption{Sketch of an SLE$_\kappa^{-1} (\kappa -6)$ from $-1$ to $1$ in the unit disk stopped while tracing a loop. The bold arc denotes the ``wired'' portion of $D_t$. The arrow indicates the ``inside'' part of the partially traced simple loop which corresponds to this wired part. (Note that the trunk of this SLE$_\kappa^{-1} (\kappa -6)$ is in fact a boundary-touching non-simple curve, but it is drawn here as a simple curve to simplify the understanding of this sketch).}
  \label{Pic21}
\end{figure}

One can view the $\SLE_\kappa (\kappa -6)$ processes as being constructed via a Poisson point process of ``SLE$_\kappa$ bubbles''  with intensity given by the so-called $\CLE_\kappa$ bubble measure (or more precisely with intensity given by this measure times the Lebesgue measure on $[0,\infty)$, so that these bubbles are ordered according to their arrival time $u_j$), see \cite{WernerWu,MSWCLEpercolation} for background. 
The bubble measure in the upper half-plane is the appropriately rescaled limit as $\eps \to 0$ of the law of 
$\SLE_\kappa$ from $\eps$ to $0$ in the upper half-plane. The pinned configuration is then obtained from the bubble configurations by adding in an independent $\CLE_\kappa$ 
in the outside of the bubble. So, the measure on pinned configuration can be viewed as the appropriately renormalized limit as $\eps \to 0$ of the law of $\eps$ times the wired $\CLE_\kappa$. 

If one combines this description of the SLE$_\kappa^\beta (\kappa - 6)$ processes with the Markovian property of CLE described in \cite{SheffieldWerner}, one can readily obtain the following result: 
Consider two simply connected domains $D$ and $D'$ with $D' \subset D$, a boundary point $x$ of $D$ that is at positive distance from $D \setminus D'$, and another boundary 
point (or prime end) $x'$ on $\partial D \cap \partial D'$. Let us first sample an SLE$_\kappa^\beta (\kappa -6)$ (that we call $\eta$) from $x$ to $x'$ in $D$, that is coupled with a CLE$_\kappa$ (that we denote by ${\mathcal C}$) in 
$D$. Then, just as for the CLE spatial Markovian property described in \cite{SheffieldWerner}, we consider the domain obtained by removing from $D'$ all the CLE loops (and 
their interior) that do not entirely stay in $D'$, and we consider $\tilde D$ to be the connected component of the obtained set that has $x$ (and also $x'$) on its boundary. 
The CLE spatial Markovian property, states that the conditional law (given $\tilde D$) of the set $\tilde {\mathcal C}$ of CLE loops in ${\mathcal C}$ that do stay entirely in $\tilde D$, is exactly that of a CLE in $\tilde D$. One can then wonder if $\eta$ can be viewed as an SLE$_\kappa^\beta (\kappa -6)$ process in $\tilde D$ as well. As it turns out, if $\tau$ denotes the time at which $\eta$ hits $D \setminus \tilde D$, the conditional law of $(\eta_t, t \le \tau)$ is indeed 
that of an SLE$_\kappa^\beta (\kappa -6)$ process from $x$ to $x'$ in $\tilde D$ coupled to $\tilde {\mathcal C}$, and stopped at its first hitting of $D \setminus \tilde D$. 
(We leave the details of the proof to the interested reader). This is illustrated in Figure \ref {explorestriction} and 
will be useful later in the paper (in the proof of Lemma \ref {main2}).

\begin{figure}[ht!]
  \includegraphics[width=0.4\textwidth]{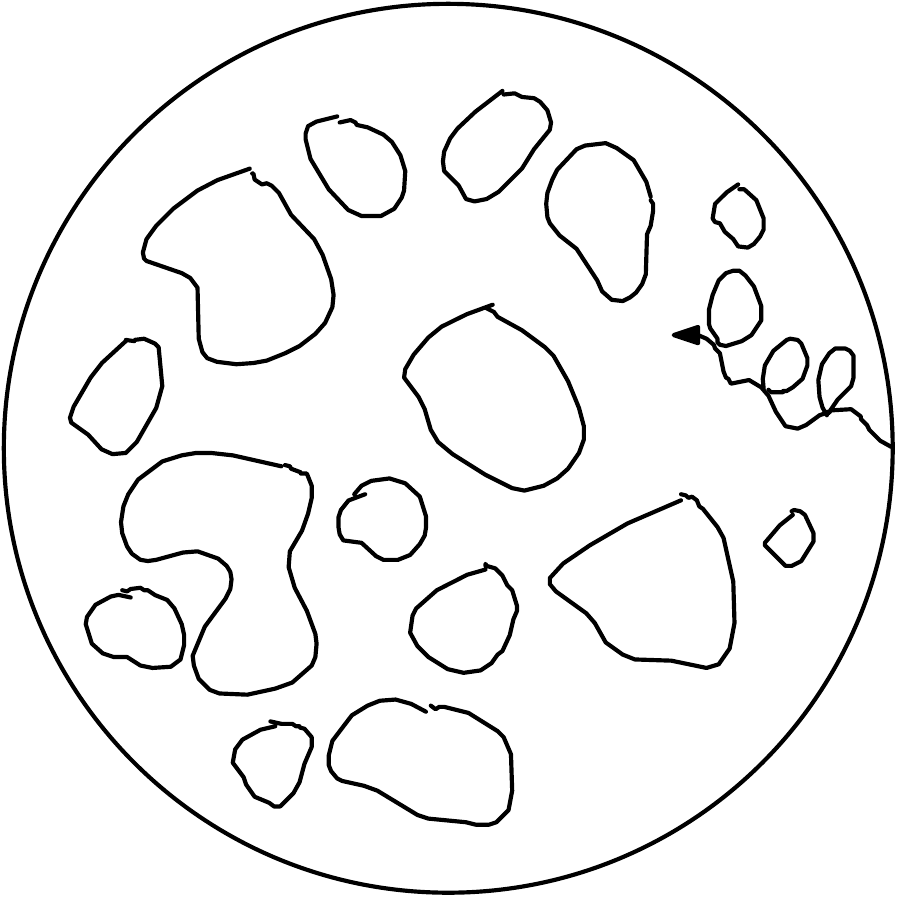} \quad  \includegraphics[width=0.4\textwidth]{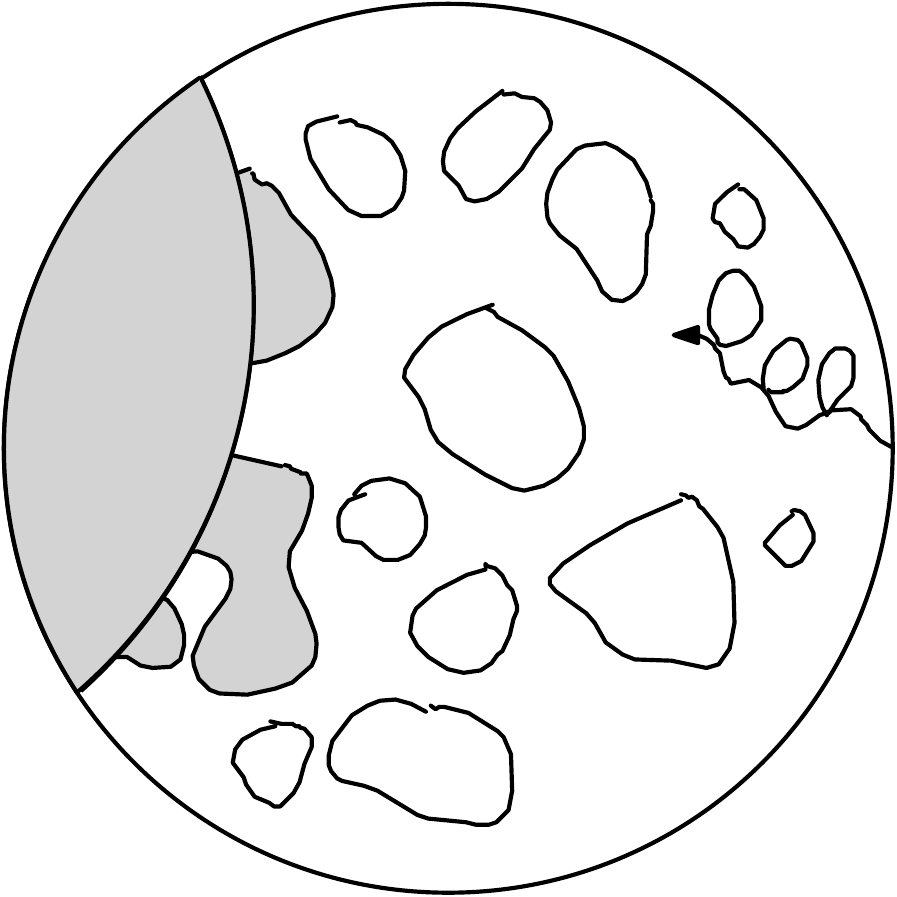}
  \caption{The beginning of the SLE$_\kappa^{1} (\kappa -6)$ in exploring the CLE in $D$ (left figure) 
  is also that of an SLE$_\kappa^1 (\kappa -6)$ exploring the CLE loops in $\tilde D$ (right figure), where $\tilde D$ has been obtained by restriction.}
  \label{explorestriction}
\end{figure}

As explained for instance in \cite{SheffieldWerner} or \cite{WernerWu}, one can also use procedures {other} than these SLE$_\kappa^\beta (\kappa -6)$ exploration tree
to discover the loops of a $\CLE_\kappa$ in a ``Markovian'' way. 
This includes for instance deterministic explorations such as discovering one after the other and in their order of appearance, all the CLE loops that intersect a given deterministic curve that starts on the boundary; for {example} in the unit disk, start 
from $1$, and trace one after the other all loops that intersect the segment $[-1,1]$ starting from $1$, until one hits the imaginary axis.
This last procedure will sometimes jump along the real axis, but the CLE property will ensure that 
the previously defined wired CLE describes also the conditional law of the CLE when one stops the exploration in the middle of a loop. 
Such an exploration can be quite useful for $\kappa \in (8/3,4)$ because 
it is a deterministic function of the CLE (which is actually not the case for the branching tree exploration when $\kappa \in (8/3, 4)$, see \cite{mswrandomness}).

Let us finally review some results about the decomposition of Brownian loop-soup clusters from \cite{QianWerner, Qian2} when they 
are partially discovered by an SLE$_\kappa^\beta (\kappa -6)$ exploration:  
Consider a $\CLE_\kappa$ ${\mathcal C}$ for $\kappa \in (8/3, 4]$ in the unit disk, that was obtained from a Brownian loop-soup ${\mathcal L}$ with intensity $c$ (as explained in \cite{SheffieldWerner}),
and start exploring from $x$  the branch of the SLE$_\kappa^\beta (\kappa -6)$ tree coupled to this CLE (in such a way that the tree and the loop-soup are conditionally independent given the CLE) from the boundary point $-1$ targeting $1$, and stop it at some stopping time $t$, when one has partially but not fully traced a CLE loop. 
Then, we have just explained in the previous paragraphs that in the remaining to be discovered domain $D_t$, the conditional distribution of the CLE is that of a wired CLE (wired on the 
already traced part $\partial_t$ of the loop that one is tracing at time $t$). 
As explained in \cite{Qian2}, one can then also derive the following aspects of the conditional 
distribution of the loop-soup itself in $D_t$ (see Figure~\ref{onerestr} for a sketch):
\begin {itemize}
 \item The outer boundary of the union of all the Brownian loops that do intersect $\partial_t$ form a one-sided restriction measure with exponent $\alpha = (6-\kappa) / (2\kappa)$ attached to $\partial_t$ (see \cite{LSWr,Wconfrest} for background and definitions about restriction measures).
 \item The set of Brownian loops in~$D_t$ that do not intersect $\partial_t$ form an independent loop-soup in~$D_t$.
\end {itemize}
This result is stated in \cite{Qian2} for the case of the deterministic Markovian explorations (that discover the loops that intersect a given deterministic path), 
but it is easy to apply the results of Section 5.2 of \cite{Qian2} about partially explored ``pinned CLE'' configurations to deduce the former statement about $\SLE_\kappa^\beta (\kappa -6)$ explorations. 
Indeed, recall that we know on the one side from the aforementioned results from \cite{SheffieldWerner, WernerWu} that $\SLE_\kappa^\beta (\kappa -6)$ processes can be built from a Poisson point process of $\SLE_\kappa$ bubbles, and we know on the other side from \cite{QianWerner} that the conditional distribution of the Brownian loops given the CLE that they generate is composed of independent samples inside each CLE loop. This shows that in order to obtain the previous decomposition result, it is enough to derive the corresponding result for partially explored bubbles; these are exactly the statements in Section 5.2 of \cite{Qian2}. 

\begin{figure}[ht!]
  \includegraphics[width=0.4\textwidth]{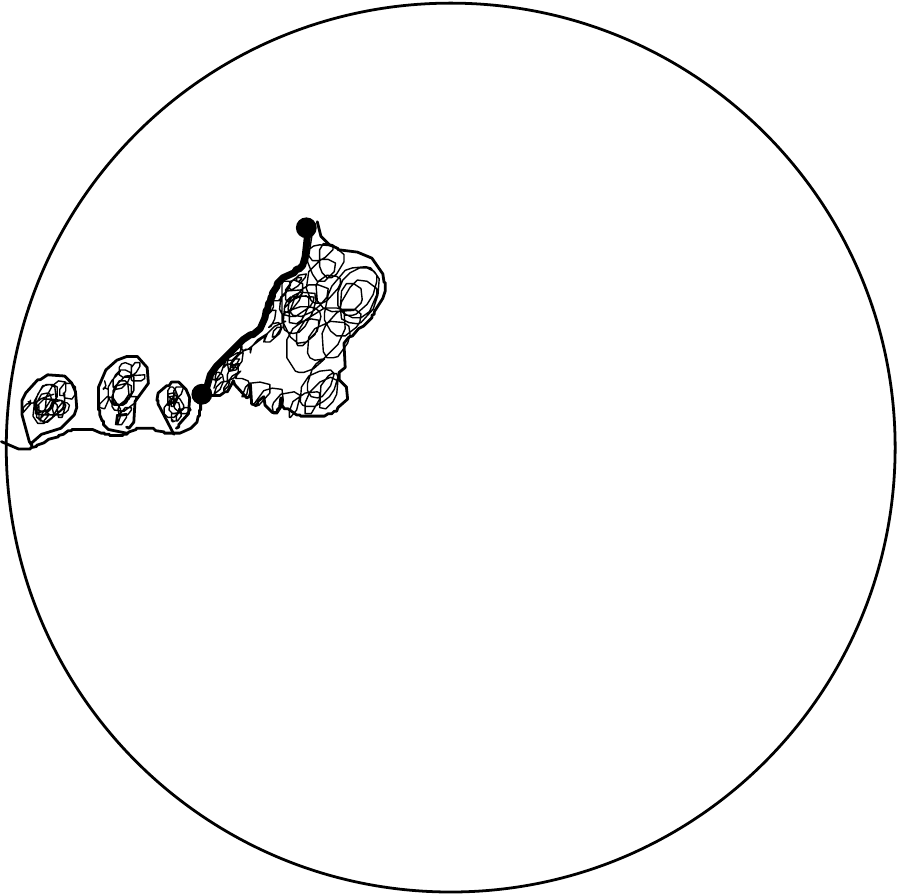}
   \quad
   \includegraphics[width=0.4\textwidth]{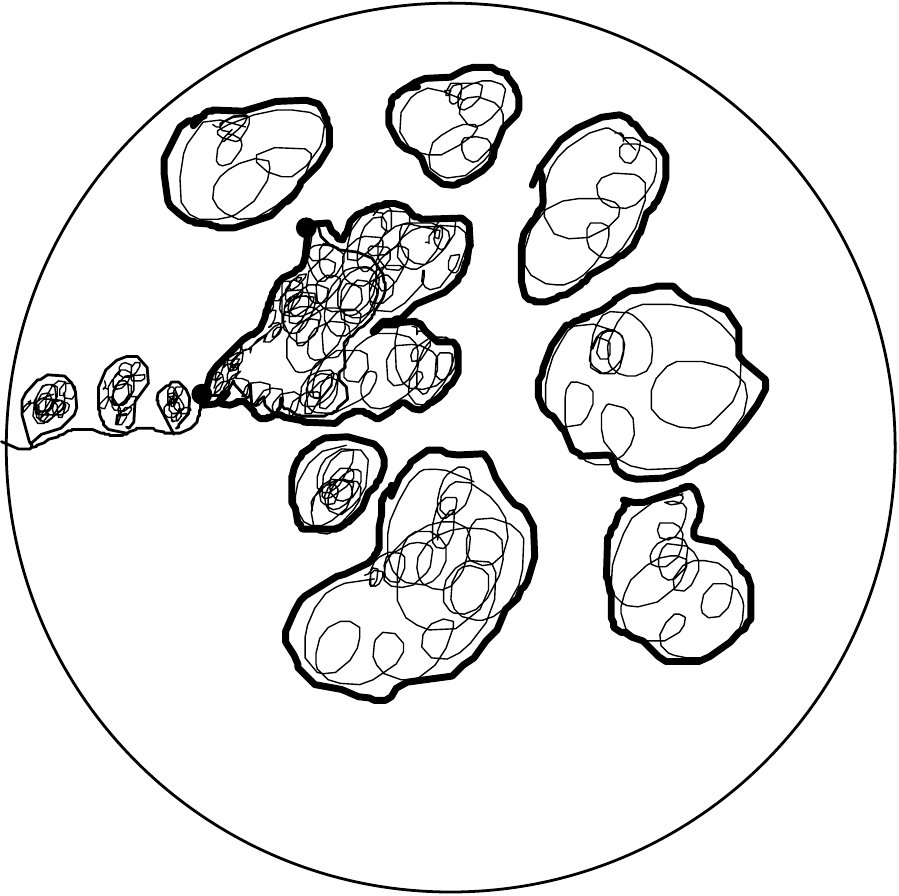}
  \caption{An exploration of the CLE with the encountered Brownian loops (left). Adding the remaining Brownian loops completes the CLE (right).}
  \label{onerestr}
\end{figure}

\subsection {CLE with two ``wired'' boundary arcs for $\kappa \in (8/3, 4)$} 
\label {ss2ba}

When $\kappa \in (8/3, 4)$, the exploration tree is not a deterministic function of the CLE (see \cite{mswrandomness}) that it constructs, 
 so that it is less clear how to properly define the joint law of two explorations of the same CLE starting from different two points $x$ and $x'$, which is what 
 we want to do in order to define CLE with two wired arcs. 
One natural option would be to consider two explorations that are conditionally independent given the CLE, and to try control how they are correlated.
We will instead (though this could be shown to be an equivalent definition) build on some results from \cite{MSWCLEpercolation} about the loop-trunk decomposition of 
these SLE$_\kappa (\kappa -6)$ processes and their relations to CLE. 
The idea is now the following (here we suppose that $\kappa \in (8/3, 4)$ is fixed and we choose to work with SLE$_\kappa^{-1} (\kappa -6)$):   

- First start an SLE$_\kappa^{-1} (\kappa -6)$ exploration $\gamma$ (coupled with a CLE ${\mathcal C}$ from $x$ targeting $x'$ as before, and stop it at some stopping time $t$ (at which it is tracing a loop). The conditional law of the CLE given that branch is then given by the CLE in $D_t$ with one wired boundary arc (joining $w_t$ and $o_t$). Let us call ${\mathcal F}$ the 
$\sigma$-field generated by the SLE$_\kappa^{-1} (\kappa -6)$ up to this time. 

- Then, we can choose to complete the loop that this SLE$_\kappa^{-1} (\kappa -6)$ is tracing at time $t$. This provides some additional information that is 
not contained in ${\mathcal F}$. Then in the new and smaller remaining domain $\hat D_t \subset D_t$,
the conditional law of the CLE loops $\hat {\mathcal C}$ in that domain is that of an ordinary CLE$_\kappa$. 

- Then, in $\hat D_t$, we trace now an SLE$_\kappa^1 (\kappa -6)$ exploration path of this CLE $\hat {\mathcal C}$ starting from $x'$ (targeting $o_t$).
Note that this SLE$_\kappa^1 (\kappa -6)$  is now going in the ``opposite direction'', which is why we choose $\beta =1$ instead of $\beta=-1$ for it).
This exploration is what we choose to be our second exploration path $\gamma'$ starting from $x'$. 

Note however that the two following important points need to be stressed: 

- When we observe $\gamma$ up to time $t$ and then $\gamma'$ up to some stopping time $s$, we then look at the joint law of these two processes but have ``forgotten'' about 
the additional information that is provided in the second step. In other words, we look at the conditional law of $\gamma'$ given ${\mathcal F}$. 

- When $\gamma'$ does hit the loop that $\gamma$ was tracing at time $t$ (and this does happen almost surely at some time $\tau'$, because $\gamma'$ is targeting $o_t$), then we choose to continue $\gamma'$ by just moving along that CLE$_\kappa$ loop counterclockwise (which is the opposite orientation than $\gamma$). 

Hence, when we are looking at a configuration of $(D_{t,s}, w_t, o_t, o_s', w_s')$, we do not know for sure whether $\tau' < s $ or not. This corresponds exactly to 
the question whether $\gamma$ at time $t$ and $\gamma'$ at time $s$ are tracing the same loop of ${\mathcal C}$ or not. 

Building on the loop-trunk decomposition of \cite{MSWCLEpercolation}, it is explained in Section 3.2 of \cite{mswrandomness} (this is Lemma 3.5 in \cite{MSWCLEpercolation}, 
with the role of $\beta =1$ and $\beta = -1$ reversed, which just corresponds to looking at a symmetric image of the CLE) that the conditional 
probability of the fact that $\gamma$ at time $t$ and $\gamma'$ at time $s$ are exploring the same loop is 
a conformally invariant function of $(D_{t,s}, w_t, o_t, o_s', w_s')$. 
More generally (see Section \ref {general}), the law of the whole configuration of loops in $D_{t,s}$ is then 
a conformal invariant function of $(D_{t,s}, w_t, o_t, o_s', w_s')$.
This is what we call the CLE$_\kappa$ in $D_{t,s}$ with two wired boundary arcs.

\begin{figure}[ht!]
   \includegraphics[width=0.4\textwidth]{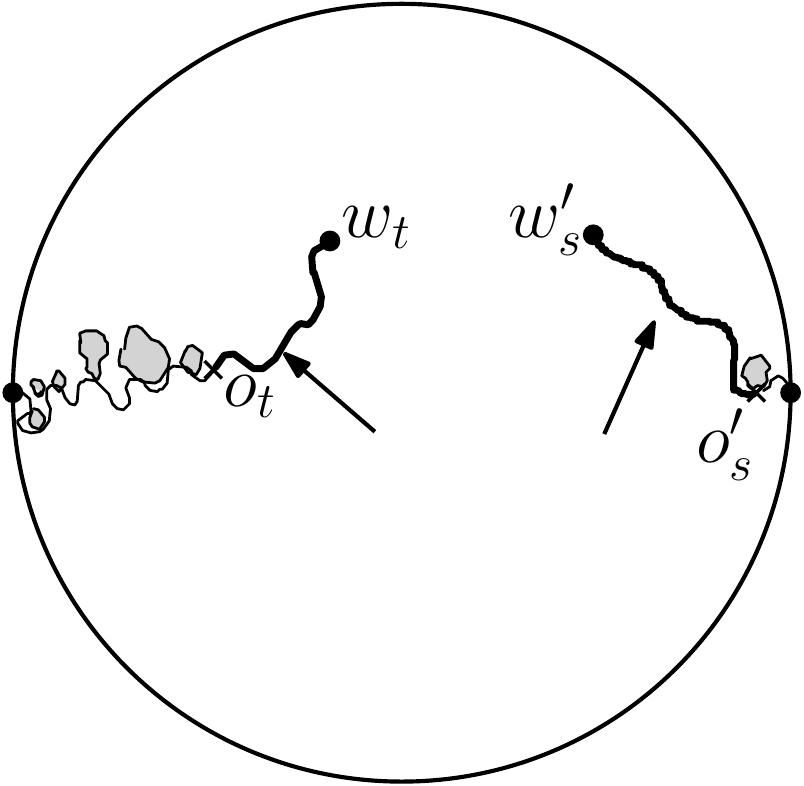}
  \caption{Exploring CLEs from both sides and creating the CLEs with two wired boundaries.}
  \label{pic23}
\end{figure}

One reason to prefer to work with the totally asymmetric explorations in the present paper is that when $\beta \not\in \{ -1, 1\}$, the information about the 
direction in which $\gamma$ and $\gamma'$ are tracing their respective loops can provide a bias about the fact that these two loops 
are the same or not (indeed, in the loop-trunk setup from \cite{mswrandomness} $\gamma$ and $\gamma'$ trace them in the same orientation, then they can not be the same loop). We will see a similar feature in our analysis of the $\kappa =4$ case.

\subsection {A general remark} 
\label {general}
 Finally, we can note that for all $\kappa \in (8/3, 8) \setminus \{ 4 \}$, if we give ourselves a positive $L$, it is possible to choose $t$ and $s$ in such a way that 
with probability one,  $D_{t,s}$ can be mapped conformally onto the rectangle $[0,L] \times [0,1]$ in such a way {that} the four marked boundary points get mapped onto the four corners of the rectangle (for instance, choose first any $t$ and then explore the second strand until the first time $s$ at which $D_{t,s}$ is conformally equivalent 
to such a rectangle -- we know that $s$ exists because the two explorations will eventually hook up). {For all given $L$ and each $\kappa$, this allows one} to define a law $P_L$ on configurations in $[0,L]\times [0,1]$ such that for any $(s,t)$ as described in the previous paragraphs, the conditional law of the CLE in $D_{t,s}$ is the conformal image of $P_L$ where $L$ is the value such that $D_{t,s}$ and the four boundary points get mapped to the four corners of the rectangle. We will call these distributions $P_L$ the CLE with wired boundary conditions on the two vertical boundaries of the rectangle, and $P_{D, a_1, a_2, a_3, a_4}$ will denote the law of the configuration in a simply connected domain with four distinct marked boundary points $a_1, \ldots, a_4$.

\begin{figure}[ht!]
  \includegraphics[width=0.4\textwidth]{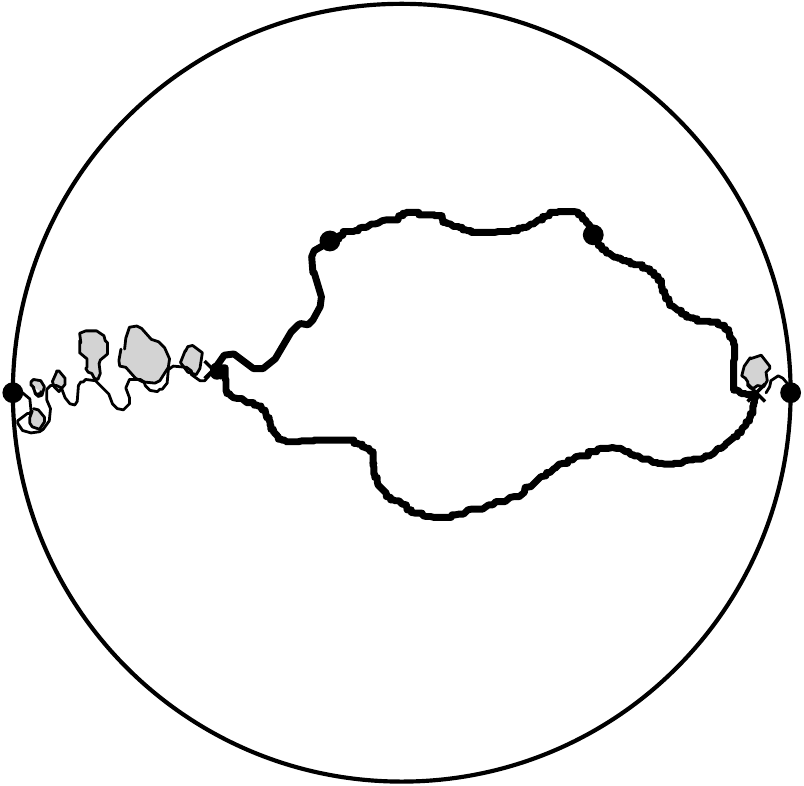}
   \quad
   \includegraphics[width=0.4\textwidth]{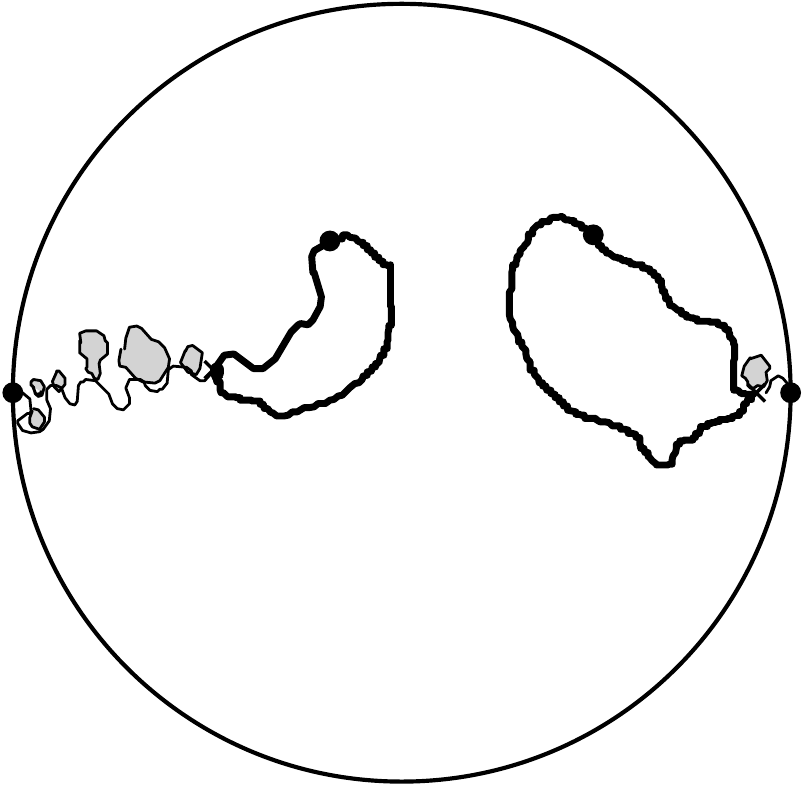}
  \caption{The two connection possibilities for the  partially explored simple CLE from Figure~\ref{pic23}.}
  \label{pic241}
\end{figure}

The law $P_L$ can be described in two steps: 
\begin{itemize}
\item One can first complete the strands that start from the four corners. This will complete the loop(s) (which turns out to be one loop or two loops, depending on how the strands
hook-up) that one had partially discovered, see Figure~\ref{pic241}. Note that we have however not (yet) described at this point how to sample them. 
\item Then, in the remaining domain (outside of the traced loops), one samples independent CLEs.
\end{itemize}

Hence, in {order} to fully describe $P_L$, it is in fact sufficient to describe the law of the strands. 
Let us already mention that Dub\'edat's commutation relation arguments (or bichordal SLE arguments)  that we will recall in the next section will do this to a certain extent.
They  for instance imply that once one knows the hook-up probabilities (i.e., the probability that the four strands hook-up in the way to create one loop, as a function of $L$, see Figure~\ref{pic241} for the two possible options), then one can deduce the joint law of the strands. The main purpose of this paper is actually to determine this hook-up probability.

\section {The special case $\kappa=4$ and the GFF} 
\label {SGFF}

The reader may have noticed that we have not yet discussed the definition of CLE$_4$ with two wired boundary arcs. 
Let us briefly do this in the present section, and show 
how Theorem~\ref{mainprop} can be derived directly and easily when $\kappa =4$, using the {relationship} between $\CLE_4$ and the GFF. 

When one defines $\CLE_4$ via an SLE$_4 (-2)$ branching tree, one necessarily has to use a symmetric side-swapping version (i.e., for $\beta=0$) -- see for example \cite{MSWCLEpercolation} and the references therein, so that the previous setup with SLE$_\kappa^{-1} (\kappa -6)$ processes 
does not apply. In the present section, we will only use the $\SLE_4^{0,0} (-2)$ explorations (i.e., with $\mu =0$ in the terminology of \cite{MSWCLEpercolation}) and refer to them simply as $\SLE_4 (-2)$ processes. 

Recall that $\SLE_4$ can be viewed as a level line of the Gaussian Free Field \cite{ss2010continuumcontour,DubedatGFF}. The corresponding coupling of 
$\CLE_4$ with the GFF  was introduced by Miller and Sheffield  (\cite{MScle}, see also \cite{ASW} for details) and can be described as follows: 
Sample one $\CLE_4$ in a simply connected domain $D$, and toss an independent fair Bernoulli coin $\epsilon_j \in \{ -1, 1 \}$ for each CLE loop $\gamma_j$. 
Then, in the domains encircled by each of these loops, sample an independent GFF $\Gamma_j$ with zero boundary conditions on $\gamma_j$ (the GFF $\Gamma_j$ is equal to $0$ 
in the outside of $\gamma_j$). Then, for {a certain} explicit choice of $\lambda >0$, the field 
\[ \Gamma := \sum_j ( 2 \eps_j \lambda + \Gamma_j ) \]
is exactly a GFF in $D$. Furthermore, the $\CLE_4$ and the labels $(\eps_j)$ are deterministic functions of the obtained field, and the side-swapping exploration of the CLE can be viewed as a deterministic function of the CLE and of the labels $(\eps_j)$ (see \cite{MSWCLEpercolation,ASW} for details).

One way to describe the joint exploration of a CLE$_4$ from two distinct starting points in such a way that it does not provide a hook-up bias due to the 
information about the orientation of the partially explored loops (when one then defines the CLE$_4$ with two wired boundary arcs) goes as follows: 
We consider that the two symmetric side-swapping $\SLE_4 (-2)$ 
explorations that are conditionally independent given the $\CLE_4$. 
This can be achieved by using two independent i.i.d.\ collections $(\eps_j)$ and $(\eps_j')$
 and to view the $\SLE_4^0 (-2)$ processes as deterministic functions of the two corresponding GFFs (see \cite{MSWCLEpercolation} and the references therein). 
 One can then stop these two explorations along the way as described above.  The discussion below will in fact show that the conditional law of the $\CLE_4$ 
 in the remaining domain $D_{t,s}$ is conformally invariant, and this is what we can then call the CLE$_4$ with two wired boundary arcs.

We can also assume that we have chosen to stop our explorations at times $t$ and $s$,  in such a way that 
$(D_{t,s}, w_t, o_t, w_s', o_s')$ is a conformal square (we can for instance do this by first stopping the first exploration at 
some deterministic time, and then stop the second one at the first time $s$ at which the configuration is a conformal square, and  to restrict ourselves to the case 
where the two explorations are disjoint). 
Let us denote by $E_1$ the event that the four strands are then hooked up so that they 
create a single $\CLE_4$ loop (of the original $\CLE_4$), and by $E_2$ the event that the four boundary strands are hooked up in the way that will create two disjoint $\CLE_4$ loops. 

Let us couple the $\CLE_4$ with a third GFF $\Gamma$ as above, by using yet another independent collection $(\bar \eps_j)$ of labels
(so, the three collections $(\bar \eps_j)$,  $(\eps_j)$ and $(\eps_j')$ are independent, conditionally on the CLE). 
On top of the partial discovery of the $\CLE_4$, we can also discover the corresponding boundary values of $\Gamma$.  
In other words, we can also discover whether on the two wired arcs, 
the GFF boundary values are $+ 2 \lambda$ or $-2 \lambda$. Let $\tilde E$ denote the event that these boundary values are the same on both arcs.

Now we can note that $\P[ \tilde E | E_1 ] =1$ and $\P[ \tilde E | E_2 ] = 1/2$ because of the rules that determine the GFF given the CLE {(i.e., the coin flips are i.i.d.\ fair Bernoulli)}. 
On the other hand, it is known \cite{MScle,ASW} that the $\CLE_4$ is a deterministic function of the GFF, so that conditioning on the joint information of the CLE and the GFF is the same as conditioning on the GFF only. 

But, on the event $\tilde E$ where the two boundary values on the partially explored strand are equal, we are looking at a GFF in a conformal square with boundary conditions $2 \lambda, 0, 2 \lambda, 0$ or $-2 \lambda, 0, -2 \lambda, 0$ on the four arcs. Hence, by symmetry, 
$$ \P[ E_1 | \tilde E ] = \P[ E_2 | \tilde E ] = 1/2.$$
This implies that  
$$ \P[ E_1] = \P[ \tilde E  \cap  E_1 ] = \P[ \tilde E ] \P[ E_1  | \tilde E ] = \P [ \tilde E ] \P[ E_2 | \tilde E ] = \P[ \tilde E \cap E_2 ] =\frac { \P[E_2] } 2 = \frac {1- \P[E_1]} 2 $$
and that the hook-up probability $\P[E_1]$ in the conformal square is indeed $1/3$ (i.e., $\theta = 2$).  

Note that in this special case, we see that the marginal distributions of the four strands are in fact ordinary $\SLE_4 ( \rho_1, \rho_2)$ processes (as opposed to the cases where $\kappa \not= 4$, where they turn out to be intermediate SLEs). This is also related to the fact that the hook-up probabilities that we will discuss in the next section take a very simple form in that case (which was already observed in the aforementioned papers by Dub\'edat or Bauer-Bernard-Kyt\"ol\"a).

\section{Consequences of Dub\'edat's commutation relations} 
\label{S2}

Let us consider again the general case $\kappa \in (8/3, 8)$, and let us  
now review what the results on commutation relations {from} \cite{Dubedat,Dubedat2,Dubedat3,Zhan1,Zhan2,Zhan3,BBK} imply for our CLEs with two wired boundary arcs and for the hook-up probability as a function of the aspect-ratio 
of the considered conformal rectangle.
It will be somewhat more handy to work in the upper half-plane instead of the rectangle 
(i.e., to first map conformally the rectangle $(0,L) \times (0,1)$ onto the upper half-plane via a Schwarz-Christoffel transformation that maps the two vertical sides onto $(- \infty, 0)$ and $(1-x, 1)$). 
Let us define $H(x)= H_\kappa (x)$ for $x \in (0,1)$ to be the probability that for a $\CLE_\kappa$ in the upper half-plane $\HH$ with two wired boundary arcs on $(-\infty, 0)$ and $(1-x, 1)$, the two wired arcs are joined in such a way that they form one single loop (in other words, the strand starting from $0$ ends at $1-x$). { In the sequel, we will 
use this cross-ratio $x$   of $(\infty, 0, 1- x, 1)$ in $\HH$ instead of the aspect ratio $L$ of the (conformal) rectangle. 
Recall that the cross-ratio $x(L)$ and $L$ are  related to $L$ by the formula 
$$ L = \frac { F(1/2, 1/2, 1 ; 1-k^2 ) } { 2 F (1/2, 1/2, 1 ; k^2)}, $$  
where $ x = (1-k)^2 / (1+k)^2$ and $F$ denotes the hypergeometric function  $_2F_1$ (we will keep this notation throughout the paper -- see Appendix~\ref{App2} for the definition and the properties of these functions 
that we will use in this paper)}. More generally, we will refer to $x$ as the cross-ratio of $(D, x_1, x_2, x_3, x_4)$ if one can map conformally this configuration 
onto $(\HH, \infty, 0, 1-x, 1)$.

{In the $\CLE_\kappa$ setup that we consider here (and in all three cases $\kappa \in (8/3, 4)$, $\kappa = 4$ and $\kappa \in (4,8)$,} for each choice of a simply connected domain {$D$} with four distinct boundary points $a_1, \ldots, a_4 {\in \partial D}$ ordered {counterclockwise},  { we have argued in the previous sections that the distributions $P_{D, a_1, \ldots, a_4}$ that we defined provide a distribution on pairs of SLE paths that join these four boundary points with the following properties: }
\begin {itemize}
 \item They are conformally invariant.  {That is,} if $\Phi$ is a conformal transformation, then the law of the image of $P_{D, a_1, \ldots, a_4}$ under $\Phi$ is 
 $P_{\Phi (D), \Phi(a_1), \ldots, \Phi (a_4)}$. In particular,  the probability that $a_1$ hooks up with $a_4$ is in fact a function $H(x)$ of the cross-ratio $x$ 
 of  $(D, a_1, a_2. a_3, a_4)$.
 \item For any given $a_i$, if one discovers the {entire} strand $\gamma_i$ that emanates from $a_i$ (and therefore its endpoint $a_j$), then the conditional distribution of the other remaining strand is just an $\SLE_\kappa$ joining the two remaining marked point in $D \setminus \gamma_i$  { (this is due to the fact that after 
 completing one strand, one is left with only one CLE exploration).} 
\end {itemize}
Hence, if one conditions $P_{D, a_1,\ldots , a_4}$ on the event that $a_1$ is connected to $a_2$ (and therefore that $a_3$ is connected to $a_4$), one gets a distribution on pairs of paths $(\gamma, \gamma')$ joining $a_1$ to $a_2$ and $a_3$ to $a_4$ respectively, such that (i) conditionally on $\gamma$, the law of $\gamma'$ is that of $\SLE_\kappa$ in $D \setminus \gamma$, and (ii) conditionally on $\gamma'$, the law of $\gamma$ is that of $\SLE_\kappa$ in $D \setminus \gamma'$. 
It is possible to see (and this has been done using several methods) that the law on pairs $(\gamma, \gamma')$ is uniquely characterized by this last property (this is the resampling 
property of bichordal SLE as studied and used in \cite{IG1,IG2,IG3,IG4}). 

This explains why $P_{D, a_1,  \ldots , a_4}$ is fully determined once one knows the hook-up probability function $H(x)$. In fact, it suffices to know the value of $H(x_0)$ 
for one single value $x_0 \in (0,1)$ in order to deduce the entire function $H$. Indeed, if one knows $P_{D, a_1, \ldots, a_4}$ for one given choice of $(D, a_1, \ldots, a_4)$, then {$H$ is determined because the} hook-up probability {evolves as a} martingale when one lets one strand evolve, we know the law of the evolution of this strand, and the boundary values at $0$ and $1$.

Considerations of this type are in fact included in some form in the papers cited above that introduce and study commutation relations for SLE paths and their consequences (note that these in fact 
actually study a {somewhat} more general class of questions -- in the present setup, we for instance already know from the construction 
that our commuting strands will eventually hook-up and create one or two loops, which is a non-trivial feature). 
Then, it follows from these arguments that $H$ is of the form 
\begin {equation}
 \label {HvsZ}
H(x) =  \frac {Z (x)} { Z (x) +    \theta Z (1-x) } 
\end {equation}
for some positive $\theta$, where 
\begin {equation} Z (x):=  x^{2 / \kappa} (1-x)^{1- 6 /\kappa} f(x) , \end {equation}
and where here and in the remainder of this paper, $f$ will denote the hypergeometric function
\begin {equation}
 f(x) := F \left( \frac 4 \kappa, 1- \frac 4 \kappa  , \frac  8 \kappa ; x \right) \end {equation}
(see for instance  \cite[Section~4]{Dubedat3} or \cite[Section~8]{BBK} about ``4-SLE''). 
Recall (see the short Appendix~\ref{App2} where we will briefly recall basics about hypergeometric functions) 
that $f(0)=1$ and note that since $ 8/\kappa - ( 4/\kappa + 1 - 4 / \kappa ) = 8/ \kappa  - 1 > 0$, this function $f$ is continuous at $1$ with
$$ f(1) = \frac { \Gamma (8/\kappa) \Gamma (8/\kappa -1 )} {\Gamma (4/\kappa) \Gamma (12/\kappa -1 )}.$$  

Another way to phrase/interpret this in the previous setup is that the above-mentioned papers describe the law of the 
bichordal SLEs, which are the 
conditional distributions of $P_{D, a_1, \ldots, a_4}$ given one hook-up event (or given the other one), but not the actual probability of 
these hook-up events. 
In other words, in order to determine the function $H$, it only remains to identity the value of $\theta$ in terms of $\kappa$. 
Note that knowing the value of $\theta$ is equivalent to knowing  the connection probability {for a} conformal square (i.e., for $x= 1/2$) (as 
$H(1/2) =  1 / (1 + \theta)$).

Note that that as $x \to 0$, 
$$ Z (x) \sim x^{2 / \kappa }  \quad\hbox{and}\quad
Z (1-x)  \sim x^{1 - 6 / \kappa } f(1). $$ 
Hence, because $1 - 6 / \kappa < 2 / \kappa$, it follows that 
$ H (x) \sim  x^{8 / \kappa - 1} / (\theta f(1)) $
as $x \to 0$. In other words,  
\begin {equation}
\label {probab0}
  \theta^{-1} = f(1)  \lim_{ x \to 0}  \left( x^{1- 8/ \kappa}  H(x) \right) .
\end {equation}
The strategy of our proof of Theorem~\ref{mainprop} i.e., of the fact that $\theta = -2 \cos (4 \pi / \kappa)$, will be to determine the right-hand side of \eqref{probab0}, which therefore also gives the value of $\theta$. In other words, we will in fact estimate precisely the asymptotics of the hook-up probability in very thin conformal rectangles. 
We have just argued that $H(x)$ decays like some constant times $x^{ 8 / \kappa-1}$ as $x \to 0$, and our goal will be to determine the value of this constant.

\section {First main steps of the proof of Theorem~\ref{mainprop}} 

\label {Smain}
 
We now describe the main steps of the proof of Theorem~\ref{mainprop} in the cases $\kappa \in (4,8)$ and $\kappa \in (8/3, 4]$  separately.
We will defer the  proofs of two more computational lemmas to the next sections, in order to highlight here the arguments that reduce the proof 
of Theorem~\ref{mainprop} to these concrete {computations involving SLE and Bessel processes.}

\subsection {Case of non-simple CLEs}

Let us start with the case of $\CLE_\kappa$ when $\kappa \in (4,8)$. 
Consider a conditioned
$\CLE_\kappa$ in the upper half-plane with only one wired boundary arc on $\R_-$. Recall that the law of this conditioned $\CLE_\kappa$ can be sampled from using the following steps:
\begin{itemize}
	\item Sample an $\SLE_\kappa$ $\gamma$ from $0$ to $\infty$ in order to complete the partially discovered loop that runs on $\R_-$, and then
	\item Sample independent $\CLE_\kappa$'s in the remaining connected components that are ``outside'' of this loop. 
\end{itemize}

\begin{figure}[ht!]
  \includegraphics[scale=0.5]{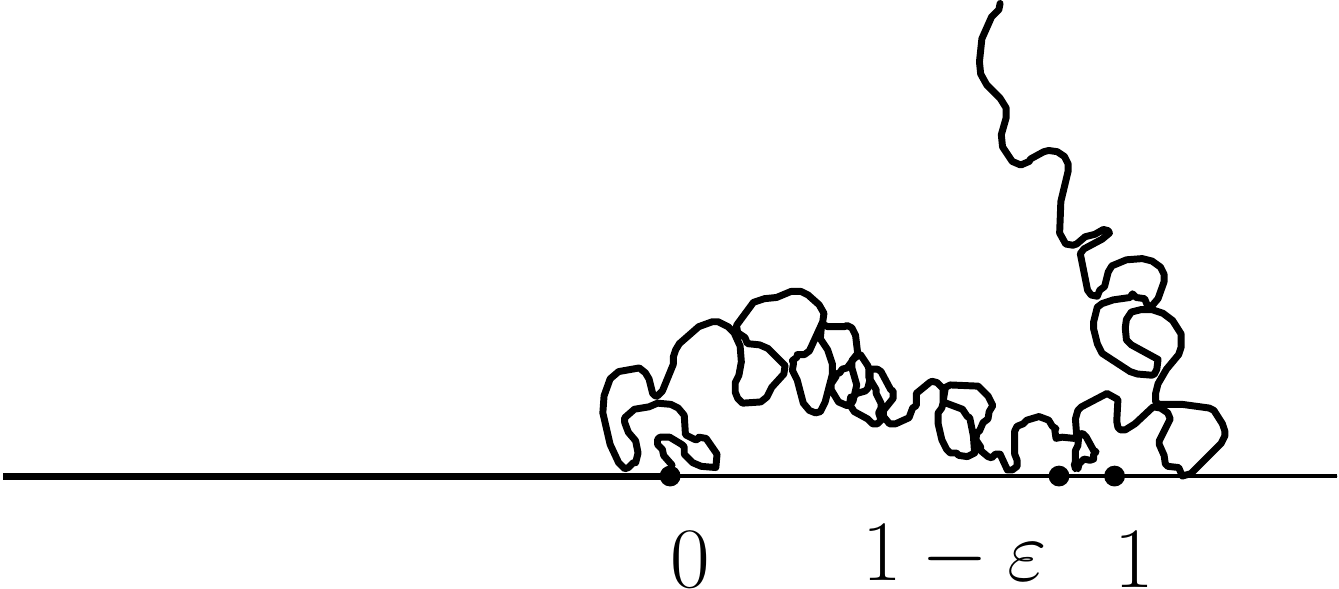}
  \caption{Sketch of the event $D(\eps)$}
  \label{pic5}
\end{figure}

We now fix $\eps > 0$ very small and define the event $D( \eps)$ that $\gamma \cap [1- \eps, 1  ] \not= \emptyset $ (see Figure~\ref{pic5}).  The probability of  $D( \eps)$ can be explicitly computed (it is in fact 
the formula that was already used by Schramm \cite{Schramm0} in his argument mentioned at the beginning of the introduction); it is
a generalization of Cardy's formula for $\SLE_\kappa$ almost identical to that  determined in \cite{LSW1} -- see for instance \cite[Lemma~6.6]{RS} or \cite[Section~3]{Wlh}):
$$
\P[ D (\eps) ] =  \frac { \int_{1 / \eps}^\infty   y^{-4 / \kappa} (1+y)^{- 4 / \kappa}  dy} { \int_{0}^\infty   y^{-4 / \kappa} (1+y)^{ -4 / \kappa} dy } .
$$
Clearly, as $\eps \to 0$,
\begin {equation}
\label {probab1}
\P[ D (\eps) ] 
\sim  \frac { \eps^{8/\kappa - 1 }} { (8/ \kappa -1 ) \int_{0}^\infty   y^{-4 / \kappa} (1+y)^{ -4 / \kappa} dy }
\sim \frac { \Gamma ( 4 / \kappa)}{  \Gamma (1- 4/ \kappa) \Gamma (8/ \kappa ) }  \eps^{8/\kappa -1 }.
\end {equation}

The idea is now to evaluate the asymptotic behavior of $\P[D ( \eps) ]$ as $\eps \to 0$ using a different two-step procedure that will involve hook-up probabilities: 
 {Let us move along the segment $[1- \eps , 1 ]$ from left to right, and each time we meet a loop of the conditioned $\CLE_\kappa$ for the first time, we trace it in 
the clockwise direction before continuing to move along the horizontal segment. This defines a continuous path $w_t$ starting from $1-\eps$ and ending at $1$ or at $0$ (if the curve $\gamma$ hits $[1-\eps, 1]$, then $w$ will trace $\gamma$ backwards). The last point $o_t$ on $[1- \eps, 1]$ that $w$ did visit before time $t$ corresponds to the beginning of the loop that is being traced at time $t$. }
We stop this path at the first time $\tau$ (if it exists) at which 
 the cross-ratio $c_\tau$ corresponding to $(\infty, 0, w_\tau , o_\tau)$ in the unbounded connected component of the complement of $w[0,\tau]$ in the upper half-plane reaches $x= \eps^{7/8}$. We call $B( \eps)$ the event that such a time exists (see Figure~\ref{picB}). 
\begin{figure}[ht!]
\null \vskip 1cm
  \includegraphics[scale=0.6]{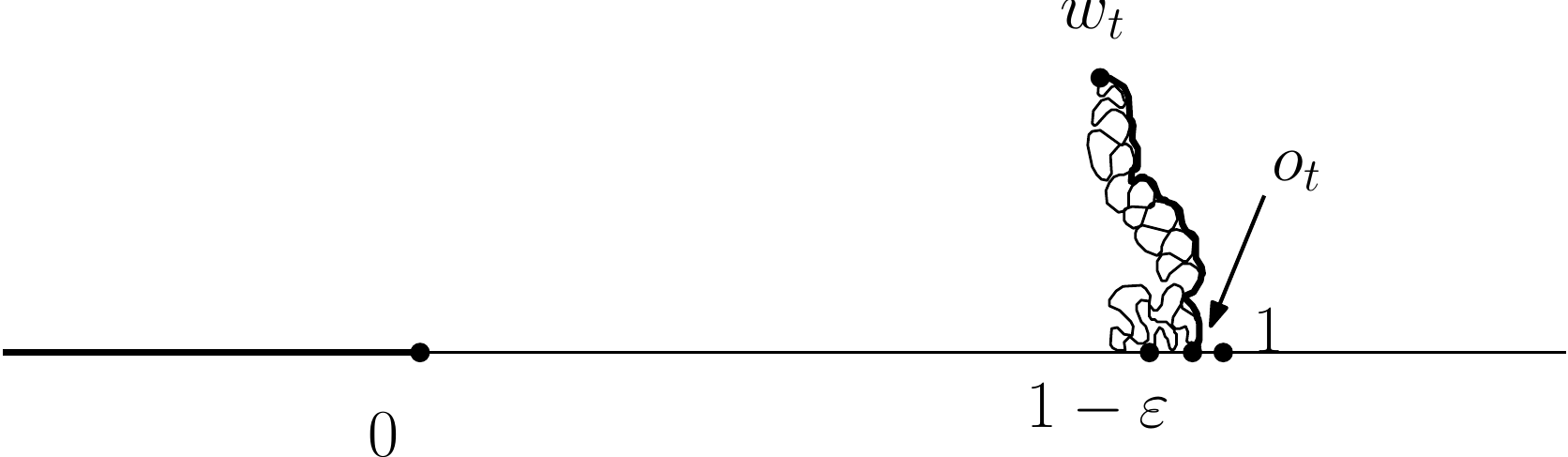}
  \caption{Sketch of the event $B(\eps)$}
  \label{picB}
\end{figure}

The following lemma will enable us to relate the asymptotic behavior of $H(x)$ as $x \to 0$ to that of $\P[B(\eps)]$:  
\begin {lemma}
 \label {main0}
One has $D ( \eps ) \subset B(\eps)$. Furthermore,  if $D'( \eps)$ denotes the event that the partially explored loop at time $\tau$ (in the definition of $B(\eps)$) does in fact correspond to a portion of $\gamma$, then the conditional probability of $D'(\eps)$ given $D(\eps)$  {tends to $1$ as $\epsilon \to 0$}.
\end {lemma}

\begin {proof}
Recall that we are working with a $\CLE_\kappa$ in the upper half-plane with wired boundary conditions on $\R_-$, which consists of an $\SLE_\kappa$ that we will denote by $\gamma$ and a family of further loops to the right of it.  When one explores the $\CLE_\kappa$ loops of such a wired $\CLE_\kappa$ that touch $[1-\eps, 1 ]$ starting from $1-\eps$, and tracing them in the clockwise direction one after the other, then in the configuration where $\gamma$ intersects this segment (i.e., when $D(\eps)$ holds), at some point, one has to trace an arc of the loop that $\gamma$ is part of, and that connects a point in $[1-\eps, 1]$ to a point that lies in $\R_-$, as depicted in Figure~\ref{Picsix1}. Just before this time, the cross-ratio $c_t$ tends to $1$ because $w_t$ approaches {$\R_-$}, which implies that it did reach ${\eps}^{7/8}$ beforehand. Hence, $D(\eps)$ is indeed a subset of $B(\eps)$. 
 \begin{figure}[ht!]
  \includegraphics[scale=0.6]{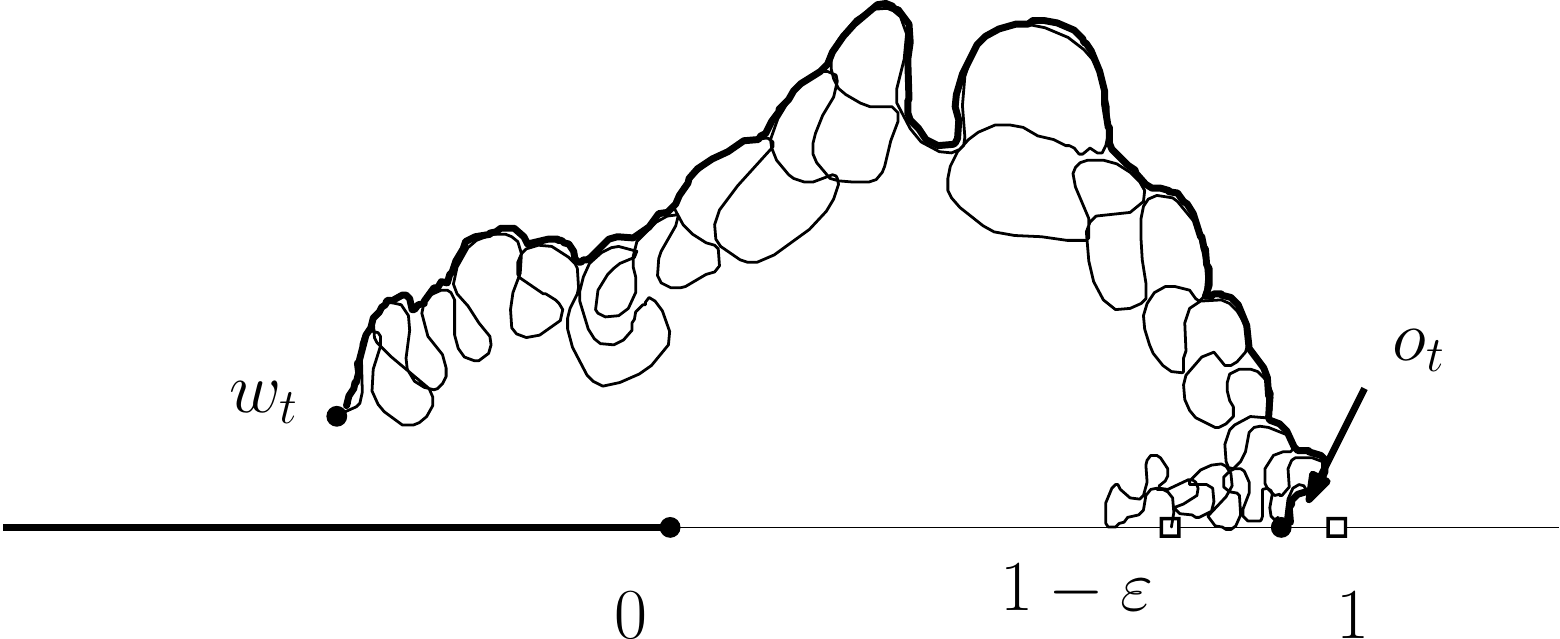}
  \caption{The two bold boundary parts get very close}
  \label{Picsix1}
\end{figure}

  Suppose now that we are in the case where $B(\eps)$ holds but not $D' ( \eps)$.
  Then, at the time $\sigma$ at which one has completed the loop that one was tracing at time $\tau$, the conditional distribution in the remaining to be explored domain will be again a $\CLE_\kappa$ with just one wired boundary arc on $\R_-$. The conditional probability that $D(\eps)$ still holds will therefore be smaller than the unconditional probability that $D(\eps)$ holds because the cross-ratio corresponding to  
  the four points $(o_\sigma , 1, \infty, 0)$ at that time is necessarily smaller than $\eps$ (see Figure~\ref{Picsix2}).
In other words,
 $$ \P[ D ( \eps) \setminus D' (\eps ) ] \le  \P[ B (\eps ) ] \times \P[ D (\eps ) ].$$ 
   \begin{figure}[ht!]
  \includegraphics[scale=0.6]{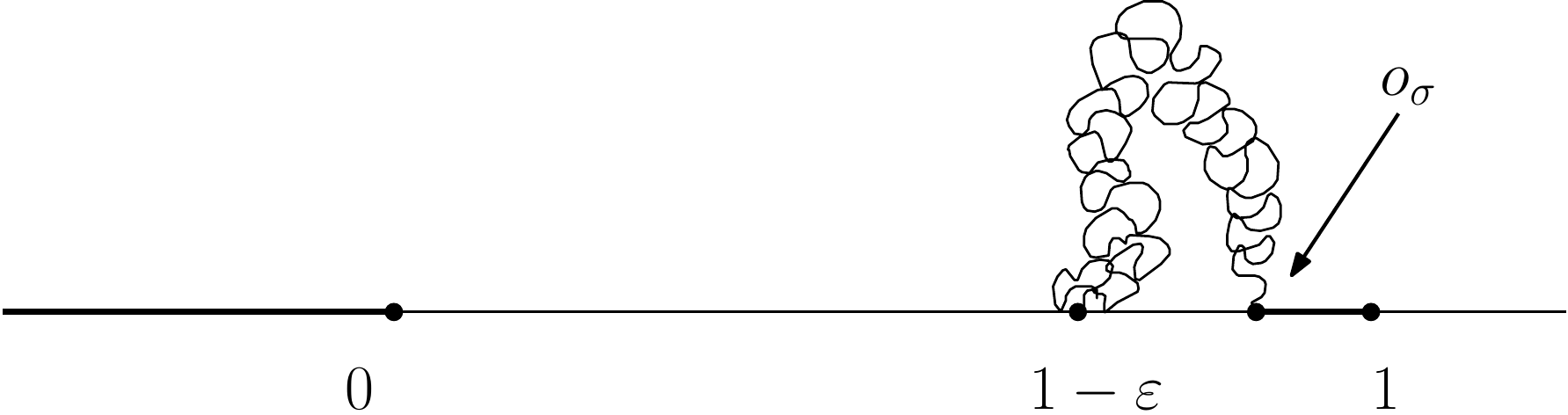}
  \caption{At such a time, the conditional probability that $D(\eps)$ holds is smaller than the unconditional probability.}
  \label{Picsix2}
\end{figure}
 It finally remains to note that $\P[B (\eps ) ] \to 0$ as $\eps \to 0$, {as a consequence of the bound} $\P[ B(\eps)] \times H(\eps^{7/8}) \le \P[ D( \eps) ]$ and our previous estimates {of} $H$ and $\P[D(\eps)]$.  Hence, the conditional probability of $D'(\eps)$ given $D(\eps)$ tends to $1$ as $\eps \to 0$, which concludes the proof.
 \end {proof}

The previous lemma implies in particular that
 \begin {equation}
  \label {probab2} 
  \frac {\P[D(\eps)]}{\P[B(\eps)]} = 
 \P[ D(\eps) | B (\eps )  ] \sim \P [D' (\eps) | B (\eps) ] = H (\eps^{7/8})
 \end {equation}
 as $\eps \to 0$. 
The proof of Theorem~\ref{mainprop} for $\kappa \in (4,8)$ will then be complete if we prove the following estimate: 
\begin {lemma}
\label {main1}
As $\eps \to 0$, 
$$ \P[ B( \eps)] \sim \frac { \Gamma ( 4 / \kappa) } { \Gamma (2- 8/\kappa) \Gamma (12/\kappa - 1)} \times (\eps^{1/8})^{8 / \kappa - 1}.$$
\end {lemma}
Indeed, combining this lemma with (\ref {probab1}) and (\ref {probab2}) shows that as $x = {\eps}^{7/8} \to 0$,  
$$ H(x) \sim \frac {\P[ D( \eps)] }{\P[B(\eps)] } \sim \frac {\eps^{8 / \kappa -1}}{(\eps^{1/8})^{8/\kappa -1}} \times \frac {\Gamma (2- 8/ \kappa) \Gamma (12/\kappa -1) }{ \Gamma (1- 4/\kappa) \Gamma (8 / \kappa ) } \sim  \frac {\Gamma (2- 8/ \kappa) \Gamma (12/\kappa -1)  x^{8/\kappa -1 }}{ \Gamma (1- 4/\kappa) \Gamma (8 / \kappa ) } .  $$
Combining this with \eqref{probab0}, we see that: 
$$
 \theta^{-1} =  f(1) \frac {\Gamma (2- 8/ \kappa) \Gamma (12/\kappa -1) }{ \Gamma (1- 4/\kappa) \Gamma (8 / \kappa   ) } = 
  \frac {\Gamma (2- 8 / \kappa) \Gamma (8/\kappa  -1)  }  {\Gamma (1 - 4/\kappa )  \Gamma (4 / \kappa) } = 
\frac { \sin ( 4 \pi / \kappa) }{  \sin (\pi ( 8 /  \kappa -1)) } =  \frac {-1} {2 \cos (4 \pi / \kappa)} 
$$
(recalling that $\Gamma (1-z) \Gamma (z) = \pi  / \sin ( \pi z)$). 

The proof of Lemma~\ref{main1} will be presented in Section~\ref{Sproof1}. 

\subsection {Case of simple CLEs}
In the case where $\kappa \in (8/3, 4)$, we are also going to estimate the asymptotic behavior of $H(x)$ as $x \to 0$, but we need a somewhat different strategy because $\SLE_\kappa$ 
paths do not 
hit boundary intervals anymore. The similarity with the case {$\kappa \in (4,8)$} is that we will again estimate the asymptotic behavior of $H(\eps)$ as $\eps \to 0$ by estimating the asymptotic behavior of 
the probability of another event $C(\eps)$, for which we show that $\P[C(\eps)] \sim H(\eps)$. 

We  consider a $\CLE_\kappa$ in the upper half-plane, with boundary conditions that are respectively wired, free, wired and free on $\R_-$, $[0,1- \eps]$, $[ 1-\eps, 1]$ and $[1,\infty)$.
So, we have four strands starting at $\infty$, $0$, $1- \eps$ and $1$, and $H(\eps)$ is the probability that the strand starting from $0$ hooks up with the one starting from $1- \eps$. 

Let us consider an $\SLE_\kappa$ path from $0$ to $\infty$ in the upper half-plane, and an independent 
one-sided restriction measure of exponent $\alpha = (6- \kappa) / (2\kappa)$ attached to the segment  $[1- \eps, 1]$ in the upper half-plane. Let us define  $C(\eps)$ 
to be the event that the $\SLE_\kappa$ intersects this restriction sample (see Figure~\ref{picC}).

\begin{figure}[ht!]
  \includegraphics[scale=0.6]{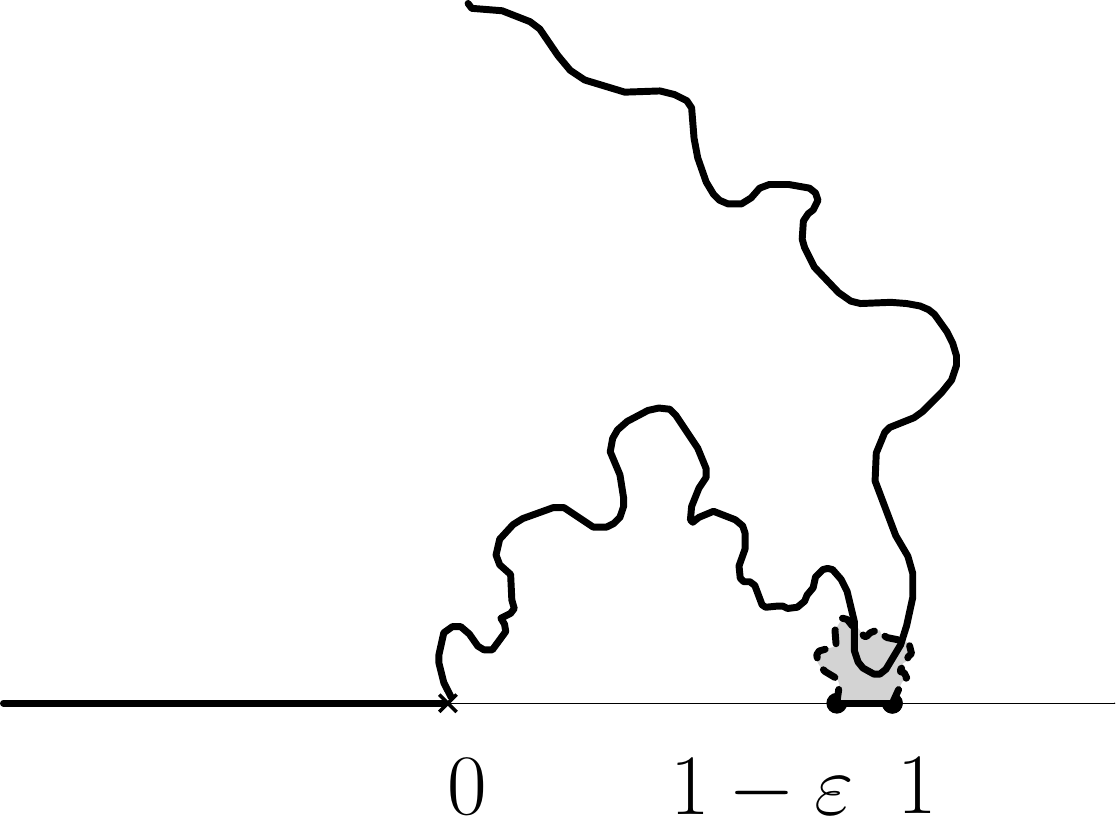}
  \caption{Sketch of the event $C(\eps)$}
  \label{picC}
\end{figure}

The first step in our proof is the following: 
\begin {lemma}
\label {main2}
{As} $\eps$ tends to $0$, $H(\eps) \sim \P[ C ( \eps) ] $. 
\end {lemma}

 { 
\begin {proof}
Let us consider a $\CLE_\kappa$ ${\mathcal C}$ for $\kappa \in (8/3, 4)$ 
in the unit disk (with no wired boundary arc) that has been obtained as {the} outermost outer boundaries of the clusters of 
a Brownian loop-soup ${\mathcal L}$. We will consider two $\SLE_\kappa (\kappa -6)$ explorations of this CLE as in Section~\ref{ss2ba}. 
The first exploration in an $\SLE_\kappa^{-1} (\kappa -6)$ that starts from $-1$ and targets $1$, and the second one
in an $\SLE_\kappa^1 (\kappa-6)$ that starts from $1$ and targets $-1$. 

Let us fix some large constant $m$ and let $\eta= m \sqrt {\eps}$. 
We start the first Markovian exploration near $-1$ targeting $1$ and we stop it at the first time $t$ at which the exploration reaches distance $\eta$ from $-1$. 
We can note that (by conformal invariance of the Markovian exploration and by a simple distortion estimate), provided that $m$ has been chosen large enough, 
then for all small $\eps$, 
the probability that the harmonic measure of $\partial_t$ in $D_t$ as seen from $0$ is greater than $\sqrt {\eps}$ is at least $3/4$.      

After having sampled this first exploration up to time $t$, we discover the second 
exploration from the opposite boundary point $1$. Let us define $s_1$ to be the first time at which 
this second exploration exits the $\eta$-neighborhood of $1$, $s_2$ to be that first time at which the cross-ratio of the corresponding marked points in $D_{t, s_2}$ is equal to $\eps$, and 
finally $s = \min (s_1, s_2)$. We can again note that (provided $m$ has been chosen large enough, and for all small enough $\eps$), 
with probability at least $1/2$, the cross-ratio of four marked points in  $D_{t, s_1}$ is greater than $\eps$. Hence, the probability of  $E(\eps) 
:= \{ s = s_2 \}$ is greater than $1/2$ for all small $\eps$. 
Our definitions of $\CLE_\kappa$ with two wired arcs and of the function $H$ say that given $E (\eps)$,
the conditional probability of the event $U$ that the four strands corresponding to the definition of $D_{t,s}$ hook-up to form one single loop is $H(\eps)$. 

Our goal is now to estimate this conditional probability $\P[ U | E (\eps) ]$ in another way. 
First, as we have recalled above, note that conditionally on the first exploration up to $t$, 
the conditional distribution of the Brownian loop-soup in the remaining domain $D_{t}$ (with $1$ on its boundary) 
can be decomposed as follows (see \cite{Qian2}): The Brownian loops that touch $\partial_t$ on the one hand (and the union of these loops form a restriction 
sample ${\mathcal R}_{t}$ of exponent $\alpha = (6 - \kappa )/(2 \kappa)$ attached to $\partial_{t}$ in $D_t$) and the 
other Brownian loops in $D_t$ on the other hand (that form an independent Brownian loop-soup in $D_t$ -- let us call ${\mathcal C}_t$ the corresponding CLE$_\kappa$ in 
$D_t$ obtained via the loop-soup clusters of that loop-soup). 

Finally, let us denote $\tau$ the first time after $t$ at which the first exploration completes the loop $L$ that it is tracing at time $t$, and let $D_\tau$ and ${\mathcal C}_\tau$ denote the corresponding domain and CLE. 
We can note that we are in the framework of  the CLE-exploration-restriction property explained in Figure~\ref {explorestriction}: Conditionally on $D_t$, one considers the CLE ${\mathcal C}_t$ in $D_t$, and the independent chord created by the restriction sample ${\mathcal R}_t$, that defines the subset $D_t \setminus {\mathcal R}_t$ of $D_t$. 
Then one obtains $D_\tau$ by considering the complement 
in $D_t$ of the union of ${\mathcal R}_t$ with the CLE loops it intersects (and one keeps only the connected component that has $1$ on its boundary).
The loops of ${\mathcal C}_\tau$ are then 
exactly the loops of ${\mathcal C}_t$ that stay in $D_\tau$. 
We can note that (by definition) the second exploration is defined to be an SLE$_\kappa (\kappa -6)$ exploration of $D_\tau$ (creating the loops of ${\mathcal C}_\tau$ 
until the first time at which it hits $L$ (which is part of the 
boundary of $D_\tau$). After the time at which it hits $L$, it starts tracing that loop $L$ and will at some time $\tilde \sigma$ then hits ${\mathcal R}_t$.  
Hence, the exploration-restriction property applied to $D_t$ and $D_t \setminus {\mathcal R}_t$ states that up to $\tilde \sigma$, the law of this second exploration does coincide exactly with that of an SLE$_\kappa (\kappa -6)$ in $D_t$ (associated with the CLE ${\mathcal C}_t$).

Hence, on the event where $s_1 < \tilde \sigma$, the second exploration is exactly the same as the exploration of ${\mathcal C}_t$ up to the same time. 
Here, we can note that when $\sigma \le s_1$, then necessarily, there exists a Brownian loop in
the original loop-soup in the unit disc that intersects both the $\eta$-neighborhood of $1$ and $-1$. 
The probability of this last event is easily shown to be bounded by a constant times $\eta^4$ as $\eta \to 0$ 
(because the mass of the Brownian loops that intersect both these neighborhoods behaves like $O (\eta^4)$).

Wrapping things up,
we see that we can couple the Brownian loops of ${\mathcal L}$ {in} $D_{t,s}$ with three (conditionally) independent pieces: 
\begin {itemize} 
 \item a Brownian loop-soup in $D_{t,s}$,
 \item the Brownian loops that touch $\partial_t$ and that form a restriction sample of exponent $\alpha$ attached to $\partial_t$ in $D_{t,s}$, and
 \item the Brownian loops that touch $\partial_s'$ and that form a restriction sample of exponent $\alpha$ attached to $\partial_s'$ in $D_{t,s}$
\end {itemize}
in such a way that the probability (in this coupling) that the union of these three different independent pieces does not coincide with the set of Brownian loops of ${\mathcal L}$ in $D_{t,s}$ is  $O (\eta^4)$.

But if we consider the loop-soup clusters formed by the union of these three independent pieces, the probability that $\partial_t$ and $\partial_s'$ 
are part of the boundary of the same loop-soup cluster is
equal to $\P[ C (\eps) ]$ when $s=s_2$. Indeed, the union of the first two pieces will form the $\SLE_\kappa$ and the third 
will form an independent restriction sample (see Figure~\ref{tworestr}). Hence, we can conclude that 
$$ H(\eps) \P[E]  = \P[  U \cap E ] =  \P[C (\eps) ] \P[E]   +   O(\eta^4).$$

\begin{figure}[ht!]
  \includegraphics[scale=0.7]{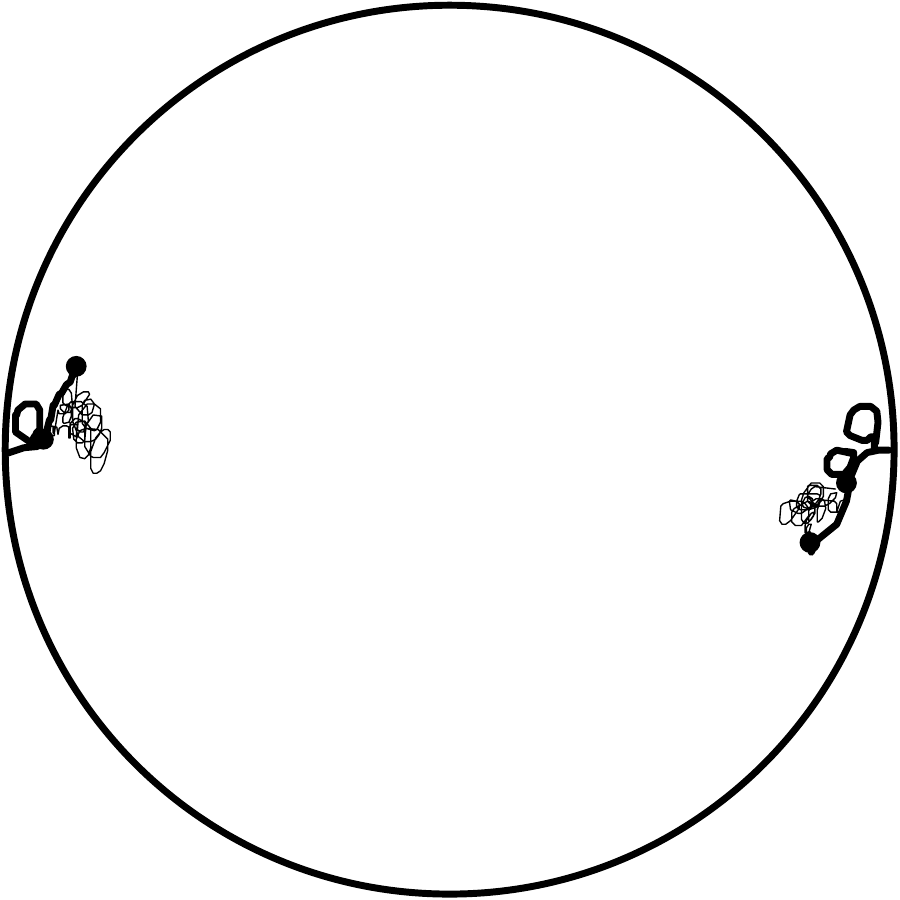}
   \hskip 1cm
   \includegraphics[scale=0.7]{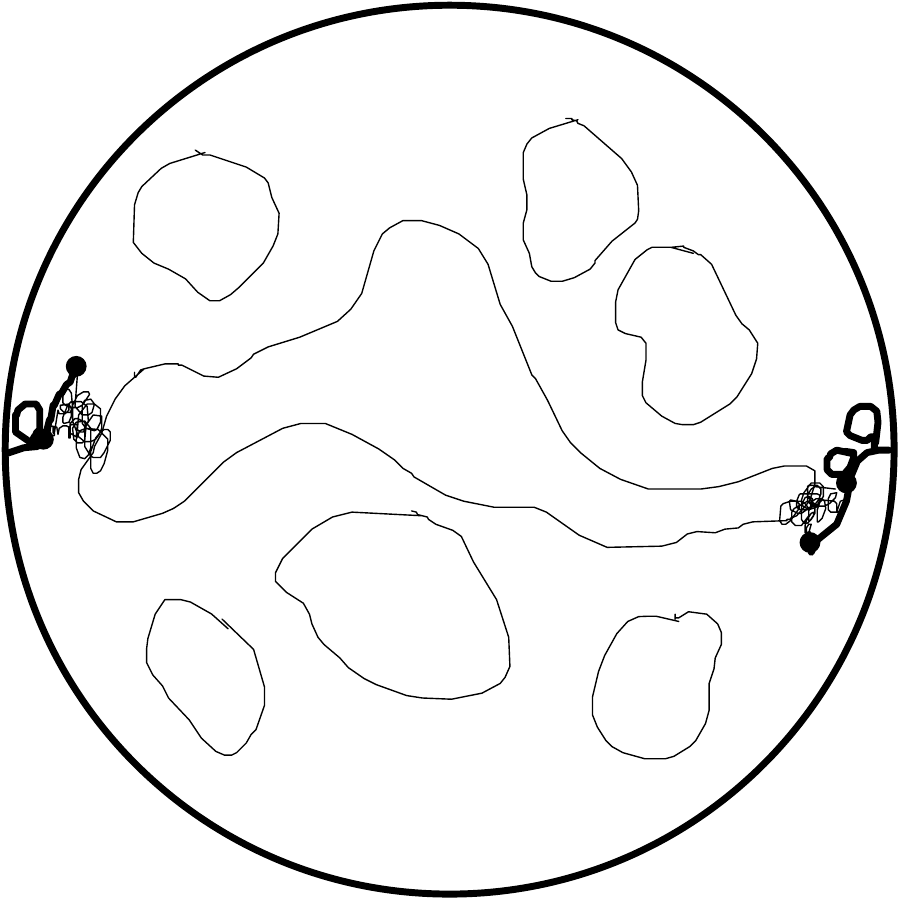}
  \caption{The two explorations with the Brownian loops that they discovered (left). The two wired boundaries are part of the same loop if a $\CLE_\kappa$ loops intersects both restriction samples (up to 
  a small probability).}
  \label{tworestr}
\end{figure}

Recall that $\P[E (\eps)] \ge 1/2$ for all small $\eps$. Furthermore, we know by (\ref {probab0}) that $H ( \eps) \eps^{ 1 - 8 / \kappa}$ tends to a positive 
constant as $\eps \to 0$. Since $\eta^4 = m^4 \eps^2 = o (\eps^{8 / \kappa -1 })$,   we can therefore finally conclude that $\P[ C ( \eps) ] \sim H (\eps) $ as $\eps \to 0$. 
\end{proof}

The proof of Theorem~\ref{mainprop} for $\kappa \in (8/3, 4)$ will then be complete if we establish the following estimate:}
\begin {lemma}
 \label {main3}
 {As} $\eps \to 0$,
$$ \P[ C(\eps) ]  \sim   \frac {\eps^{8/\kappa -1 }}{  f(1) \times (-2 \cos (4 \pi / \kappa))}.$$  
\end {lemma}
Indeed, combining this lemma with (\ref {probab0}) and Lemma~\ref{main2} then shows that $\theta = - 2 \cos (4 \pi / \kappa)$.  The proof of Lemma~\ref{main3}
will be explained in  Section~\ref{Sproof2}.

\section {Proof of Lemma~\ref{main1}} 
\label {Sproof1}
In the present section, we will prove Lemma~\ref{main1}. This will complete the proof of Theorem~\ref{mainprop} in the case where $\kappa \in (4,8)$.

Let us first derive another result that will be useful in our proof, and that deals with the usual $\CLE_\kappa$ (with no wired boundary part) for $\kappa \in (4,8)$.  
Note that $8/ \kappa \in (1,2)$ {for $\kappa \in (4,8)$} (we will implicitly and repeatedly use this fact in the following arguments). 
Let $(w_t)$ be an $\SLE_\kappa (\kappa -6)$ in $\HH$ starting from $w$, with initial marked point $o \ge w$, and {targeting $\infty$}. 
Denote by $g_t$ the usual Loewner map from the unbounded connected component of the complement of $w[0,t]$ in $\HH$ into $\HH$ such that $g_t ( z)  = z + o(1)$ as 
$z \to \infty$. 

Recall that the {law of the} driving function $W_t := g_t (w_t)$ {of} this Loewner chain can be {sampled from} using the following two steps: 
\begin {itemize}
 \item Sample a reflected Bessel process $X$ with dimension $d = 3- 8/ \kappa \in (1,2)$ started from $(o-w) / \sqrt {\kappa}$ (at the end of the day, the process $\sqrt {\kappa} X_t$ will be equal to $O_t-W_t$, {the difference between the force point and the driving process}).  
 \item Set 
 \[ O_t = o + \int_0^t \frac{2}{\sqrt {\kappa} X_s} ds \quad\text{and}\quad W_t = O_t - \sqrt {\kappa }X_t .\]
\end {itemize} 
{Note that the image under $g_t$ of the leftmost point $o_t$ on $[o, \infty)$ that has not been swallowed by the Loewner chain before time $t$ is equal to $O_t$.  For each $b > o$, let $T_b$ be the first time at which $b$ is swallowed by the Loewner chain (see Figure \ref{picLT} for the case $w=o$).  As the Bessel process dimension $d$ is {strictly} between $1$ and $2$, we have that $T_b < \infty$ almost surely.}

\begin{figure}[ht!]
  \includegraphics[scale=0.6]{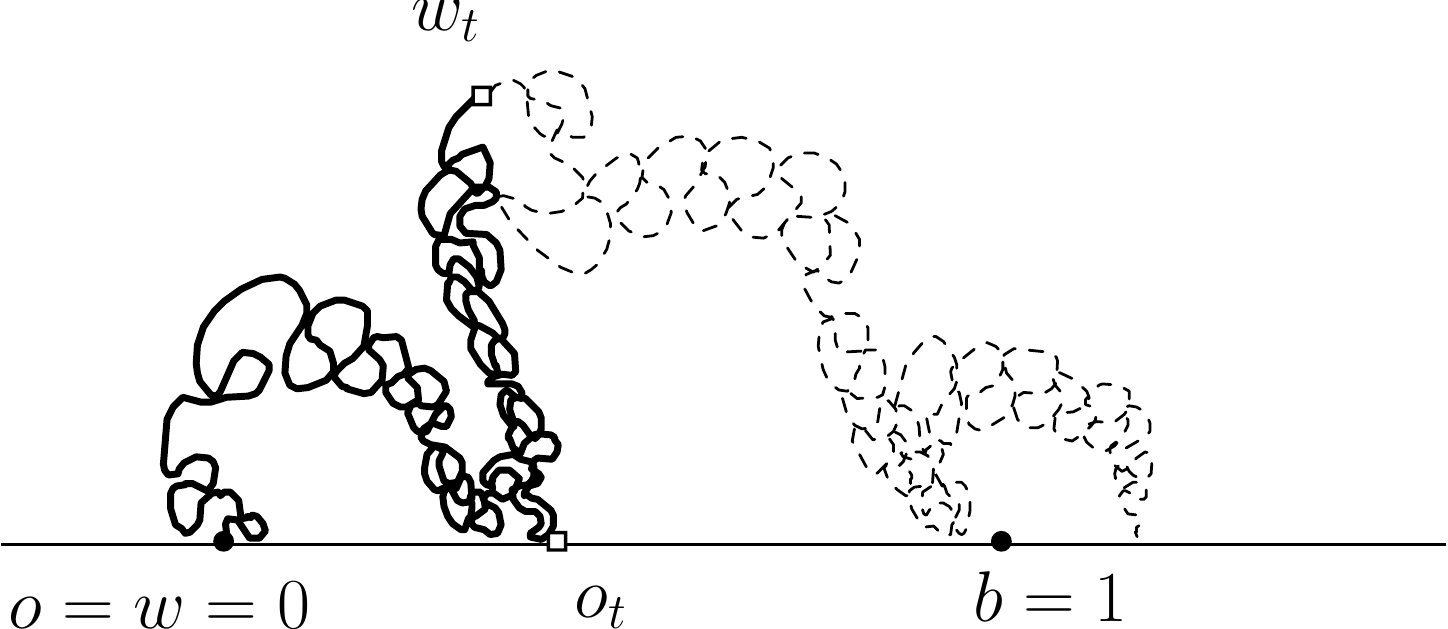}
  \caption{Sketch of the $\SLE_\kappa (\kappa -6)$ {for $w=o=0$}. Here, $b=1$, $t < T_b$ and the dashed curve corresponds to the curve between $t$ and $T_b$.}
  \label{picLT}
\end{figure}

We will also use the local time at the origin of the process $O-W$, which is a multiple of the local time at the origin of the Bessel process $X$.  More precisely, we define this local time as
\[ \ell_t := \lim_{\eps \to 0} \eps^{8/\kappa -1} N_t^{0 \to \eps}\] 
where $N_t^{0 \to \eps}$ denotes the number of upcrossings from $0$ to $\eps$ by $O-W$ before time $t$.  {Let $\tau_\eps$ be the first time that $O-W$ hits $\eps$.  Due to our particular normalization for the definition of the local time,  the expected value of $\ell$ at time $\tau_\eps$ is exactly $\eps^{8/\kappa -1 }$ when $w=o=0$.}

The goal of this section is to prove the following fact, that can be viewed as a statement about the Bessel flow.
\begin {lemma}
 \label {main5}
{Suppose that $w=o=0$.} Then  
\[ \E[ \ell_{T_1} ] = \frac { \Gamma ( 4 / \kappa) } { \Gamma (2- 8/\kappa) \Gamma (12/\kappa - 1)} .\]
\end {lemma}

\begin {proof}
Let us define the more general function $L(a,b)$ to be the expected value of $\ell_{T_b}$ when the process is started from $w=-a$, $o=0$.  {By scaling, we have that}
\[ L (a, b) = b^{8 / \kappa -1 } L ( a/b, 1).\]
Let us define  the function $U$ on $[1, \infty)$ so that $L(a/b , 1 ) = U ( a/b + 1 )$. With this notation, our goal is therefore to determine
\[ {u_1 :=} L(0,1)=U(1).\]

When the Bessel process evolves away from $0$, the local time at zero does not change. Hence, $L(O_t- W_t, g_t(1) - O_t)$ {is} a local martingale up to the first hitting time of $0$ by $O-W$. This implies (using the standard arguments for SLE martingales) that $U$ is smooth on $(1,\infty)$ and {satisfies}
$$ 
(1-x)x U''(x) + \left( \frac 4 \kappa + \left(\frac 4 \kappa -2\right) x \right)  U' (x) +  \frac 4 \kappa  \left(\frac 8  \kappa -1 \right)  U(x)  
=0$$
{with boundary conditions
\[ U(1) = u_1 \quad\text{and}\quad \lim_{x \to \infty} U(x) = 0.\]}
In other words (see Appendix~\ref{App2}), the function $U$ is equal to 
$$ U(x) = u_1 x^{- 4 / \kappa} \frac {F (4/ \kappa, 1 , 12/\kappa; 1/x ) }{  F (4/ \kappa, 1 , 12/\kappa; 1 )}$$
(note that the ODE is exactly Equation (\ref {HDE}) for the coefficients $a=c= 4/\kappa$, $b= 1- 8 /\kappa$, so that $U$ is a multiple of 
the function $h_1$ defined in the Appendix~\ref{App2}).

Our goal in the next paragraph is to show that 
\begin {equation}
\label {Lexpansion}
 L(0,1) =  L( h, 1 ) + h^{8/\kappa -1 } + o (h^{8/\kappa -1 }) {\quad\text{as}\quad h \to 0.}
\end {equation}
{This} will then enable us to identify $u_1$.
Let us define for all positive $h$, 
 $$Y_h := \one_{\tau_h < T_1 } (g_{\tau_h} (1) - O_{\tau_h}).$$ Note that by the monotonicity properties of the Bessel flow, $Y_h \le  1$ when one starts with $o=w=0$.  Using the Markov property at $\min (T_1, \tau_h)$ and  the additivity of the local time  
we see that 
$$ 
 L( 0, 1) =\E[ \ell_{T_1} ] =   \E[ \one_{T_1 < \tau_h} \ell_{T_1} ] + \E[ \one_{ \tau_h < T_1 } \ell_{\tau_h} ] +
 \E[ L (h, Y_h) ] $$
 (note that by definition $L (h, Y_h) = \one_{\tau_h < T_1 } L (h, Y_h)$).
 Using the scaling property, the fact that the probability that $T_{1/h} < \tau_1$ tends to $0$ as $h \to 0$, and the fact that $\E[ \ell_{\tau_h} ] 
 = h^{8/\kappa -1 }$ (this is where we use our actual normalization in the definition of the local time),
 we get that 
 \begin {eqnarray*}
 \lefteqn {  |\E[ \one_{T_1 < \tau_h} \ell_{T_1} ] + \E[ \one_{ \tau_h < T_1 } \ell_{\tau_h} ]  - h^{8/\kappa -1 }| }
\\ && =  \E[ \one_{ T_1 < \tau_h } ( \ell_{\tau_h } -  \ell_{T_1})  ] 
 \le  \E[ \one_{ T_1 < \tau_h }  \ell_{\tau_h }   ]  =   h^{8/\kappa -1 } \E[ \one_{ T_{1 / h } < \tau_1} \ell_{\tau_1} ] 
 = o ( h^{8/\kappa -1 } ).
 \end {eqnarray*}
 It therefore remains to estimate $ \E[ L (h, 1 ) - L (h, Y_h) ] $. Note that once we condition on $Y_h =y$, 
 we can use the same flow to couple the realizations that lead to $L (h, 1)$ and $L ( h, y)$ (defined as expected values of local times).
 {The difference between the two quantities will therefore be due 
 to the configurations in this coupling where the $\SLE_\kappa (\kappa -6)$ hits the interval $[y, 1]$ for which one then counts the local time accumulated after that time but before $T_1$. 
 This is an event that has probability bounded by a constant times $(1-y)^{\beta }$ for some $\beta {>0}$.  We remark that it is in fact known that 
 $\beta = (\kappa -4)^2 / (2 \kappa) >0 $, see \cite[Theorem~1.8]{MillerWu}.  For what follows, we will only use that $\beta > 0$ and in fact not need such a precise bound.}  Hence, for some constant $C >0$ and all $h >0$ and $y < 1$, {we have that}
\[ | L (h, 1 ) - L (h, y ) | \le C (1-y)^{\beta + 8/\kappa -1} \]
and therefore
\begin {eqnarray*}
\lefteqn { | \E[ L ( h, {Y_h} ) ] - L(h, 1 ) | }\\
&&  \le C \E[ (1-Y_h)^{\beta+ 8 / \kappa -1  } ] 
 \le  C'  ( \E[ | O_{\tau_h} |^{\beta + 8/\kappa - 1} ] + \E[\one_{\tau_h < T_1}  |1-g_{\tau_h} (1) |^{\beta+ 8 / \kappa -1  } ])
 \end {eqnarray*}
{where $C' > 0$ is a constant.}  By scaling, we note that 
\[ \E[ | O_{\tau_h} |^{\beta + 8/\kappa - 1} ] = h^{8/\kappa -1  + \beta} \E[ | O_{\tau_1} |^{\beta + 8/\kappa - 1} ].\]
 On the other hand, it is also easy to see that 
$ \E[\one_{\tau_h < T_1}  |1-g_{\tau_h} (1) |^{\beta+ 8 / \kappa -1  } ])$ decays to $0$ faster than $h^{8/\kappa -1 + \beta}$ as well: Typically, $\tau_h$ will be of order $h^2$ (because of scaling), so that $g_{\tau_h} (1) - 1$ will {be} of order $h^2$ as well. Simple estimates about the probability that $\tau_h$ is exceptionally large or $W$ fluctuates exceptionally on a small time interval allows {us} to conclude.  Putting the pieces together, we then get indeed (\ref {Lexpansion}). 

This is now enough to pin down the exact value of $u_1$. Indeed (\ref {Lexpansion}) {implies} that when $h \to 0$,
\[ U(1) - U(1+h) \sim h^{8/\kappa -1}.\]
Since we know that in the right neighborhood of $1$, $U$ has to be a linear combination of 
$F ( 4 / \kappa, 1 - 8/ \kappa, 2 - 8/\kappa ; 1-x)$   and of $ (x-1)^{8 /\kappa} F (0, 12/\kappa -1 , 8/ \kappa ;  1-x)$, by looking at 
the expansion near $1$, we conclude that 
\[ U(x)= u_1 F ( 4 / \kappa, 1 - 8/ \kappa, 2 - 8/\kappa ; 1-x) - (x-1)^{8 /\kappa} F (0, 12/\kappa -1 , 8/ \kappa ;  1-x).\]
On the other hand, we have seen that $U$ is also a multiple of the function 
$x^{- 4 / \kappa} F (4/ \kappa, 1 , 12/\kappa; 1/x )$, and by comparing this with the connection formula (\ref{connection2})
that relates these three hypergeometric functions, we see that $-u_1$ is the ratio of the two coefficients on the right-hand side of 
 (\ref{connection2}) for the appropriate choice of $a=4/\kappa$, $b=1- 8/ \kappa$ and $c=4/\kappa$, and we obtain that 
\[ u_1 = \frac { - \Gamma ( 8/ \kappa -1) \Gamma ( 4 / \kappa)}{\Gamma (8/\kappa) \Gamma (12/\kappa -1 ) \Gamma (1- 8/ \kappa)}\]
which proves the claim. 
\end {proof}

This lemma will be used in the proof of Lemma~\ref{main1} via the following corollary: 
\begin {corollary}
\label {co}
Assume that $w=b=0$. Then, as $y \to 0$, 
$$ \P[ \tau_{y^{3/4}}  < T_{y}  ] \sim (y^{1/4})^{8/ \kappa -1 } \times \frac { \Gamma ( 4 / \kappa) } { \Gamma (2- 8/\kappa) \Gamma (12/\kappa - 1)}. $$
\end {corollary}

\begin {proof}
Let us first consider the Poisson point process of excursions $(e_{l_i})$ away from the origin of the process $O-W$, indexed by the Bessel local time (with our choice of 
normalization). In other words, for each given $l$, the ordered concatenation of all the excursions $e_{l_i}$ with $l_i < l$ is the process $O-W$ up to the first time at which its local time at $0$ hits $l$. 
Clearly, for any positive $u$, the number of excursions in the sets $(e_{l_i}, l_i \in [ku, (k+1)u ))$ that reach level $y^{3/4}$ for $k \ge 0$ form an i.i.d. sequence of Poisson random variables 
with mean $u (y^{3/4})^{- (8/\kappa -1) }$. It then follows from Wald's identity, that the number 
 ${\mathcal N} (u)$ of excursions of $O-W$ that reach level $y^{3/4}$ and that have been completed before the first time $T_y (u)$ after $T_y$ at which the local time at $0$ of $O-W$ is a multiple 
 of $u$ is equal to  
$$ \E[ {\mathcal N} (u)  ] =  u (y^{3/4})^{- (8/\kappa -1) } \times \E[ \ell_{T_y (u)} / u  ] = (y^{3/4})^{- (8/\kappa -1) } \times \E[ \ell_{T_y (u)}  ]  .$$
Letting $u \to 0$, we deduce  by monotone convergence the corresponding Wald's identity for the Poisson process: 
The expectation of the  number ${\mathcal N}$ of excursions of $O-W$ that reach level $y^{3/4}$ and that have been completed before time $T_y$
is $$ \E[ {\mathcal N} ] = (y^{3/4})^{- (8/\kappa -1) } \times \E[ \ell_{T_y } ] .$$
By scaling, this quantity is also equal to 
$$ (y^{1/4})^{ 8/\kappa -1 } \times \E[ \ell_{T_1} ].$$ 

It is also easy to see, using similar arguments,  that $\P[ {\mathcal N} \ge n ] \le \P[{\mathcal N} \ge 1 ]^n$, 
 so that in fact, 
 $$ \E[ {\mathcal N} \one_{{\mathcal N} \ge 2} ] =  \P[{\mathcal N}\ge 2] + \sum_{ n \ge 2} \P[  {\mathcal N} \ge n ] \le 4\P[{\mathcal N} \ge 1]^2 $$
 for all small enough $y$. It follows that as $y \to 0$,  
$$ \P[ \tau_{y^{3/4}}  < T_y ] = \P[ {\mathcal N} \ge 1 ] \sim \P[ {\mathcal N} = 1 ] \sim \E[ {\mathcal N }] \sim 
\frac {\Gamma (4/\kappa)}{\Gamma (2 - 8 / \kappa) \Gamma (12/ \kappa -1 )} \times  (y^{1/4})^{8/\kappa -1 }.$$   
\end {proof}

Let us now explain how to deduce Lemma~\ref{main1} from Corollary~\ref{co}:

\begin {proof}[Proof of Lemma \ref {main1}]
Instead of working with the $\CLE_\kappa$ in $\HH$ with wired boundary on $\R_-$ and 
using the additional marked points at $1-\eps$ and $1$, {we will instead} work in the unit disk $\D$ and choose the four points 
$a_\eps$, $\overline a_\eps$, $-a_\eps$ and $- \overline a_\eps$ on the unit circle, where $a_\eps$ is chosen very close to $1$ so that the cross-ratio corresponding to 
in the unit disk these four points is exactly $\eps$. Note that as $\eps \to 0$, $|a_\eps -1|$ is of the order of $\sqrt {\eps}$. 
These four points define two small boundary arcs $\partial_\eps$ and $-\partial_\eps$, respectively near $1$ and $-1$. 

By conformal invariance, the event $B(\eps)$ becomes the event $B' (\eps)$ that, if one looks at the CLE with wired boundary condition on $-\partial_\eps$ and 
explores the loops (of this wired CLE) attached to $\partial_\eps$ in their order of appearance starting from $\overline a_\eps$, one finds a time at which the cross-ratio 
between $( - a_\eps, - \overline a_\eps , w_t, o_t)$ in the domain $D_t$ reaches ${\eps}^{7/8}$.
 {Note already that this will occur for loops attached to $\partial_\eps$ that have diameter of at least 
$O(\eps^{3/8})$, because the arc $- \partial_\eps$ has a length of the order of $\eps^{1/2}$. The goal of the next few lines is to estimate the probability of $B(\eps)$ 
with accuracy.}  

In order to apply our previous estimates for non-conditioned CLE's, we first sample the $\SLE_\kappa$ $\gamma$ that joins the end-points of $-\partial_\eps$. By the same $8/\kappa-1 $
boundary exponent for SLE, we know that the probability that this SLE has diameter greater than $\eps^{1/4}$ is bounded by a constant times $(\eps^{1/4})^{8/\kappa -1 }$ 
 {(recall that the endpoints of this SLE are at distance $\eps^{1/2}$ from each other, so this event corresponds roughly to the event that an SLE from $0$ to $1$ in the upper half-plane reaches distance $\eps^{1/4} / \eps^{1/2} = \eps^{-1/4}$).} 
On the event that the diameter of $\gamma$ is smaller than $\eps^{1/4}$ (which has probability very close to $1$ by the previous estimate),
we  now look  at the CLE in the complement of this small SLE: We can first map the connected component 
of the complement of this curve which has $\partial_\eps$ on its boundary back to the unit disk in such a way that $a_\eps$, $\overline a_\eps$ are fixed 
and (say) the two extremal points on $\gamma \cap \partial \D$ are mapped onto symmetric points on the real axis -- this defines a conformal map $\varphi$ that (by standard distortion 
estimates) is uniformly  very close to the identity map (the derivative of this map is uniformly close to $1$ on the right-half of the unit disk) in the neighborhood of $1$.  

We can also now discover the CLE in this disk, by using the $\SLE_\kappa (\kappa -6)$ exploration in the upper half-plane (which defines a process $W$ and $O$) and mapping it 
back onto the disk via the conformal map from $\HH$ onto $\D$ that maps $\infty$ to $-1$ and is normalized in the neighborhood of $\infty$. 
Then, distortion estimates {for conformal maps} show also that the cross-ratio corresponding to the four points $(-a_\eps, - \overline a_\eps, w_t, o_t)$ 
in the domain $D_t$ is very close to $\sqrt{\eps}$ times $W_t- O_t$ (i.e., the ratio between the two is uniformly close to $1$, as long as the $\SLE_\kappa (\kappa -6)$ stays in the 
right-hand half of the unit disk). 

Note finally that if the $\SLE_\kappa (\kappa -6)$ starting from $\overline a_\eps$ up to the swallowing time of $a_\eps$ does not stay in the right-hand side of the unit disc, then there 
exists a CLE loop in the unit disk of diameter at least $1/4$ that intersects the small arc (of size of order $\sqrt {\eps}$) between $\overline a_\eps$ and $a_\eps$. 
We will show in Lemma \ref {lastlemma} (in a conformally equivalent setting) that this quantity is bounded by a constant times $(\sqrt {\eps})^{8/\kappa -1 }$ which is an $o((\eps^{1/4})^{ 8/\kappa -1})$.

\begin{figure}[ht!]
  \includegraphics[height=6cm]{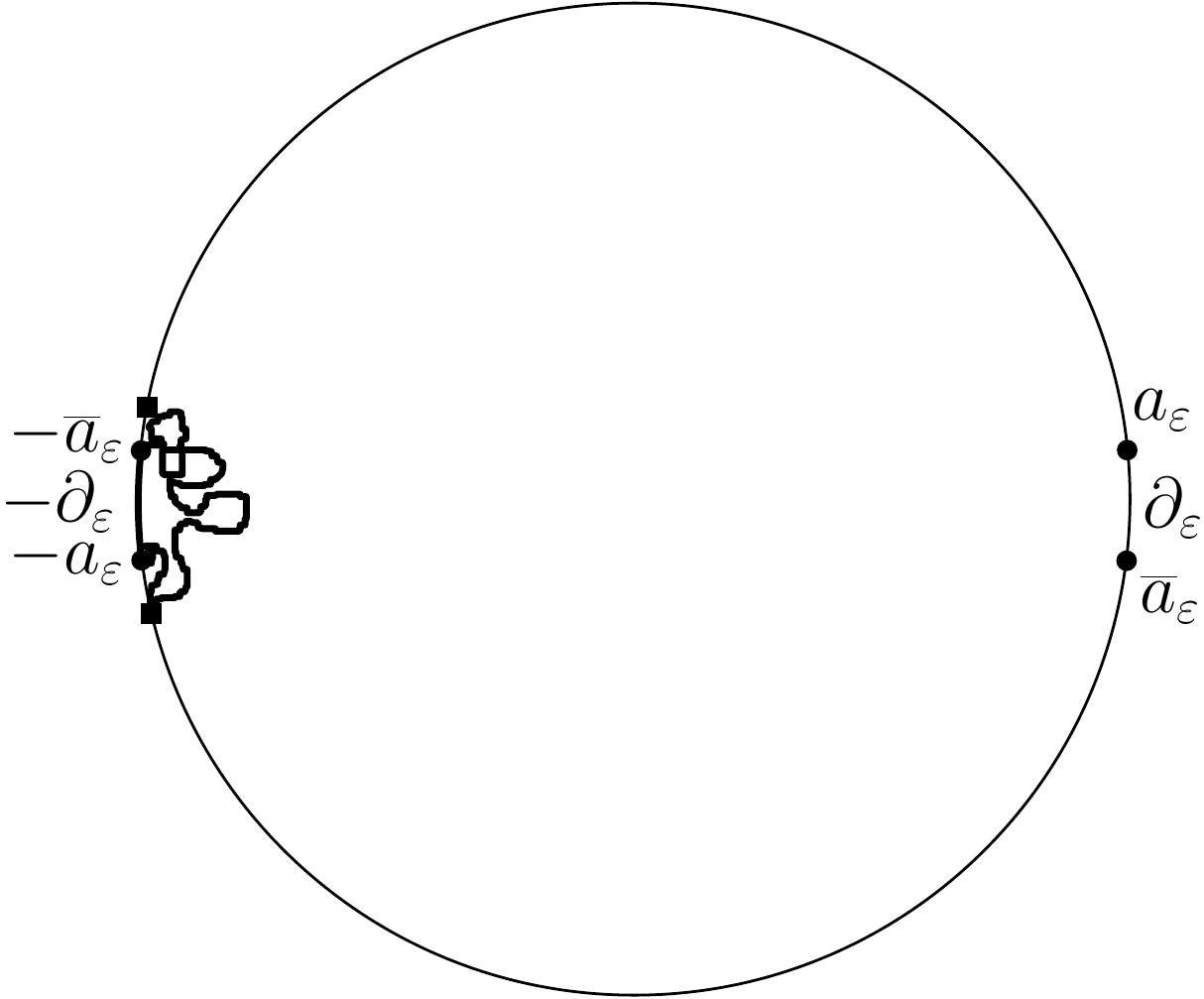}
  \quad
  \includegraphics[height=6cm]{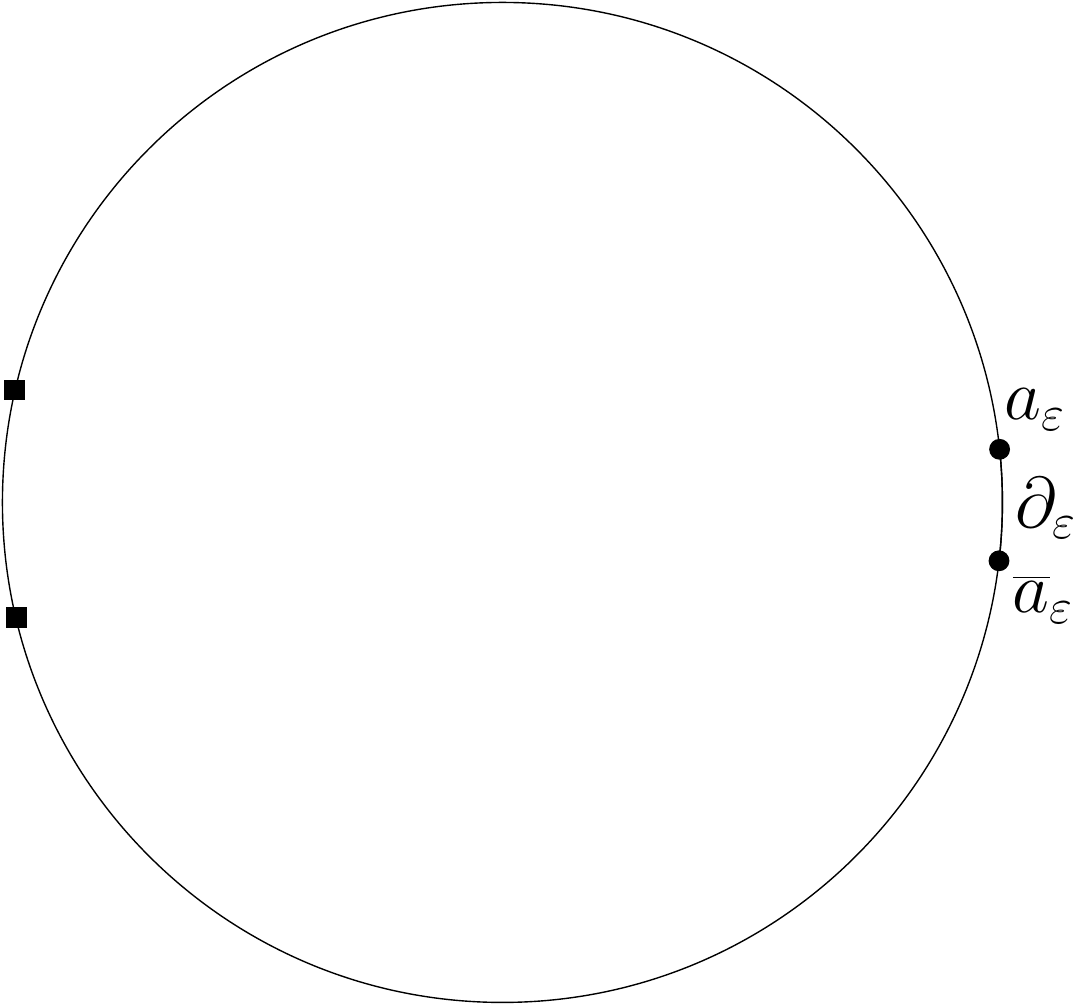}
  \caption{The SLE from $-a_\eps$ to $-\overline a_\eps$ in the wired $\CLE_\kappa$. The uniformizing map onto the unit disk is very close to the identity in the right-hand side of the disk, and therefore also near $\partial_\eps$.}
  \label{pic63}
\end{figure}

Wrapping things up, we get that the probability of $B(\eps)$ is up to an error of order $O(\eps^{1/4})^{8/\kappa -1 })$ asymptotic to the probability that $O-W$ hits $\eps^{3/8}$ before swallowing $\partial_\eps$ i.e., to 
$\P[ \tau_{y^{3/4}}  < T_{y}  ]$ in Corollary~\ref{co} for $y = \sqrt {\eps}$. Hence, we conclude that indeed, 
$$ 
\P[ B (\eps) ] \sim 
(\eps^{1/8})^{8/ \kappa -1 } \times \frac { \Gamma ( 4 / \kappa) } { \Gamma (2- 8/\kappa) \Gamma (12/\kappa - 1)}, $$
which concludes the proof of Lemma~\ref{main1}.
\end {proof}

It now finally remains to prove the following fact: 
\begin {lemma}
\label {lastlemma}
Consider a $\CLE_\kappa$ in the upper half-plane with $\kappa \in (4,8)$. Then there exists a constant $C= C(\kappa)$ such that the probability 
of the event $E_R$ that there exists a $\CLE_\kappa$ loop that intersects both the circle of radius $R$ and the interval $[-1,1]$ is bounded by $C  \times R^{-(8/\kappa -1 )}$ for all $R > 2$. 
\end {lemma}

\begin {proof}
Here are one possible way to deduce this result from our previous estimates: 

Let us first discover the $\CLE_\kappa$ loops that intersect $[-1,1]$ from left to right (and tracing the loops in a clockwise manner), starting from $-1$ via an $\SLE_\kappa (\kappa -6)$ started at $-1$.  Then, the event $E_R$ holds if and only if this $\SLE_\kappa (\kappa -6)$ will hit the circle of radius $R$ for the first time (at some point $R 
\exp ( {i \theta_-}) $) before it disconnects $1$ from infinity. 

One can also use the symmetric procedure with respect to the imaginary axis: One can start from $1$ and discover the loops from right to left, and trace them 
in a counterclockwise manner. The event $E_R$ holds if and only if this process will 
hit the circle of radius $R$ (at some point $R  \exp ({i \theta_+})$) before disconnecting $-1$ from infinity. 

Clearly, when $E_R$ holds, $0< \theta_+ < \theta_- < \pi$ so that at least one of the two events $\{ \theta_+ \le  \pi /2 \}$ and $ \{ \theta_- \ge \pi /2 \}$ holds. By symmetry, these two events have the same probability.  But if $\theta_- \ge \pi /2$, then it is easy to see that for some constant $c$ independent of $R$, the value of $O_T - W_T$ corresponding to that hitting time $T$ of the circle of radius $R$ is greater than some constant $c$ times $R$: One can for instance bound from below (by $c_1 R /y$) the probability that a Brownian motion started from $iy$ for some very large $y$, hits the circle of radius $R$ on the side with positive real value, then stays in the positive quadrant until its imaginary value hits $R/2$, and then exits the upper half-plane on $(-R, -R/2)$ without exiting the disk of radius $R$ around the origin -- and note that when this event happens, then necessarily the Brownian motion (obtained by conformal image of the previous one) 
started from $iy + o(y)$ exits the upper half-plane in the interval $[W_T, O_T]$ (and this event has probability of order $c_2(O_T-W_T)/y$ as $y \to \infty$). 

Hence, we conclude that the probability of $E_R$ is bounded by twice the probability that for the $\SLE_\kappa (\kappa -6)$ started from $-1$, $O-W$ reaches $cR$ before swallowing $1$, which is the quantity we have  derived the asymptotic behavior of  in Corollary \ref {co} (modulo scaling)  
\end {proof}

\section {Proof of Lemma~\ref{main3}}
\label {Sproof2}
In the present section, we will prove Lemma~\ref{main3}. This will complete the proof of Theorem~\ref{mainprop} for the remaining case $\kappa \in (8/3, 4)$.

Let us consider $\gamma$ an $\SLE_\kappa$ from $0$ to infinity in the upper half-plane,  driven by 
$W_t = \sqrt {\kappa} {\beta_t}$, and let $g_t$ be the uniformizing {conformal} map from $\HH \setminus \gamma [0,t]$ onto $\HH$ so that $g_t(z) = z + o(1)$ as  $z \to \infty$. Recall that  $\partial_t g_z (z) = 2 / ( g_t (z) - W_t )$.

We denote by ${\mathcal R}$ the independent one-sided restriction sample attached to $[1- \eps, 1]$. Our goal is to estimate the probability of the event $C (\eps)$ that 
$\gamma$ intersects ${\mathcal R}$. Let us define the function $Q$ on $(0,1)$ by 
$$ 
Q ( 1 - \eps ) := 1- \P[ C (\eps) ].$$
Then, 
$$
Q( 1 - \eps)  = \P[ \gamma \cap {\mathcal R} = \emptyset ] = \lim_{t \to \infty} \P[ \gamma [0,t] \cap {\mathcal R} = \emptyset ] = \lim_{t \to \infty}
\E \left[ \left( \frac { \eps^2  g_t' (1-\eps) g_t' (1)} {(g_t (1) - g_t (1- \eps))^2} \right)^\alpha \right].$$
If we write $I_t = g_t (1)$ and $V_t = g_t (1 - \eps)$, then we note that 
$$ 
Q (( V_t - W_t ) / (I_t - W_t)) \times \left(  \frac {\eps^2  g_t' (1-\eps) g_t' (1)} {(g_t (1) - g_t (1- \eps))^2} \right)^\alpha 
$$ 
is a bounded martingale, and we deduce, using the standard machinery that $Q$ is smooth and is a solution to the ODE 
$$ \frac {\kappa}{2} x^2 (x-1)  Q'' (x) + ((\kappa -2 ) x - 2) x Q' (x) - 2 \alpha (x-1) Q (x) = 0 
$$ 
on $(0,1)$ {with} boundary conditions $Q(0) = 0$ and $Q(1)=1$. In fact, if we write 
$Q(x) = x^{2/ \kappa} A(x)$, then $A$ solves the following hypergeometric differential equation
\begin {equation} 
 \label {ODE2}
 x(1-x) A'' + ( -2x + 8/ \kappa) A' + (16/\kappa^2 - 4/\kappa ) A = 0 .
 \end {equation}
This means (see Appendix~\ref{App2}) in particular that $A$ is a linear combination of the same function $f$ as in Section~\ref{S2}, i.e., 
$$f(x):= F ( 4/ \kappa, 1- 4/\kappa, 8/\kappa ; x)$$ and of another function that {diverges} like 
$x^{1-8/\kappa}$ as $x \to 0^+$. By the boundary conditions for $Q$, we conclude that $A$ is a multiple of $f$ and more precisely that 
$ A(x) = f(x) / f(1) $, so that 
$$ Q (x) = x^{2/ \kappa} f(x) / f(1).$$

Recall that our goal is to estimate $\P[ C (1-x) ] =  1 - Q (x)$ as $x \to 1$. For this purpose, we express (via the connection formula (\ref {connection1}))
the hypergeometric function $f$ as a linear combination of the two natural 
independent hypergeometric functions that solve the same ODE in the neighborhood of $1$, i.e., we write $f$ as a linear combination 
of $$f_1 (x) :=  F( 4/ \kappa, 1- 4 / \kappa , 2 - 8 / \kappa ; 1-x )$$ 
and of 
$$ f_2 (x) := (1-x)^{8/\kappa - 1} F( 4/\kappa, 12/\kappa -1 , 8/\kappa ; 1-x )$$
and we get
$$ f(x) = f(1) f_1 (x) - \eta  f_2 (x),$$
where $$\eta: = - \frac { \Gamma (8/\kappa) \Gamma (1- 8 / \kappa)}{\Gamma (4/\kappa) \Gamma (1-  4/\kappa)} = \frac 1 {- 2 \cos (4\pi / \kappa)}. $$
Hence, we see that as $x = 1 - \eps \to 1$, 
$$ f(1-\eps) = f(1) - \eps f'(1)  - \eta \eps^{8/\kappa -1} + O(\eps^2)$$
(recall that $1 \le  8/\kappa - 1 < 2$ because $ \kappa \in (8/3, 4]$). We can note that $f'(1) = - 2 f(1)/\kappa$ (which follows for example from (\ref {ODE2})), so that
$$ \frac {f(1-\eps)}{f(1)} = 1 + \frac {2}{\kappa} \eps - \frac {\eta}{f(1)} \eps^{8/\kappa -1 } + O (\eps^2).$$
If we now expand $Q(1-\eps) = (1-\eps)^{2 / \kappa} f(1-\eps) / f(1)$ when $\eps \to 0$, 
we get that 
\[ \P[ C ( \eps) ] =  1 - Q(1-\eps)  \sim      \frac {\eta}{f(1)} \eps^{8/\kappa -1 }, \]
 which concludes the proof.

\appendix
\section{Connection probabilities for discrete O($N$) models}
\label {App1}
Let us very quickly browse through the properties about hook-up probabilities in squares of discrete FK($q$) percolation models and of O($N$) models that we have been referring to in this paper.
All these facts are elementary and classical (the reader can consult for instance \cite{DuminilParafermionic} and the references therein).

We will first describe the example of the fully packed version of the O($N$) model on the square lattice 
(which is in fact directly related to the critical FK model on the square lattice for $q = N^2$).
This fully-packed O($N$) model is the model where each small square in the domain is filled with one 
\begin{figure}[ht!]
  \includegraphics[scale=0.25]{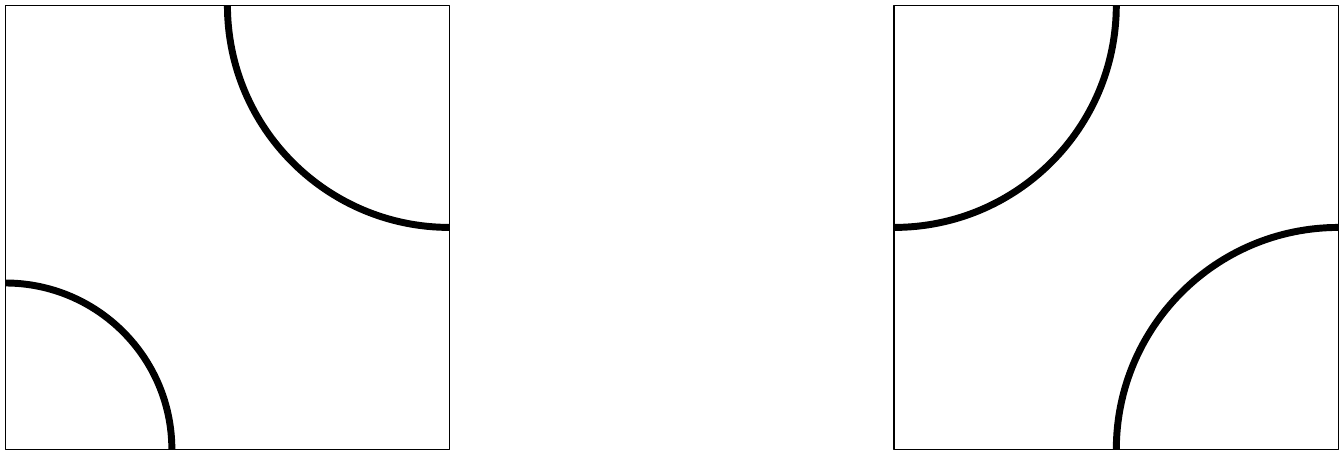}
  \caption{The two tiles for the fully packed O($N$) model}
  \label{pictwooptions}
\end{figure}
of the two possible options depicted in Figure~\ref{pictwooptions}. 

Then, when one sets boundary conditions as in the middle of Figure~\ref{picON}, one gets a collection of loops as in the right of Figure~\ref{picON}. 
In the fully-packed  O($N$) model, the probability of a configuration is chosen to be proportional to $N^L$ where $L$ is the number of loops in the configuration. 
\begin{figure}[ht!]
  \includegraphics[scale=0.45]{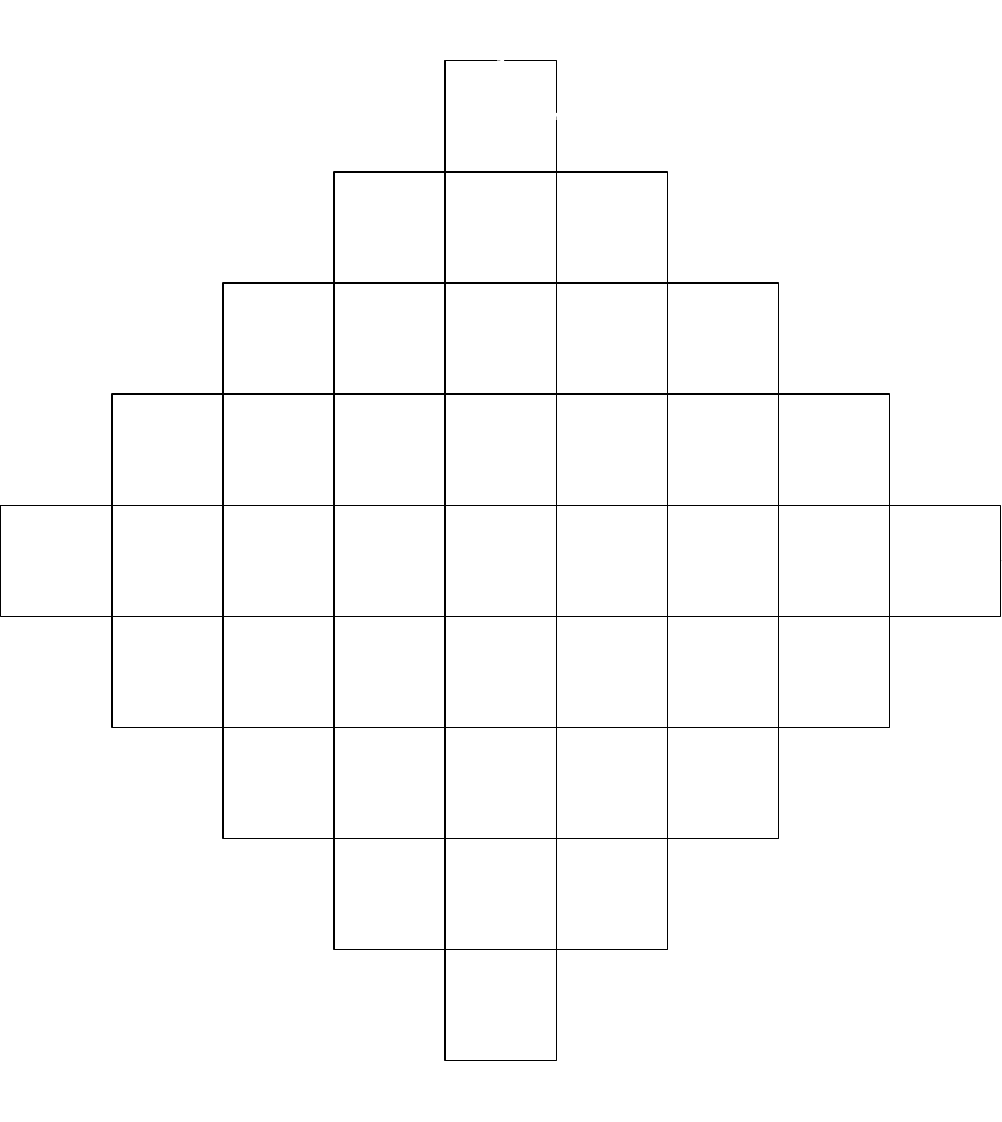}
   \includegraphics[scale=0.45]{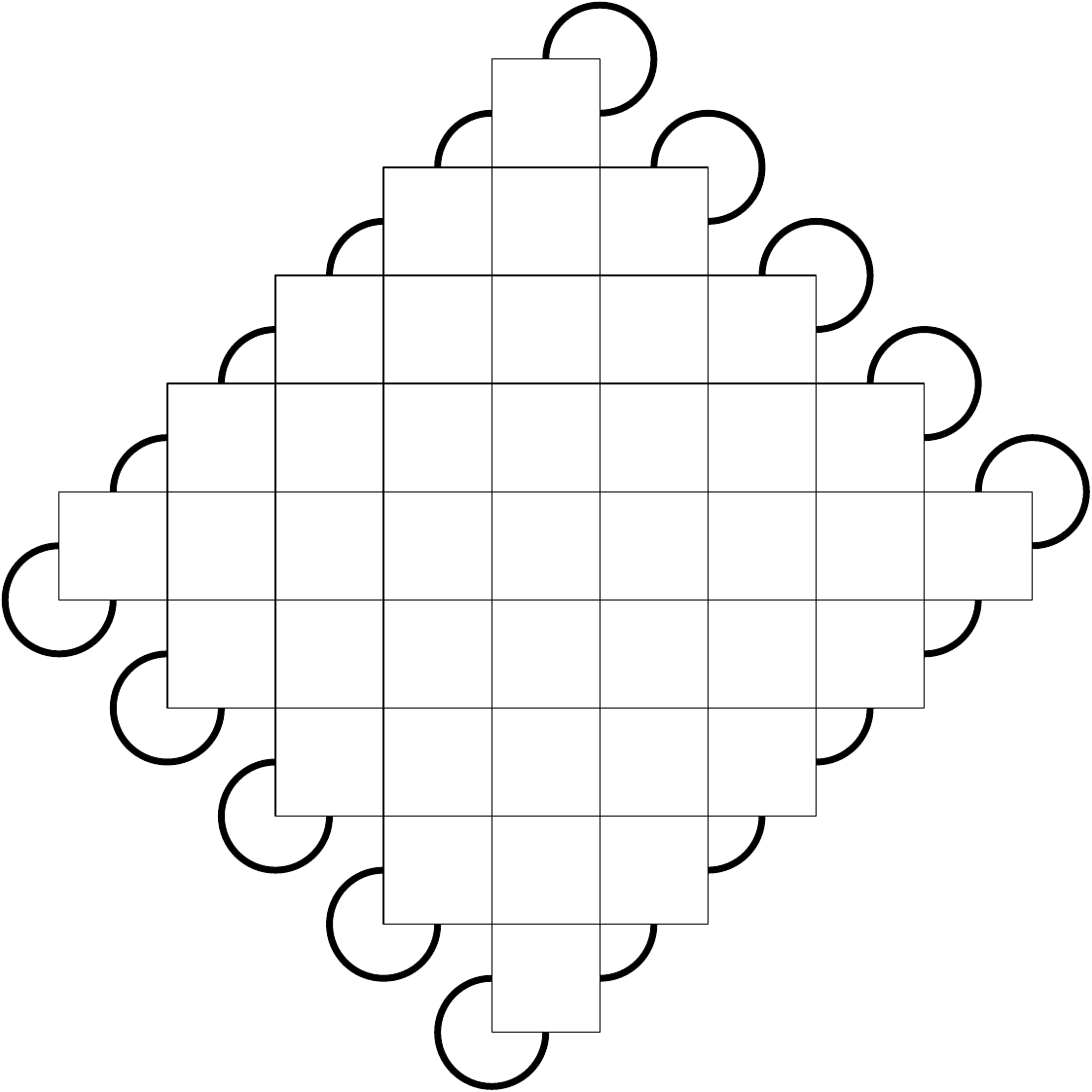}
   \includegraphics[scale=0.45]{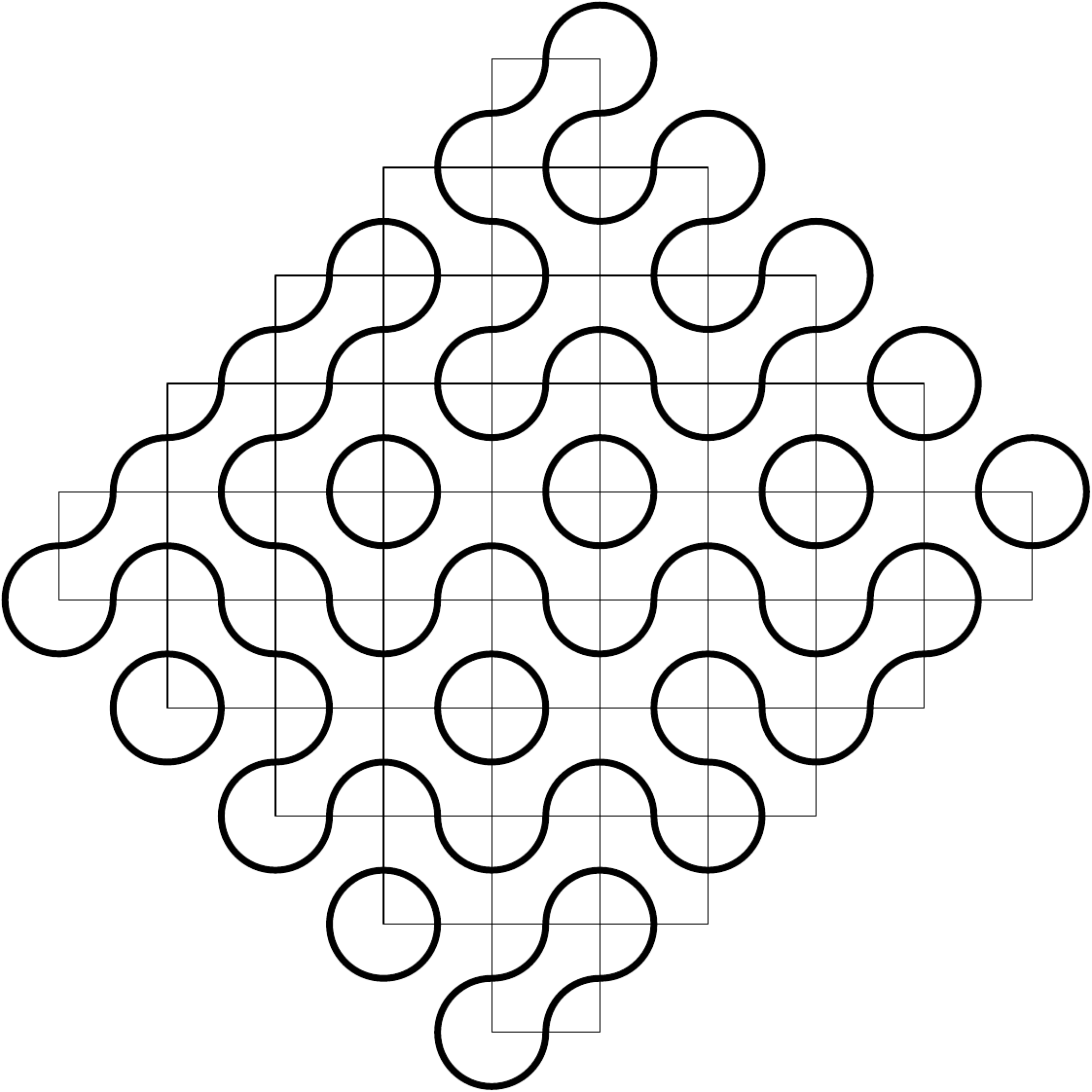}
  \caption{The fully-packed O($N$) model (with free boundary conditions)}
  \label{picON}
\end{figure}
When one explores the tiles of the O($N$) model starting from two corners, exploring the loops a ``discrete Markovian way''  {(loosely speaking: one picks a strand emanating from the boundary and explores this strand and the tiles it traverses, until one completes the loop, and then one picks a new strand from which to  explore without using any information of the not-yet-revealed tiles),} one ends up with a configuration {as} in Figure~\ref{picONdiscovery}. The conditional 
distribution of the remaining-to-be discovered configuration is now the discrete analog of our CLE with two wired boundary conditions. 
\begin{figure}[ht!]
  \includegraphics[scale=0.45]{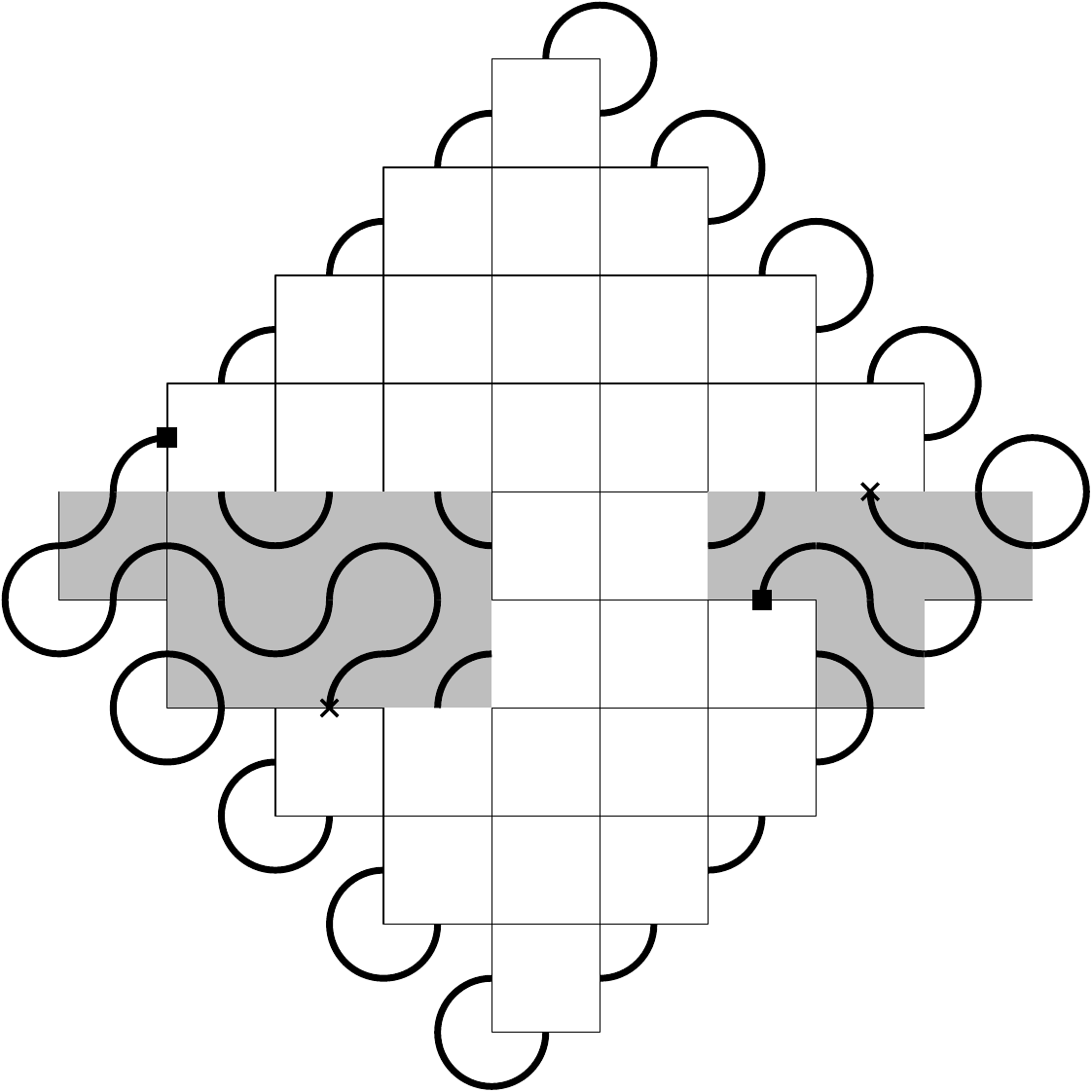}
  \caption{Exploring the fully packed O($N$) model leads to an O($N$) model with two (longer) wired arcs.}
  \label{picONdiscovery}
\end{figure}
Examples of {the} fully-packed O($N$) {model} with two wired boundary parts in the original square correspond to the  choice of boundary conditions depicted in the left of Figure~\ref{picONsample}, 
where the probability of a configuration is 
still proportional to the number of created loops (also taking into account the loops that go through the boundary). 

We can note that when one is given a configuration for which the two boundary arcs are joined together into a single loop, then if one rotates the configuration by 90 degrees without rotating the boundary conditions, one has a configuration with exactly one more loop.  
It therefore follows that the probability that the two boundary arcs are joined into a single loop is $N$ times smaller than the probability that 
they are part of two different loops. 
\begin{figure}[ht!]
\includegraphics[scale=0.45,angle=90]{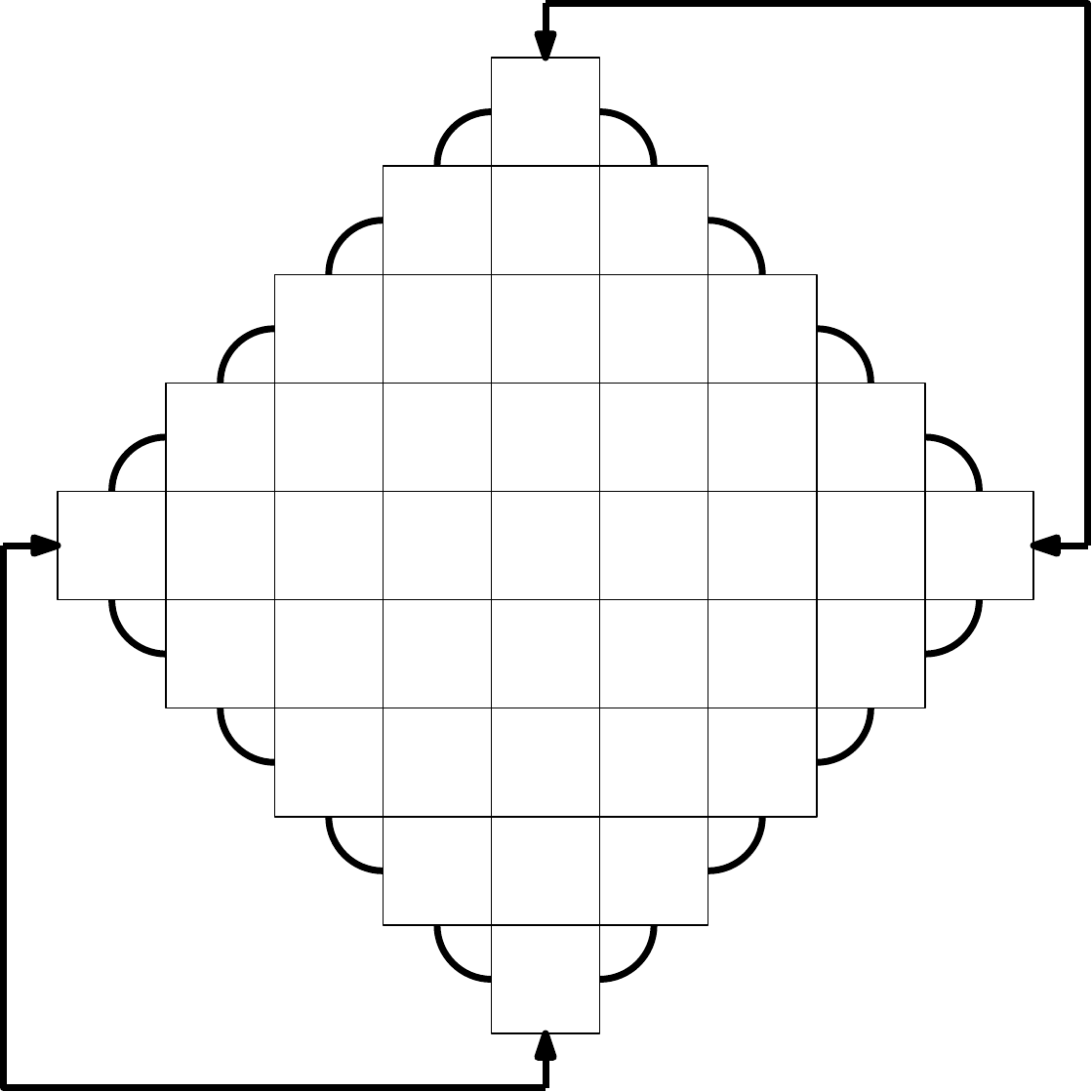}
\quad
  \includegraphics[scale=0.45]{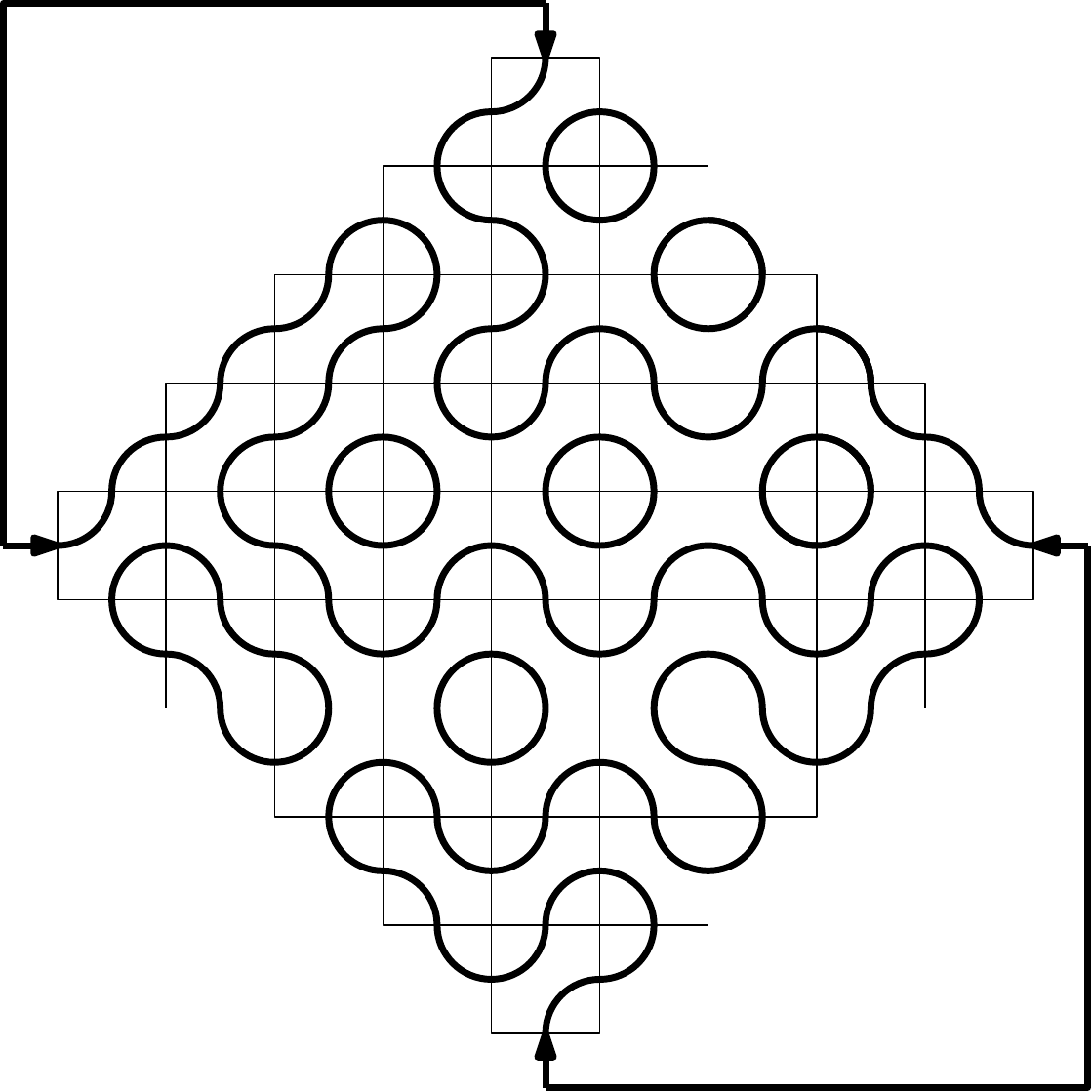}
  \quad
  \includegraphics[scale=0.45,angle=90]{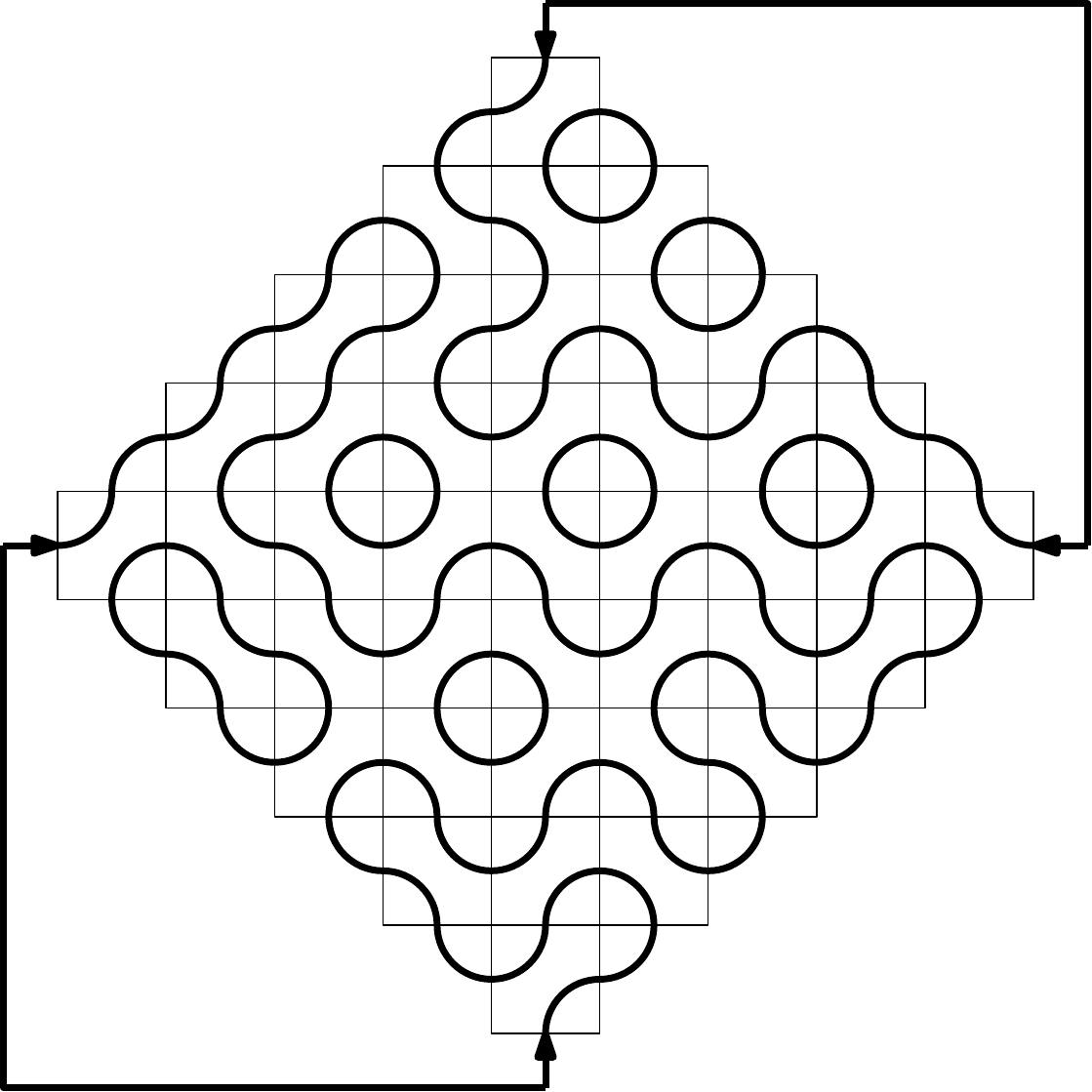}
  \caption{The boundary conditions (left). Rotating the middle configuration by 90 degrees creates exactly one additional loop. 
  The probability that the boundary arcs hook up into one single loop (as in the middle picture) is $1 / (1+N)$.}
  \label{picONsample}
\end{figure}In other words, in Figure~\ref{picONsample}, the probability 
of the event that the two wired boundary arcs are part of the same loop is $1 / ( 1+ N)$. 
\medbreak

A variation of the previous fully-packed loop model is to allow for {additional} configurations. This time, one considers the square {as} on the left-hand side of Figure~\ref{picON}, and an admissible configuration is when one fills each 
tile with one of the seven tiles depicted in Figure~\ref{seventiles}, in such a way that one only creates closed loops.
\begin{figure}[ht!]
  \includegraphics[scale=0.45,angle=-90]{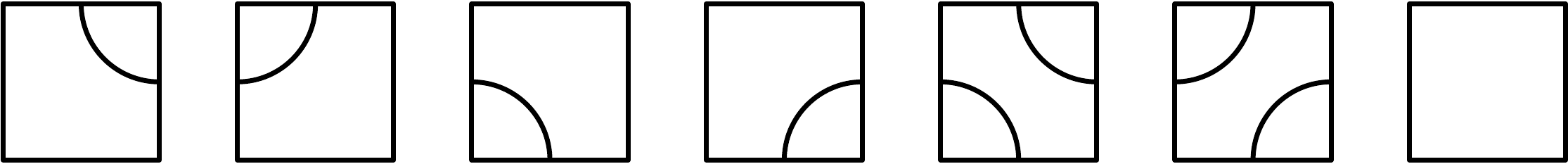}
  \caption{The seven tiles}
  \label {seventiles}
\end{figure}
One can then choose a parameter $\mu$, and weight each configuration by $N^L \mu^l$ where $l$ denotes the cumulated length of all the loops.
The previous fully-packed case corresponds to the limit 
when $\mu \to \infty$. Then exactly the same arguments lead to the definition of the corresponding model in the square with two wired boundaries, as depicted in Figure~\ref{dilute}, and to the fact that for this model the probability that the two boundary arcs are part of the same loop is also $1/ (1+N)$, regardless of $\mu$. 
\begin{figure}[ht!]
  \includegraphics[scale=0.45]{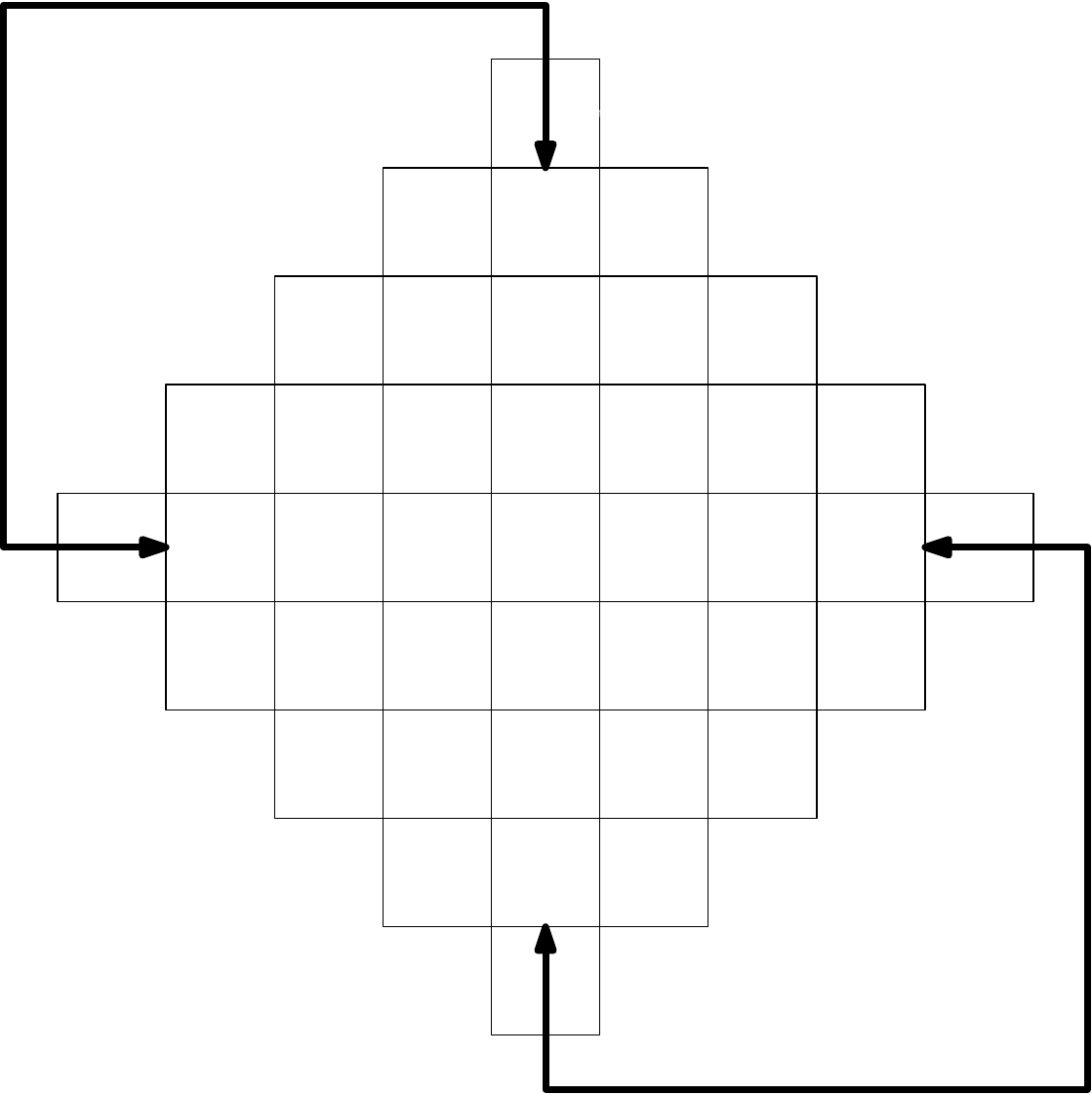}
  \quad
  \includegraphics[scale=0.45]{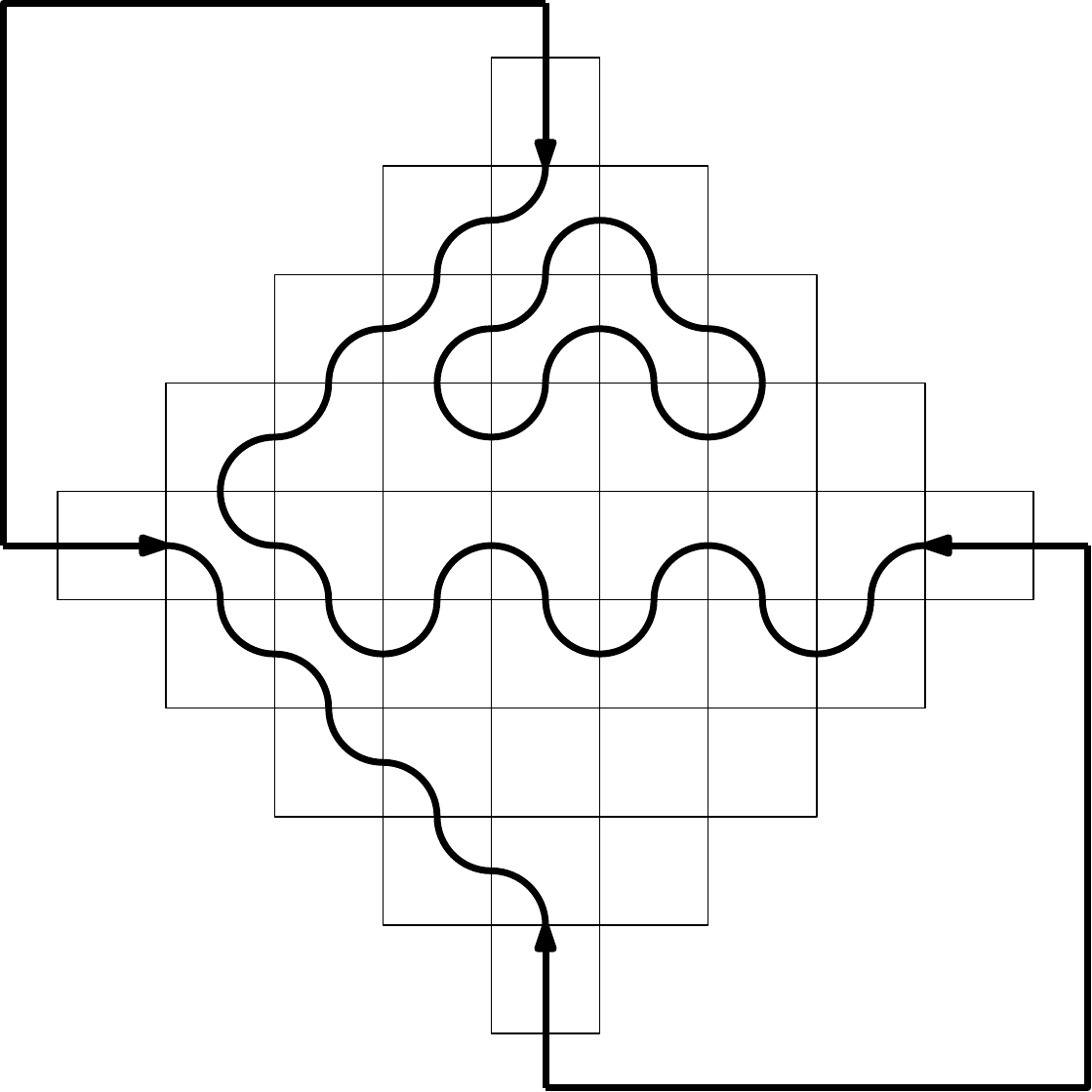}
  \caption{The boundary conditions (left). The probability that the boundary arcs hook up into one single loop is again $1 / (1+N)$.}
  \label {dilute}
\end{figure}
Note that this property of O($N$) models is actually quite robust and works also for more general models as long as the probability of a configuration is the product of local weights times $N^L$. It also 
holds on other lattices, as long as they have enough symmetries. 

The asymptotic behavior of the O($N$) model {depends on the choice of $\mu$}.  It is conjectured (see \cite{kn,Sheffield}) that when $N \le 2$, for a well-chosen critical value $\mu(N)$, it converges to a simple $\CLE_\kappa$, while for large $\mu$, it converges to a non-simple $\CLE_\kappa$. 

For the FK-percolation model on the square lattice (see for instance \cite{GrimmettRC} and the references therein for {its} definition and basic properties), the corresponding crossing property can be stated as follows. Consider $q > 0$, and the FK($q$) model on the rectangle 
$[0,n+1] \times [0, n]$, where the left-hand boundary is wired and the right-hand boundary is wired (i.e., all points on the left boundary are identified as one single point, and all points 
on the right hand boundary are identified as another point). Consider also the self-dual value of the parameter $p$, i.e., take $p = \sqrt {q} / (1 + \sqrt {q})$. Then, the probability of a left-to-right crossing of this rectangle is $1 / ( 1+ \sqrt {q} )$. 

One way to see it is via the usual duality trick because the dual configuration $w^*$ to a configuration $w$ is also a critical FK($q$) model on the dual graph, which is the rectangle $[1/2, N+1/2] \times [-1/2, N +1/2 ]$ but with the top and bottom sides identified as one single site (not two as for the left and right boundaries before). Hence, it follows exactly the same 
law as $w$ rotated by 90 degrees, except that the configurations get an extra weight $1/q$ when there is no top to bottom crossing. Hence, if $\pi$ is the probability of a left to right crossing 
for $w$, one has $   1 - \pi = \pi   / ( \pi + (1 - \pi)/q )$ 
from which the statement follows. 
Another simple way is just to note that if we rotate the picture by 45 degrees, and look at the union of the outer boundaries of the collection of clusters and of the outer boundaries of dual clusters, one gets exactly 
the previous fully-packed O($N$) model with $N = \sqrt {q}$, with boundary conditions just as described in the O($N$) case above, so one can apply directly the previous considerations on fully-packed O($N$) models. 
See also for instance Section 2 of \cite{SmirnovFK}, or \cite{DuminilParafermionic}.

\section{Hypergeometric functions} 
\label {App2}
For the convenience of those readers {who} are not so acquainted {with} the basic properties of hypergeometric functions that we are using (or to refresh their memories), we try to very briefly recall them in the following page. When $a$, $b$, $c$ are real numbers, the hypergeometric function $F( a, b, c ;z) =  \null  _2F_1 ( a, b, c ;z) $ is defined for all $z$ in the open unit disk by the power series
$$ F ( a,b,c ; z) = \sum_{n \ge 0} \frac {(a)_n (b)_n}{(c)_n n! } z^n,$$ 
where $(a)_n = a ( a+1) \ldots (a+n-1)$ is the rising Pochhammer symbol (with the convention $(a)_0 =1$). 
When $c > a+b$ (which is  in fact the case for all hypergeometric functions that we write out explicitly as a function of $\kappa$ in this paper) this series converges also at $z =1$ and the function is continuous on the interval $[0,1]$. The value at $1$ can then be expressed in terms of the $\Gamma$ function: 
$$ F (a,b,c ; 1) = \frac { \Gamma (c) \Gamma (c-a-b)} {\Gamma (c-a) \Gamma (c-b)}.$$
When $a+b > c$, then the hypergeometric function $F(a,b,c ; x)$ diverges like a constant times $(1-x)^{c-a-b}$ when $x \to 1^-$. Indeed, one can check that 
$$ F (a,b,c ; x) = (1-x)^{c-a-b} F(c-a, c-b, c ; x).$$

The hypergeometric function $F(a,b,c; x)$ is a solution of the hypergeometric differential equation
\begin {equation}
 \label {HDE}
 x (1-x) f'' + ( c - (a+b+1) x) f' - ab f = 0
\end {equation}
on the interval $(0,1)$. Conversely, it is easy to check that when $c$ is not a non-negative integer, any solution to this equation on the interval $(0,1)$ is a linear combination 
of the two functions $F(a, b , c ; x)$ and $x^{1-c} F ( 1+a-c, 1+b-c , 2-c ; x)$. 

If we use the change of variables $y=1-x$, we can note that the equation (\ref {HDE}) gets transformed into another hypergeometric equation.
It therefore follows that a solution to (\ref {HDE}) on $(0,1)$ is also necessarily a linear combination of the two functions 
\begin {eqnarray*}
f_1 (x) &:= &F( a, b, a+b+1 - c ; 1-x) \\
f_2 (x) &:= &(1-x)^{c-a-b} F ( c-a , c-b, 1 + c - a- b; 1-x).
\end {eqnarray*}
In particular, using the particular values of those functions at $0$ and $1$, one gets that on $(0,1)$,
\begin {equation}
 \label {connection1}
 F(a,b,c ; x ) = \frac {\Gamma (c) \Gamma (c-a-b)} {\Gamma (c-a) \Gamma (c-b)} f_1 (x) + \frac {\Gamma (c) \Gamma (a+b -c)}{\Gamma (a) \Gamma (b)} f_2 (x),
\end {equation}
which is one of the connection formulas between hypergeometric functions (this is for instance 15.3.6 in \cite{AS}). 

Similarly, if one looks for solutions to (\ref {HDE}) on the interval $(1, \infty)$, one can use the change of variables $y= 1/x$ and see that when $a-b$ is not an integer, such a solution is necessarily a linear combination of the two functions
\begin {eqnarray*} 
 h_1 (x) &:=&   x^{-a} F (a, 1+a -c, 1 + a -b ; 1/x ), \\
h_2 (x)&:=&  x^{-b} F (b, 1+b-c, 1+ b-a; 1/x).
\end {eqnarray*}

Again, one can see that when $x \in (1,2)$ such a solution is also a linear combination of $f_1$ and 
$${\tilde f}_2 (x) := (x-1)^{c-a-b} F ( c-a , c-b, 1 + c - a- b; 1-x) .$$
In particular, 
\begin {equation}
 \label {connection2}
 h_1 (x) =  \frac { \Gamma (a-b+1) \Gamma (c-a-b)}{\Gamma (1-b) \Gamma (c-b)} f_1 (x) + 
 \frac{\Gamma (a-b+1 ) \Gamma (a+b-c)}{\Gamma (a) \Gamma (a-c+1)} {\tilde f}_2 (x),
\end {equation}
which is the other connection formula that we use in this paper (note that it describes in particular the precise asymptotic expansion of $h_1 (x)$ in the limit when $x \to 1^+$).

\subsubsection*{Acknowledgments}
 JM thanks Institut Henri Poincar\'e for support as a holder of the Poincar\'e chair, during which this work was completed.
WW is part of the NCCR Swissmap and acknowledges the support of the SNF grants 155922 and 175505.  He also thanks the Statslab of the University of Cambridge for its hospitality on several occasions during which part of the work for this project was completed.
We would also like to thank the referees for their comments.

\end{document}